\documentclass[11pt]{article} 

\usepackage[utf8]{inputenc}

\usepackage{geometry}
\usepackage{graphicx} 
\usepackage{booktabs} 
\usepackage{array} 
\usepackage{paralist}
\usepackage{verbatim} 
\usepackage{subfig} 

\usepackage{array}
\usepackage{adjustbox}

\usepackage{sectsty}
\allsectionsfont{\sffamily\mdseries\upshape}

\usepackage{xcolor}
\usepackage{graphicx}
\usepackage{amssymb}

\newcommand{\qed}{\hfill $\square$}
\usepackage{pdfpages}

\usepackage{tikz}
\usetikzlibrary{shapes.geometric}
 \usepackage{psfrag}
\usepackage{tikz}
\usepackage{amsmath, amssymb, graphics, setspace}

\usepackage{tikz,tikz-3dplot}
\tdplotsetmaincoords{80}{45}
\tdplotsetrotatedcoords{-90}{180}{-90}

\usepackage{pgfplots}
\usepackage{tikz-3dplot} 
\pgfplotsset{compat=1.3}

\usepackage{tkz-euclide,amsmath} 

\usepackage{tikz}
\usetikzlibrary{shapes.geometric}

\usepackage{titlesec}
\renewcommand{\thesection}{\arabic{section}}
\renewcommand{\thesubsection}{\thesection.\arabic{subsection}}
\renewcommand{\thesubsubsection}{\thesubsection.\arabic{subsubsection}}

\titleformat{\section}{\normalfont\large\mdseries}{\thesection}{1em}{}
\titleformat{\subsection}{\normalfont\normalsize\mdseries\upshape}{\thesubsection}{1em}{}
\titleformat{\subsubsection}{\normalfont\normalsize\mdseries\upshape}{\thesubsubsection}{1em}{}

\usepackage{makecell}

\usepackage{epigraph}
\usepackage{titlesec}

\usepackage{graphicx}
\usepackage{multicol,multirow}
\usepackage{amsmath,amssymb,amsfonts}
\usepackage{mathrsfs}
\usepackage{amsthm}
\usepackage{rotating}
\usepackage{appendix}
\usepackage{array}

\usepackage{amsmath}

\theoremstyle{plain}
\newtheorem{theorem}{Theorem}[section]

\newtheorem{lemma}[theorem]{Lemma}
\newtheorem{corollary}[theorem]{Corollary}
\newtheorem{prop}[theorem]{Proposition}
\newtheorem*{prop*}{Proposition}
\theoremstyle{definition}
\newtheorem{definition}[theorem]{Definition}
\newtheorem*{thm*}{Theorem}
\theoremstyle{oupremark}
\newtheorem{remark}[theorem]{Remark}
\newtheorem{example}[theorem]{Example}
\theoremstyle{oupproof}

\newtheorem*{conj*}{Conjecture}

\usepackage{amsthm}



\usepackage{color}
\usepackage{graphicx}
\usepackage{color}
\usepackage{graphicx}
\usepackage{etoolbox}
\usepackage{tikz}
\usepackage{adjustbox}
\usepackage{wrapfig}
\usepackage{diagbox}
\usepackage{array}
\usepackage{multirow}
\usepackage{todonotes}

\usepackage{hyperref}
\hypersetup{
    colorlinks = true,
    linkbordercolor = {white}, 
linkcolor=black,
  citecolor=black
}

\newcommand{\hop}{\vskip .3cm\noindent} 
\newcommand{\hip}{\vskip .1cm\noindent}

\newcommand{\Isom}{\operatorname{Isom}}

\usepackage{fourier-orns}
\usepackage{fancyhdr}

\usepackage{adjustbox}

\usepackage{emptypage}

\usepackage{pifont}

\usepackage{caption}

\usepackage{multirow}
\usepackage{adjustbox}

\usepackage{amsmath}

\usepackage{dynkin-diagrams}
\usepackage[nice]{nicefrac}

\usepackage{fancyhdr}

\makeatletter
\newcommand{\subjclass}[2][1991]{%
  \let\@oldtitle\@title%
  \gdef\@title{\@oldtitle\footnotetext{#1 \emph{MSC.} #2}}%
}
\newcommand{\keywords}[1]{%
  \let\@@oldtitle\@title%
  \gdef\@title{\@@oldtitle\footnotetext{\emph{Keywords.} #1.}}%
}
\makeatother

\title{On ADEG-polyhedra in hyperbolic spaces} 
\author{Naomi Bredon}

\date{}

\begin{document}

\maketitle

\setlength{\skip\footins}{0.5cm}
\renewcommand{\thefootnote}{\fnsymbol{footnote}} 

\vspace{-1.5em}

\begin{abstract}
In this paper, we establish that the non-zero dihedral angles of hyperbolic Coxeter polyhedra of large dimensions are not arbitrarily small. Namely, for dimensions $n\geq 32$, they are of the form $\frac{\pi}{m}$ with $m\leq 6$.  
Moreover, this property holds in all dimensions $n\geq 7$ for Coxeter polyhedra with mutually intersecting facets. 
Then, we develop a constructive procedure tailored to Coxeter polyhedra with prescribed dihedral angles, from which we derive the complete classification of ADEG-polyhedra, characterized by having no pair of disjoint facets and dihedral angles $\frac{\pi}{2}, \frac{\pi}{3}$ and $\frac{\pi}{6}$, only.  Besides some well-known simplices and pyramids, there are three exceptional polyhedra, one of which is a new polyhedron $P_{\star}\subset \mathbb H^9$ with $14$ facets.
\end{abstract}

\section*{Introduction}

Let $\mathbb H^n$ be the real hyperbolic space of dimension $n\ge2$. A hyperbolic {\it Coxeter polyhedron} is a convex polyhedron $P\subset \mathbb H^n$ (of finite volume) all of whose dihedral angles are all  integer submultiples of $\pi$. 

\hip
The fundamental works of Poincaré \cite{Poin} and Andreev \cite{Andreev} have provided a comprehensive understanding of hyperbolic Coxeter polyhedra in dimensions $2$ and $3$. However, for dimensions beyond $3$, the classification of hyperbolic Coxeter polyhedra remains far from being complete.  
Nonetheless, several families of Coxeter polyhedra have been classified. 
For instance, all Coxeter polyhedra with a small number of facets are known, including all simplices, 
which exist only up to dimension  $4$ in the compact case, and up to dimensions $9$ in the non-compact case; 
see the works of Lanner  \cite{Lanner} and Chein \cite{Chein}. 
Additionally, hyperbolic Coxeter polyhedra with exactly  $n+2$ facets have been classified through the works of Esselmann \cite{Ess}, Kaplinskaya \cite{Kaplinskaya} and Tumarkin \cite{Tum1},  as well as  all {\it compact} Coxeter polyhedra with $n+3$ facets  \cite{Tum3}; see the survey \cite{Fweb} for further results.  
More broadly, all Coxeter cubes have been classified by Jacquemet and Tschantz and cease to exist in dimension $6$ (see \cite{JT}), and all Coxeter pyramids have been classified through the works of Tumarkin  \cite{Tum1,Tum2} and Mcleod \cite{Mcleod}, and they exist up to dimension $17$.

\hip 
In this work, we encounter Coxeter polyhedra of various combinatorial structures and with an arbitrarily large number of facets, but with prescribed dihedral angles. 
This restriction is motivated by the fact that some families of Coxeter polyhedra do not have arbitrarily small dihedral angles,  as stated in the following proposition. 

\begin{prop*}
Let $P\subset\mathbb H^n$ be a finite-volume Coxeter polyhedron. 
Then, the following holds for all the non-zero dihedral angles $\frac{\pi}{m}$ of $P$.  
\begin{enumerate} 
\item  If $n\geq 32$, then $m\leq 6$.  
\item If $n\geq 7$ and $P$ has mutually intersecting facets, then $m\leq 6$.   
\item If $n\geq 4$ and $P$  is ideal, then $m\leq 5$.  
\end{enumerate}
\end{prop*}

\hip 

According to Felikson and Tumarkin's work \cite{FT1},  Coxeter polyhedra with mutually intersecting facets are well understood in the {\it simple} case, that is, when each of their vertices is the intersection of precisely $n$ hyperplanes. Namely, they establish that they are either simplices or products of simplices. 
The understanding of the simple case combined with the above proposition is a great motivation to investigate for a method to complete the classification of Coxeter polyhedra with mutually intersecting facets.

\hip 
Inspired by the works of Allcock \cite{All} and Prokhorov \cite{Prk}, we establish a  constructive procedure tailored for polyhedra with prescribed dihedral angles, which enables to obtain polyhedra of any combinatorial type that differs from a simplex. 
We illustrate this method through the 
classification of {\it ADEG-polyhedra}, that is, Coxeter polyhedra with mutually intersecting facets and dihedral angles  $\frac{\pi}{2}, \frac{\pi}{3}$ and  $\frac{\pi}{6}$.  Our main result is stated below.

 \begin{thm*}
{\it Let $P\subset \mathbb H^n$ be an ADEG-polyhedron.
Then,  $P$ is non-compact for $n>2$, non-simple for $n>3$ and  $P$ is combinatorially equal to one of the following polyhedra; see also Table \ref{adeg}. 
\begin{enumerate}
\item[\ding{71}] $P$ is a triangle or a tetrahedron.
\item[\ding{71}] $P$ is a doubly-truncated $5$-simplex. 
\item[\ding{71}] $P$ is a pyramid over a product of two or three simplices.  
\item[\ding{71}] $P$ is the polyhedron $P_{\star}\subset \mathbb H^{9}$ depicted in Figure \ref{figN9}.
\end{enumerate}}
\end{thm*}
 
\hip
We emphasize that our approach  can accommodate any prescribed set of dihedral angles.  
However, handling configurations involving pairs of disjoint facets is more delicate and is not addressed in our method. 

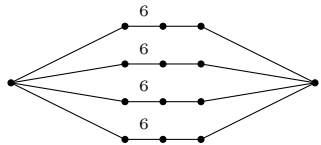
\begin{figure}[!h]
\centering
\begin{tikzpicture}[
every edge/.style = {draw=black},
vrtx/.style args = {#1/#2}{%
circle, draw,
minimum size=1mm, label=#1:#2}
]
\tikzstyle{every node}=[font=\tiny]
\node(A) [vrtx=above/,scale=0.22,fill=black] at (0,0) {};
\node(B) [vrtx=above/,scale=0.22,fill=black] at (0.5,0) {};
\node(C) [vrtx=above/,scale=0.22,fill=black] at (1,0) {};
\node(D) [vrtx=above/,scale=0.22,fill=black] at (0,0.5) {};
\node(E) [vrtx=above/,scale=0.22,fill=black] at (0.5,0.5) {};
\node(F) [vrtx=above/,scale=0.22,fill=black] at (1,0.5) {};
\node(G) [vrtx=above/,scale=0.22,fill=black] at (0,1) {};
\node(H) [vrtx=above/,scale=0.22,fill=black] at (0.5,1) {};
\node(I) [vrtx=above/,scale=0.22,fill=black] at (1,1) {};
\node(J) [vrtx=above/,scale=0.22,fill=black] at (0,1.5) {};
\node(K) [vrtx=above/,scale=0.22,fill=black] at (0.5,1.5) {};
\node(L) [vrtx=above/,scale=0.22,fill=black] at (1,1.5) {};
\node(y) [vrtx=above/,scale=0.22,fill=black] at (-1.5,0.75) {};
\node(x) [vrtx=above/,scale=0.22,fill=black] at (2.5,0.75) {};
\draw (A) -- (B) node[midway,above] {$6$};
\path (C) edge (B);
\draw (D) -- (E) node[midway,above] {$6$};
\path (E) edge (F);
\draw (G) -- (H) node[midway,above] {$6$};
\path (I) edge (H);
\draw (J) -- (K) node[midway,above] {$6$};
\path (L) edge (K)
(y) edge (A)
(y) edge (D)
(y) edge (G)
(y) edge (J)
(x) edge (C)
(x) edge (F)
(x) edge (I)
(x) edge (L);
\end{tikzpicture}
\caption{The Coxeter polyhedron $P_{\star} \subset \mathbb H^9$}
\label{figN9}
\end{figure}

 \hip 

 Preliminary material for this work is presented in Section \ref{Section1}.  Section  \ref{Section2} focuses on  Coxeter polyhedra with mutually intersecting facets and contains a description of the simple case in view of Felikson and Tumarkin's work \cite{FT1}, as well of the ADE-case due to Prokhorov's work \cite{Prk}. 
The proposition outlined above is established in Section \ref{Section2.3},  and the classification of ADEG-polyhedra is stated in Section \ref{Section3}.  
The proof of the main theorem is given in Section \ref{Section4},  followed by a short discussion on the arithmeticity and commensurability classes of the associated Coxeter groups in Section \ref{Section5}. Some auxiliary data are collected in Appendices \ref{AppendixA},  \ref{AppendixB} and \ref{AppendixC}. 
This work is primarily based  on a chapter of the author's Ph.D. thesis; see  \cite[Chapter 4]{Bredon}.

\hop
 {\bf Acknowledgements.}
\noindent
The author thanks Prof. Dr. Ruth Kellerhals for their guidance, helpful comments and insights all throughout the doctoral research and the writing of the PhD manuscript, on which this article is primarily based.

\section{Hyperbolic Coxeter polyhedra}
\label{Section1}

\subsection{Hyperbolic Coxeter polyhedra and associated Coxeter groups}
Let $\mathbb H^n$  be the hyperbolic space interpreted as the upper sheet of the hyperboloid in $\mathbb R^{n+1}$ 
\[
\mathbb H^n=\{x\in\mathbb R^{n+1}\mid \,\langle x,x\rangle\,=-1\,,\,x_{n+1}>0\}\,,
\]
where $\,\langle x,x\rangle\,=x_1^2+\dots+x_n^2-x_{n+1}^2$ is the standard Lorentzian form, and let its boundary be identified with the set 
\[
\partial \mathbb H^n=\{x\in\mathbb R^{n+1}\mid \,\langle x,x\rangle_{n,1}\,=0\,,\,\sum_{k=1}^{n+1}x_k^2=1\,,\,x_{n+1}>0\}\,.
\]

\hip 

Consider a (finite-volume) polyhedron $P\subset \mathbb H^n$,   that is,  the non-empty intersection of finitely many closed half-spaces $H_{v_i}^-:=\{x \in \mathbb H^n \mid \langle x, v_i\rangle \leq 0 \}$  bounded by  hyperplanes $H_{v_i}$ in $ \mathbb H^n$ 
\begin{equation*}
P=\bigcap_{i=1}^N H_{v_i}^- \ , 
\label{polyhedron}
\end{equation*}
where $N\geq n+1$. 
Here, $v_i$ denotes a unit vector that is normal to the hyperplane $H_{v_i}$ and pointing outwards
of $P$,  interpreted as a spacelike vector in $\mathbb R^{n,1}$. 
\hip 
By construction,  $P$ is convex and completely determined by its normal vectors  $v_1, \dots, v_N$ 
up to isometry. The polyhedron $P$ is said to be a   {\it Coxeter polyhedron} if all of its dihedral angles are integer submultiples of $\pi$.

\hip 
Assume that $P\subset \mathbb H^n$ is a Coxeter polyhedron. 
The Gram matrix $\hbox{Gr}(P)=(\langle v_i, v_j\rangle)_{i,j}$ of $P$  is a real matrix with $1$'s 
on the diagonal and non-positive coefficients off the diagonal given by  \\ 
\begin{equation*}
\langle v_i, v_j\rangle= \left \{ 
\begin{array}{ll}
-\hbox{cos}\frac{\pi}{m_{ij}} &  \text{ if }  \ \measuredangle H_i, H_j = \frac{\pi}{m_{ij}} \ ,  \\ 
-1  & \text{ if }  \  H_i, H_j \text{ are parallel }  \ ,  \\ 
-\hbox{cosh} l   & \text{ if }  \  d(H_i, H_j) = l > 0  \ . 
\end{array}
\right . 
\end{equation*}
The Gram matrix $\hbox{Gr}(P)$ is symmetric and has signature $(n,1)$.

\hop 
The {\it Coxeter diagram} $\Sigma$ of $P$ is the non-oriented graph whose nodes correspond to the bounding hyperplanes of $P$,  where two nodes are joined by an edge labeled by $m_{ij}$ whenever $m_{ij}>2$. Note that the label $m_{ij}$ is omitted if $m_{ij}=3$. If two hyperplanes are disjoint in $\mathbb H^n\cup\partial\mathbb H^n$, the corresponding nodes in $\Sigma$ are connected by an edge labeled by $m_{ij}=\infty$. \\
A Coxeter diagram is said to have  {\it rank} $k$ if the associated Gram matrix is positive definite of  rank $k$,  or positive semidefinite of rank $k$. Recall that the Coxeter diagram of a hyperbolic Coxeter polyhedron of finite volume is connected. 
 
\hip In the following, we denote the  linear Coxeter diagram with successive edges labeled by $k_1,\dots,k_r$ by its  {\it Coxeter symbol} $[k_1,\dots, k_r]$.

\hop

Naturally associated with $P\subset \mathbb H^n$ is the discrete group $\Gamma(P)\subset\hbox{Isom}\mathbb H^n$ generated by the reflections $s_i$ with respect to the bounding hyperplanes $H_{v_i}$ of $P$
 \begin{equation}\Gamma(P)=\langle s_1,\dots, s_N \mid s_i^2=1, (s_is_j)^{m_{ij}}=1 \rangle \, ,  
\label{presentation}\end{equation}
where $m_{ij}\in\{2,3,\ldots, \infty\}$ for all $i\neq j$, called the {\it geometric Coxeter group associated with} $P$. 
In this work, the Coxeter groups that will mainly appear satisfy a crystallographic condition, that is, $m_{ij}\in \{2,3,4,6\}$ for all $i\neq j$.

\subsection{Some important correspondences} 
\label{Section1.1}

In what follows, we describe vertices and faces of hyperbolic Coxeter polyhedra and we recall some important correspondences at the level of their Coxeter diagrams.

\hop 

Let $P  \subset \mathbb H^n$ be a Coxeter polyhedron and denote by $\Sigma$  its Coxeter diagram.  

\hip
A {\it $k$-face} (or a face of dimension $k$) of $P$ is the non-empty intersection of 
$n-k$ bounding hyperplanes of $P$ with $\mathbb H^n$. 
A {\it facet} of $P$ is an $(n-1)$-face of $P$, an {\it edge} is a $1$-face of $P$,  
and an {\it ordinary vertex} is a $0$-face of $P$. 

\hip 
From the work of Vinberg \cite{V1}, it is known that there is a one-to-one correspondence between the 
$(n-k)$-faces of $P$ and  the spherical subdiagrams $\sigma$ of $\Sigma$ of rank $k$.
In particular, every ordinary vertex corresponds to a spherical subdiagram of rank $n$.  

\hip 
 All connected spherical Coxeter diagrams of type $A, D, E$ and $G$ are listed in Table \ref{sphericaldiag},  as they are relevant for this work.  A  complete list can be found, for instance,  in \cite{V0}.

\hip A point $v_{\infty} \in \partial\mathbb H^n$ is an {\it ideal vertex} of $P$ if $v_{\infty} \in\overline{P}$ and the intersection $P\cap S_{v_{\infty}}$ of $P$ with a sufficiently small horosphere $S_{v_{\infty}}$ centered at $v_{\infty}$  is compact when viewed as an $(n-1)$-dimensional Euclidean polyhedron.   
In particular, an ideal vertex of $P$ correspond to an affine subdiagram $\sigma_{\infty}\subset \Sigma$ of rank $n-1$.

\hip
 The connected affine Coxeter diagrams of type $\widetilde A, \widetilde D, \widetilde E$ and $\widetilde G$ are listed in Table \ref{sphericaldiag}. Recall that a disjoint union of affine diagrams of rank $k_i$  is affine of rank  $\sum k_i$; see   \cite{V0}.

\begin{table}[!h]
\footnotesize
\tabcolsep=10pt%
\renewcommand*{\arraystretch}{1.7}
\begin{center}
\begin{tabular}{| c | p{0.75cm}  c  || c | p{0.75cm}  c |}
\hline
$n\geq 1$ & $  A_n$ & \begin{tikzpicture}
\fill[black] (0,0) circle (0.05cm);
\fill[black] (1/2,0) circle (0.05cm);
\fill[black] (1.6,0) circle (0.05cm);
\fill[black] (2.1,0) circle (0.05cm);
\draw (0,0) -- (1/2,0) ;
\draw (0.7,0) -- (1/2,0);
\draw (0.7,0) -- (1.4,0) [dotted];
\draw (1.4,0) -- (1.6,0);
\draw (2.1,0) -- (1.6,0);
\end{tikzpicture}  & $n=1$ & $ \widetilde A_1$ & \begin{tikzpicture}
\tikzstyle{every node}=[font=\tiny]
\fill[black] (0,0) circle (0.05cm);
\fill[black] (1/2,0) circle (0.05cm);
\draw (0,0) -- (1/2,0) node [above,midway]{$\infty$} ;
\end{tikzpicture}
\\ 
&&&
$n\geq 2$ &$\widetilde A_n$ & \begin{tikzpicture}
\fill[black] (0,0) circle (0.05cm);
\fill[black] (1/2,0) circle (0.05cm);
\fill[black] (1.6,0) circle (0.05cm);
\fill[black] (2.1,0) circle (0.05cm);
\fill[black] (1.1,1/3) circle (0.05cm);
\draw (0,0) -- (1/2,0) ;
\draw (0.7,0) -- (1/2,0);
\draw (0.7,0) -- (1.4,0) [dotted];
\draw (1.4,0) -- (1.6,0);
\draw (2.1,0) -- (1.6,0);
\draw (0,0) -- (1.1,1/3) ;
\draw (2.1,0) -- (1.1,1/3) ;
\end{tikzpicture}  \\ 
& & & & & 
\\ 
$n\geq 3$ & $D_n$  & \begin{tikzpicture}
\tikzstyle{every node}=[font=\tiny]
\fill[black] (1/2,0) circle (0.05cm);
\fill[black] (1,0) circle (0.05cm);
\fill[black] (1.5,0) circle (0.05cm);
\fill[black] (2.5,0) circle (0.05cm);
\fill[black] (3,1/4) circle (0.05cm);
\fill[black] (3,-1/4) circle (0.05cm);
\draw (1,0) -- (1/2,0) ; 
\draw (1,0) -- (1.5,0) ;
\draw (2.5,0) -- (3,1/4) ;
\draw (2.5,0) -- (3,-1/4) ;
\draw (1.5,0) -- (1.7,0) ;
\draw (2.3,0) -- (1.7,0)[dotted];
\draw (2.3,0) -- (2.5,0) ;
\end{tikzpicture}  & $n\geq 4$ &  $\widetilde D_n$ & \begin{tikzpicture}
\fill[black] (1/2,1/4) circle (0.05cm);
\fill[black] (1/2,-1/4) circle (0.05cm);
\fill[black] (1,0) circle (0.05cm);
\fill[black] (1.5,0) circle (0.05cm);
\fill[black] (2.5,0) circle (0.05cm);
\fill[black] (3,-1/4) circle (0.05cm);
\fill[black] (3,1/4) circle (0.05cm);
\draw (1/2,1/4) -- (1,0) ;
\draw (1/2,-1/4) -- (1,0) ;
\draw (1,0) -- (1.5,0) ;
\draw (2.5,0) -- (3,1/4) ;
\draw (2.5,0) -- (3,-1/4) ;
\draw (1.5,0) -- (1.7,0) ;
\draw (2.3,0) -- (1.7,0)[dotted];
\draw (2.3,0) -- (2.5,0) ;
\end{tikzpicture} 
\\ $n=6$ & 
$E_6$  & \begin{tikzpicture}
\tikzstyle{every node}=[font=\tiny]
\fill[black] (0,0) circle (0.05cm);
\fill[black] (1/2,0) circle (0.05cm);
\fill[black] (2/2,0) circle (0.05cm);
\fill[black] (3/2,0) circle (0.05cm);
\fill[black] (4/2,0) circle (0.05cm);
\fill[black] (2/2,1/2) circle (0.05cm);
\draw (0,0) -- (4/2,0)  ; 
\draw (2/2,0) -- (2/2,1/2) ; 
\end{tikzpicture} & 
$n=6$ & 
$\widetilde E_6$  & \begin{tikzpicture}
\tikzstyle{every node}=[font=\tiny]
\fill[black] (0,0) circle (0.05cm);
\fill[black] (1/2,0) circle (0.05cm);
\fill[black] (2/2,0) circle (0.05cm);
\fill[black] (3/2,0) circle (0.05cm);
\fill[black] (4/2,0) circle (0.05cm);
\fill[black] (2/2,1/2) circle (0.05cm);
\fill[black] (2/2,1) circle (0.05cm);
\draw (0,0) -- (4/2,0)  ; 
\draw (2/2,0) -- (2/2,1) ; 
\end{tikzpicture} 
 \\$n=7$ & 
$E_7$ & \begin{tikzpicture}
\tikzstyle{every node}=[font=\tiny]
\fill[black] (1/2,0) circle (0.05cm);
\fill[black] (2/2,0) circle (0.05cm);
\fill[black] (3/2,0) circle (0.05cm);
\fill[black] (4/2,0) circle (0.05cm);
\fill[black] (5/2,0) circle (0.05cm);
\fill[black] (6/2,0) circle (0.05cm);
\fill[black] (3/2,1/2) circle (0.05cm);
\draw (1/2,0) -- (6/2,0)  ; 
\draw (3/2,0) -- (3/2,1/2) ; 
\end{tikzpicture}
& 
$n=7$ & 
$\widetilde E_7$ & \begin{tikzpicture}
\tikzstyle{every node}=[font=\tiny]
\fill[black] (0,0) circle (0.05cm);
\fill[black] (1/2,0) circle (0.05cm);
\fill[black] (2/2,0) circle (0.05cm);
\fill[black] (3/2,0) circle (0.05cm);
\fill[black] (4/2,0) circle (0.05cm);
\fill[black] (5/2,0) circle (0.05cm);
\fill[black] (6/2,0) circle (0.05cm);
\fill[black] (3/2,1/2) circle (0.05cm);
\draw (0,0) -- (6/2,0)  ; 
\draw (3/2,0) -- (3/2,1/2) ; 
\end{tikzpicture} \\ 
$n=8$ & $E_8$ & \begin{tikzpicture}
\tikzstyle{every node}=[font=\tiny]
\fill[black] (0,0) circle (0.05cm);
\fill[black] (1/2,0) circle (0.05cm);
\fill[black] (2/2,0) circle (0.05cm);
\fill[black] (3/2,0) circle (0.05cm);
\fill[black] (4/2,0) circle (0.05cm);
\fill[black] (5/2,0) circle (0.05cm);
\fill[black] (6/2,0) circle (0.05cm);
\fill[black] (2/2,1/2) circle (0.05cm);
\draw (0,0) -- (6/2,0)  ; 
\draw (2/2,0) -- (2/2,1/2) ; 
\end{tikzpicture}
& $n=8$ & $\widetilde E_8$ & \begin{tikzpicture}
\tikzstyle{every node}=[font=\tiny]
\fill[black] (0,0) circle (0.05cm);
\fill[black] (1/2,0) circle (0.05cm);
\fill[black] (2/2,0) circle (0.05cm);
\fill[black] (3/2,0) circle (0.05cm);
\fill[black] (4/2,0) circle (0.05cm);
\fill[black] (5/2,0) circle (0.05cm);
\fill[black] (6/2,0) circle (0.05cm);
\fill[black] (7/2,0) circle (0.05cm);
\fill[black] (2/2,1/2) circle (0.05cm);
\draw (0,0) -- (7/2,0)  ; 
\draw (2/2,0) -- (2/2,1/2) ; 
\end{tikzpicture} \\  
& & & & & 
\\ 
$n=2$ &  $G_2^{(m)}$ & \begin{tikzpicture}
\tikzstyle{every node}=[font=\tiny]
\fill[black] (0,0) circle (0.05cm);
\fill[black] (1/2,0) circle (0.05cm);
\draw (0,0) -- (1/2,0) node [above,midway]{$m$} ;
\end{tikzpicture} 
& $n=2$ &  $ \widetilde G_2$ &
\begin{tikzpicture}
\tikzstyle{every node}=[font=\tiny]
\fill[black] (0,0) circle (0.05cm);
\fill[black] (1/2,0) circle (0.05cm);
\fill[black] (1,0) circle (0.05cm);
\draw (0,0) -- (1/2,0) node [above,midway]{$6$} ;
\draw (1,0) -- (1/2,0) ;
\end{tikzpicture} \\
\hline 
\end{tabular}
\end{center}
\caption{The irreducible spherical and affine Coxeter groups of type A, D, E and G}
\label{sphericaldiag}
\end{table}

\hip 
 The polyhedron $P$ is  {\it non-compact} if it admits at least one ideal vertex, and it is {\it compact} otherwise.

 \hip 
A vertex  of $P$ is  said to be {\it simple} if it is the intersection of exactly $n$ hyperplanes among the $N$ bounding hyperplanes of $P$.  The  polyhedron $P$ is said to be {\em simple} if all of its vertices are simple. 
Note that all compact hyperbolic polyhedra are simple, and that an ideal vertex of a non-compact polyhedron is non-simple if and only if $\sigma_{\infty}$ consists of at least two affine components.

\begin{example}\normalfont
Consider the Coxeter polyhedron $P\subset \mathbb H^7$ whose Coxeter diagram is depicted below. 
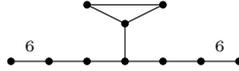
\vspace{0.3em}
$$
 \begin{tikzpicture}
\tikzstyle{every node}=[font=\tiny]
\fill[black] (0,0) circle (0.05cm);
\fill[black] (1/2,0) circle (0.05cm);
\fill[black] (1,0) circle (0.05cm);
\fill[black] (3/2,0) circle (0.05cm);
\fill[black] (3/2,1/2) circle (0.05cm);
\fill[black] (4/2,0.75) circle (0.05cm);
\fill[black] (2/2,0.75) circle (0.05cm);
\fill[black] (4/2,0) circle (0.05cm);
\fill[black] (5/2,0) circle (0.05cm);
\fill[black] (6/2,0) circle (0.05cm);
\draw (0,0) -- (1/2,0) node [above,midway] {$6$} ;
\draw (1,0) -- (1/2,0);
\draw (1,0) -- (3/2,0) ;
\draw (5/2,0) -- (6/2,0) node [above,midway] {$6$} ;
\draw (5/2,0) -- (4/2,0);
\draw (4/2,0) -- (3/2,0) ;
\draw (3/2,0) -- (3/2,1/2) ;
\draw (4/2,0.75) -- (3/2,1/2) ;
\draw (1,0.75) -- (3/2,1/2) ;
\draw (2/2,0.75) -- (4/2,0.75) ;
\end{tikzpicture}$$ 
\noindent
The polyhedron $P$ is a pyramid whose apex   is a non-simple vertex whose vertex link corresponds to the affine subdiagram  $\widetilde A_2 \cup \widetilde G_2\cup\widetilde G_2$.  It has two more ideal vertices, which are simple and  correspond to the subdiagrams of type $\widetilde E_6$.  
\end{example}

\subsection{The finite-volume criterion of Vinberg} 
\label{Section1.2}

  The following theorem established by Vinberg provides a characterization for hyperbolic Coxeter polyhedra of finite volume. It will be fundamental for our work.

\begin{theorem}[Vinberg \cite{V1}] 
A polyhedron $P\subset \mathbb H^n$ has finite volume if and only if the following holds.
\begin{enumerate}
\item $P$ has at least one (ordinary or ideal) vertex. 
\item For every (ordinary or ideal) vertex $v$ of $P$ and every edge of $P$ emanating from $v$,  there is exactly one other (ordinary or ideal) vertex of $P$ on that edge.
\end{enumerate}
\label{thmVin2}
\end{theorem}
\noindent
Considering the above section, 
Theorem \ref{thmVin2} can be reformulated as follows 
for a Coxeter polyhedron $P\subset \mathbb H^n$ with Coxeter diagram $\Sigma$. 
\begin{enumerate}
\item $\Sigma$ contains at least one spherical subdiagram of rank $n$ or one affine subdiagram of rank $n-1$.
\item Every spherical subdiagram of rank $n-1$ of $\Sigma$ can be extended in exactly
two ways to a spherical subdiagram of rank $n$ or to an affine subdiagram of rank $n-1$.  
\end{enumerate}

\hip 
We state below another non-trivial result, due to Felikson and Tumarkin \cite{FT2}, 
which will be useful later on. 

\begin{prop}[Felikson, Tumarkin \cite{FT2}] Let $P\subset \mathbb H^n$ be a hyperbolic Coxeter polyhedron with Coxeter diagram $\Sigma$.  
Then, no proper subdiagram of $\Sigma$ is the Coxeter diagram of a (finite-volume) Coxeter polyhedron in $\mathbb H^n$. 
\label{propsubdiagram}
\end{prop}

\subsection{Faces of Coxeter polyhedra}
\label{Section1.3}

Let $P\subset \mathbb H^n $ be a hyperbolic Coxeter polyhedron with Coxeter diagram $\Sigma$, 
and let $\sigma$ be a spherical subdiagram of rank $k$ of $\Sigma$. 
The associated face $F=F(\sigma)$ of $P$ has codimension $k$, 
and it is called  a {\it $\sigma$-face} in what follows.

\hip 
Due to the work of Allcock \cite{All}, we have a good understanding of the $\sigma$-face $F$. 
Consider two facets  $f_1=F_1\cap F$ and  $f_2=F_2\cap F$ of $F$, where  $F_1, F_2$ are  two  facets of $P$.  Then, the dihedral angle $\measuredangle (f_1, f_2)$ satisfies $
\measuredangle (f_1, f_2) \leq 
\measuredangle (F_1, F_ 2)$. 
As a consequence, the face $F\subset \mathbb H^{n-k}$ is an acute-angled polyhedron (of finite volume),  but it is {\it not} necessarily a Coxeter polyhedron.  
\hip 
The following theorem provides a sufficient condition for the face $F$ to be a Coxeter polyhedron.  It was first proven by Borcherds \cite{Bor} in a more general context 
and has been reformulated as follows  by Allcock \cite{All}.

\begin{theorem}[Allcock \cite{All}]
Let $P$ be a hyperbolic Coxeter polyhedron with Coxeter diagram $\Sigma$.  Let $F$ be a $\sigma$-face of $P$.   Assume that $\sigma$ does not contain any component of type $A_{l}$,  for any $l\geq1$,  or $D_5$.  Then, $F$ is itself a  Coxeter polyhedron.  
\label{thmAll}
\end{theorem}

\hip 
Observe that the facets of $F$ correspond in $\Sigma$ to the nodes which,  together with $\sigma$,  form a spherical subdiagram of $\Sigma$.  These latter nodes in $\Sigma\setminus\sigma$ are called {\it good neighbours} of $\sigma$. If the subdiagram spanned by $\sigma$ and a node $\nu$ is not spherical, $\nu$ will be called a {\it bad} neighbour of $\sigma$.

\hip 
In addition, the  dihedral angle  $\measuredangle (f_1, f_2)$  formed by  the facets $f_1, f_2$ of $F$ 
can be computed  as follows according to Allcock \cite{All}.

\begin{theorem}[Allcock \cite{All}]
Let $P$ be a Coxeter polyhedron. Consider a $\sigma$-face $F$ of $P$ where $\sigma$ does not contain any component of type $A_{l}$,  for $l\geq1$,  or $D_5$. Let $f_1, f_2$ be two facets of $F$ corresponding to two good neighbours  $\nu_1$ and $\nu_2$ of $\sigma$. 
Then,  $\measuredangle (f_1, f_2)$  is characterized as follows. 
\begin{enumerate}
\item If neither $\nu_1$ nor $\nu_2$ attaches to $\sigma$,  then $\measuredangle (f_1, f_2 ) = \measuredangle (F_1, F_2)$. 
\item Assume that $\nu_1$ and  $ \nu_2$ attach to different components of $\sigma$.  If  $ \measuredangle (F_1, F_2)=\frac{\pi}{2}$,  then  $\measuredangle (f_1, f_2) =\frac{\pi}{2}$.  Otherwise,  $f_1$ and $f_2$ are disjoint.  
\item Assume $\nu_1 $ and  $ \nu_2$ attach to the same component $\sigma_0$ of $\sigma$.  If $\nu_1$ and $\nu_2$ are not joined by an edge and  $\{\sigma_0,  \nu_1, \nu_2\}$ yields a diagram $E_6$ (respectively, $E_8$ or $F_4$),   then  $ \measuredangle (f_1, f_2 ) =\frac{\pi}{3}$ (respectively,  $\frac{\pi}{4}$).   Otherwise,  $f_1$ and $f_2$ are disjoint.  
\item Assume  that $\nu_1$ attaches to a component $\sigma_0$ of $\sigma$,  and $\nu_2$ does not attach to $\sigma$.  If $ \measuredangle (F_1, F_2)=\frac{\pi}{2}$,  then  $ \measuredangle (f_1, f_2 ) =\frac{\pi}{2}$.    If $\nu_1$ and $\nu_2$ are joined by a simple edge and $\{\sigma_0,  \nu_i, \nu_j\}$ yields a diagram $B_k$ (respectively,  $D_k$,  $E_8$ or $H_4$),   then $ \measuredangle (f_1, f_2 ) =\frac{\pi}{4}$ (respectively,  $\frac{\pi}{4}$,  $\frac{\pi}{6}$ or $\frac{\pi}{10}$).   Otherwise,  $f_1$ and $f_2$ are disjoint.  
\end{enumerate}
\label{thmAll2}
\end{theorem}

\begin{example}
\normalfont
Let $P\subset \mathbb H^5$ be the Coxeter polyhedron depicted below.  
$$
\begin{tikzpicture}
\tikzstyle{every node}=[font=\tiny]
\fill[black] (0,0) circle (0.05cm);
\fill[black] (1/2,0) circle (0.05cm);
\fill[black] (1,0) circle (0.05cm);
\fill[black] (3/2,0) circle (0.05cm);
\fill[black] (4/2,0) circle (0.05cm);
\fill[black] (5/2,0) circle (0.05cm);
\fill[black] (6/2,0) circle (0.05cm);
\draw (0,0) -- (1/2,0) node [above,midway]{$5$} ;
\draw (0,0) -- (6/2,0) ;
\draw[dotted] (5/2,0) -- (6/2,0) ;
\end{tikzpicture}
$$
By Theorem \ref{thmAll},  the $G_2^{(5)}$-face of $P$  corresponding to  the subdiagram $G_2^m=[5]$ is the Coxeter prism with Coxeter symbol $[10,3,3,\infty]$.   
\label{ex6}
\end{example}

\begin{remark}
\normalfont
Observe that for any integer $m\geq 6$,
a node that is connected to a subdiagram of type  $\sigma=G_2^{(m)}=[m]\subset \Sigma$
 is necessarily a bad neighbour of $\sigma$; see Table \ref{sphericaldiag}.  
Furthermore,  by Theorem \ref{thmAll2},  the dihedral angles formed by the facets of $F$ are precisely given by the dihedral angles of the corresponding facets of $P$.
As a consequence,   the corresponding $\sigma$-face $F$ of dimension $n-2$  of $P$ is easily identified as its Coxeter diagram $\sigma_F$ is, in fact, a {\it subdiagram}  of $\Sigma$.  
\label{remG2}
\end{remark}

\begin{example}
\normalfont 
Let $P\subset \mathbb H^5$ be the Coxeter polyhedron depicted below, that is a doubly-truncated simplex; see the work of Im Hof \cite{ImHof}. 
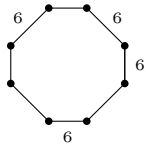
$$
\begin{tikzpicture}
\tikzstyle{every node}=[font=\tiny]
\fill[black] (0,0) circle (0.05cm);
\fill[black] (1/2,1/2) circle (0.05cm);
\fill[black] (1,1/2) circle (0.05cm);
\fill[black] (1.5,0) circle (0.05cm);
\fill[black] (1.5,-1/2) circle (0.05cm);
\fill[black] (0,-1/2) circle (0.05cm);
\fill[black] (1/2,-1) circle (0.05cm);
\fill[black] (1,-1) circle (0.05cm);
\draw (0,0) -- (1/2,1/2) node [midway,yshift=0.3em,xshift=-0.4em] {$6$} ;
\draw (1,1/2) -- (1/2,1/2);
\draw (1.5,0) -- (1,1/2) node [midway,xshift=0.4em,yshift=0.3em] {$6$} ;
\draw (1.5,0) -- (1.5,-1/2)node [right,midway] {$6$} ;
\draw (1.5,0) -- (1.5,-1/2);
\draw (1,-1) -- (1.5,-1/2) ;
\draw (1,-1) -- (1/2,-1)node [below,midway] {$6$} ;
\draw (1/2,-1) -- (0,-1/2);
\draw (0,0) -- (0,-1/2);
\end{tikzpicture}$$
\noindent 
By Theorem \ref{thmAll} and Remark \ref{remG2},  it has four $G_2^{(6)}$-faces, given by the Coxeter tetrahedra with Coxeter symbols $[6,3,3]$
 (or $[3,3,6]$),  $[6,3,6]$ and $[3,6,3]$. 
\label{ex6}
\end{example}

\section{Coxeter polyhedra with mutually intersecting facets}
\label{Section2}

In this section, we consider hyperbolic Coxeter polyhedra with {\it mutually intersecting facets},  that is,  
admitting no pair of disjoint facets in $\mathbb H^n\cup\partial\mathbb H^n$. 

\subsection{Simple polyhedra}
 \label{Section2.1}

We start with a classification result due to Felikson and Tumarkin   \cite{FT1}  in the subclass of  {\it simple} Coxeter polyhedra.  

\begin{theorem}[Felikson,  Tumarkin \cite{FT1}]
 Let  $P\subset \mathbb H^n$ be a Coxeter polyhedron with mutually intersecting facets .
\begin{enumerate}
\item 
 If $P$ is compact,  then $n\leq 4$ and $P$ is either a simplex,  or it is an Esselmann polyhedron. 
\item 
If $P$ is non-compact but simple,  then  $n\leq 9$ and $P$ is either a simplex,  or  it is isometric to the polyhedron  $P_0 \subset \mathbb H^4$  depicted in Figure \ref{figP0}. 
\end{enumerate}
\label{thmFT}
\end{theorem}

\begin{figure}[!h]
\centering
\begin{tikzpicture}
\tikzstyle{every node}=[font=\tiny]
\fill[black] (-1/2,1/4) circle (0.05cm);
\fill[black] (0,0) circle (0.05cm);
\fill[black] (1/2,0) circle (0.05cm);
\fill[black] (0,1/2) circle (0.05cm);
\fill[black] (1/2,1/2) circle (0.05cm);
\fill[black] (1,1/4) circle (0.05cm);
\draw (0,0) -- (-1/2,1/4) node [below,midway] {4} ;
\draw (-1/2,1/4) -- (0,1/2) node [above,midway] {4} ;
\draw (1/2,0) -- (1/2,1/2) ;
\draw (1/2,1/2) -- (0,1/2) ;
\draw (1/2,0) -- (0,0) ;
\draw (1,1/4) -- (1/2,0) node [below,midway] {4} ;
\draw (1,1/4) -- (1/2,1/2) node [above,midway] {4} ;
\draw (0,0) -- (0,1/2) ; 
\end{tikzpicture} 
\caption{The Coxeter polyhedron $P_0\subset\mathbb H^4$}
\label{figP0}
\end{figure}
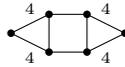

\hip  
From Theorem \ref{thmFT}, all simple Coxeter polyhedra with mutually intersecting facets are known. Let us recall that  Esselmann polyhedra are $4$-dimensional polyhedra classified by Esselmann and their list can be found in \cite{Ess}. 
\hip 
 Furthermore, 
from the combinatorial structure of non-simple vertices in dimensions $n\leq 4$,  we deduce the following corollary.  

\begin{corollary} 
Let $n\leq4$.  If $P\subset \mathbb H^n$ is a finite-volume Coxeter polyhedron with mutually intersecting facets, 
then  $P$ is a simplex,  an Esselmann polyhedron, or $P$ is isometric to the polyhedron $P_0\subset\mathbb H^4$  depicted in Figure \ref{figP0}.  
\label{cor}
\end{corollary}

\begin{proof}
Let  $2\leq n\leq 4$ and let   $P\subset \mathbb H^n$ be a Coxeter polyhedron with mutually intersecting facets.   
By means of Theorems \ref{thmFT},  if $P$ is simple,  
$P$ is either a simplex or an Esselmann polyhedron,  
or it is isometric to the polyhedron $P_0\subset \mathbb H^4$.   

\hip Assume that $P$ is non-simple. In particular,  assume that $n>2$ as polygons are all simple. Let $v_{\infty}\in\partial\mathbb H^n$  be a non-simple vertex of $P$. 
In the Coxeter diagram $\Sigma$ of $P$, 
the ideal vertex $v_{\infty}$ corresponds to an affine Coxeter subdiagram $\sigma_{\infty}=\sigma_1\cup \ldots \cup \sigma_m$ such that  $\hbox{rank}(\sigma_{\infty})= \sum_{i=1,\dots,m} \hbox{rank}(\sigma_i) = n-1 $.  Furthermore,  one has  $m\geq 2$ as $v_{\infty}$ is non-simple. 
We deduce that $\sigma_{\infty}$ admits at least one affine component of rank $1$.  
Therefore,  $P$ admits at least one pair of parallel facets and this finishes the proof. 
\end{proof}

\subsection{ADE-polyhedra}
 \label{Section2.2}

Among all Coxeter polyhedra with mutually intersecting facets,  we now describe a subfamily completely classified by Prokhorov \cite{Prk}. 

\begin{definition}
An  {\em ADE-polyhedron} is a Coxeter polyhedron (of finite volume) $P\subset \mathbb H^n$  satisfying the following properties.  
\begin{enumerate}
\item [i.]All facets of $P$ are mutually intersecting. 
\item[ii.] All dihedral angles of $P$ are equal to $\frac{\pi}{2}$ or $\frac{\pi}{3}$.  
\end{enumerate}
\end{definition}

\noindent The terminology is inspired by the spherical and affine cases where Coxeter polyhedra with dihedral angles $\frac{\pi}{2}$ or $\frac{\pi}{3}$ only 
 are of type 
 $A_k$,  $D_k$ or $E_ 6, E_7,  E_8$,   as well as  of type $\widetilde A_k$,  $ \widetilde D_k$,  or $ \widetilde E_6,  \widetilde E_7$ and $ \widetilde E_8$. 

\hop 
Based on Nikulin's work  \cite{Nikulin, Nikulin2} on even reflective lattices, 
Prokhorov \cite{Prk} derived that ADE-polyhedra do not exist in  $\mathbb H^n$ for $n>17$, and he established the following classification result.

\begin{theorem}[Prokhorov \cite{Prk}]
Let $P$ be an ADE-polyhedron.  Then,  $P$ is either a simplex,  a pyramid,  or one of the Coxeter polyhedra $P_1\subset \mathbb H^8$ and $P_2 \subset\mathbb H^9$ depicted in Figure \ref{Prk}.  
\label{thmPrk}
\end{theorem}

\begin{figure}[!h]
\centering
\begin{tikzpicture}[
every edge/.style = {draw=black, },
 vrtx/.style args = {#1/#2}{%
      circle, draw,
      minimum size=1mm, label=#1:#2}
                    ]
\node(x)  [vrtx=above/,scale=0.22,fill=black] at (1/2,1/4) {};
\node(y)  [vrtx=above/,scale=0.22,fill=black] at (1/4,-1/2) {};
\node(z)  [vrtx=above/,scale=0.22,fill=black] at (-1/2,-1/4) {};
\node(t)  [vrtx=above/,scale=0.22,fill=black] at (-1/4,1/2) {};
\tkzDefPoint(0,0){O}\tkzDefPoint(1,0){A}
\tkzDefPointsBy[rotation=center O angle 360/8](A,B,C,D,E,F,G){B,C,D,E,F,G,H}
\tkzDrawPoints[fill =black,size=2.2,color=black](A,B,C,D,E,F,G,H)
\tkzDrawPolygon[thin](A,B,C,D,E,F,G,H)
\path   (A) edge (x)
   (B) edge (x)
  (z) edge (x)
(C) edge (t)
   (D) edge (t)
  (t) edge (y)
(E) edge (z)
   (F) edge (z)
(G) edge (y)
   (H) edge (y);
\end{tikzpicture}
\qquad 
\begin{tikzpicture}[
every edge/.style = {draw=black, },
 vrtx/.style args = {#1/#2}{%
      circle, draw,
      minimum size=1mm, label=#1:#2}
                    ]
\node(x)  [vrtx=above/,scale=0.22,fill=black] at (1/2,1/3) {};
\node(y)  [vrtx=above/,scale=0.22,fill=black] at (1/8,-1/2) {};
\node(z)  [vrtx=above/,scale=0.22,fill=black] at (-1/2,1/4) {};
\tkzDefPoint(0,0){O}\tkzDefPoint(1,0){A}
\tkzDefPointsBy[rotation=center O angle 360/9](A,B,C,D,E,F,G,H){B,C,D,E,F,G,H,I}
\tkzDrawPoints[fill =black,size=2.2,color=black](A,B,C,D,E,F,G,H,I)
\tkzDrawPolygon[thin](A,B,C,D,E,F,G,H,I)
\path   (A) edge (x)
   (C) edge (x)
(D) edge (z)
   (F) edge (z)
(G) edge (y)
   (I) edge (y);
\end{tikzpicture}
\caption{The ADE-polyhedra $P_1\subset \mathbb H^8$ and $P_2 \subset\mathbb H^9$}
\label{Prk}
\end{figure}
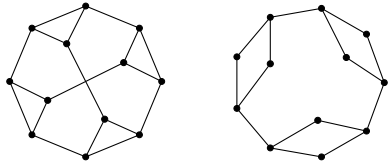

\hip 
In total,  there are  $34$ hyperbolic ADE-polyhedra, all of which are  non-compact of dimensions $n\geq 3$.  
They are all listed in Appendix \ref{AppendixA},  
as they will play a role in the proof of our main theorem. In fact, they will appear as possible $G_2$-faces of ADEG-polyhedra; see Section \ref{Section4}. 

\hip
Note that the highest-dimensional ADE-polyhedron is the highest-dimensional Coxeter pyramid $P_{17}\subset \mathbb H^{17}$ depicted in Figure \ref{17}; see also \cite{Nikulin, Tum2}. 

\vspace{0.5em}

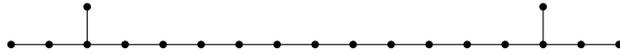
\begin{figure}[!h]
\centering
\begin{tikzpicture}
\tikzstyle{every node}=[font=\tiny]
\fill[black] (0,0) circle (0.05cm);
\fill[black] (1/2,0) circle (0.05cm);
\fill[black] (1,0) circle (0.05cm);
\fill[black] (3/2,0) circle (0.05cm);

\fill[black] (4/2,0) circle (0.05cm);
\fill[black] (5/2,0) circle (0.05cm);
\fill[black] (6/2,0) circle (0.05cm);

\fill[black] (7/2,0) circle (0.05cm);
\fill[black] (8/2,0) circle (0.05cm);
\fill[black] (9/2,0) circle (0.05cm);
\fill[black] (10/2,0) circle (0.05cm);

\fill[black] (11/2,0) circle (0.05cm);
\fill[black] (12/2,0) circle (0.05cm);
\fill[black] (13/2,0) circle (0.05cm);
\fill[black] (14/2,0) circle (0.05cm);
\fill[black] (15/2,0) circle (0.05cm);
\fill[black] (16/2,0) circle (0.05cm);

\fill[black] (2/2,1/2) circle (0.05cm);
\fill[black] (14/2,1/2) circle (0.05cm);

\draw (14/2,1/2) -- (14/2,0);
\draw (2/2,1/2) -- (2/2,0);
\draw (0,0) -- (16/2,0);
\end{tikzpicture} 
\vspace{0.5em}
\caption{The ADE-pyramid $P_{17} \subset \mathbb H^{17}$}
\label{17}
\end{figure}

\section{Universal lower bound for the dihedral angles}
 \label{Section2.3}

From the works of Khovanski \cite{Khov} and  Prokhorov \cite{Prk0} it is well known that hyperbolic Coxeter polyhedra of finite volume do not exist in dimensions greater than $995$. In addition, due to Vinberg \cite{V2}, compact hyperbolic Coxeter polyhedra do not exist in dimensions beyond $29$. 
\hip 
In what follows, are insterested in providing a universal lower bound for the dihedral angles of hyperbolic Coxeter polyhedra of large dimensions.

\begin{prop}Let $n\geq 32$ and let $P\subset \mathbb H^n$ be a finite-volume Coxeter polyhedron. Then, a non-zero dihedral angle of $P$ is of the form  $\frac{\pi}{m}$ with $m\leq 6$.  
\label{geq32}
\end{prop}

\begin{proof}
Let $P$ be a  Coxeter polyhedron of dimension  $n\geq 32$,  and denote by $\Sigma$ its Coxeter diagram. 
Observe that from the  fundamental result of Vinberg  \cite{V2} on  the non-existence of  {\it compact} Coxeter polyhedra in dimension $n\geq 30$, all faces of codimension $2$ of $P$ are non-compact. 

\hip 
Assume that $P$ admits a non-zero dihedral angle $\frac{\pi}{m}$ with $m\geq 7$, and let $\sigma=[m]$ be the corresponding spherical subdiagram in $\Sigma$. 
 By means of Theorems \ref{thmAll}  and \ref{thmAll2},   the associated face   $F=F(\sigma)$ of codimension $2$ of $P$ 
 is a non-compact Coxeter polyhedron whose  Coxeter diagram $\sigma_F$  is disjoint of $\sigma$;  Remark \ref{remG2}. 
Therefore, it contains an affine Coxeter subdiagram of rank $n-3$, which is necessarily a component of an affine subdiagram of rank $n-1$ in $\Sigma$. 
As  $\Sigma\setminus\sigma_F^{\infty}$  contains  $\sigma$ which cannot be extended to yield an affine subdiagram for $m\geq 7$,   we obtain a contradiction.
\end{proof}

When restricting to Coxeter polyhedra with mutually intersecting facets, this property holds for all dimensions $n\geq 7$. 

\begin{prop}
Let $n\geq 7$,  and let $P\subset\mathbb H^n$ be a Coxeter polyhedron with  mutually intersecting facets.  Then,  a non-zero dihedral angle of $P$ is of the form  $\frac{\pi}{m}$ with $m\leq 6$.  
\label{angles}
\end{prop}

\begin{proof}
Let $n\geq 7$, and let  $P\subset\mathbb H^n$  be a Coxeter polyhedron with mutually intersecting facets.  
Assume that $P$ has a dihedral angle $\frac{\pi}{m}$ with $m\geq 7$.  
Therefore, the Coxeter diagram  $\Sigma$ of $P$ contains a subdiagram $\sigma=G_2^{(m)}=[m]$.

\hip 
By means of Theorems \ref{thmAll}  and \ref{thmAll2}, 
we deduce that the corresponding $G_2^{(m)}$-face $F=F(\sigma)$ of $P$ is a Coxeter polyhedron with mutually intersecting facets in $\mathbb H^{n-2}$; see Remark \ref{remG2}. 
As $n-2\geq 5$,  we deduce from Theorem \ref{thmFT} that the face $F$ is a {\it non-compact} Coxeter polyhedron. 
A similar  reasoning as in the proof of the Proposition \ref{geq32} allows to conclude. 
\end{proof}

 Analogously,  we derive the following proposition for ideal polyhedra.

\begin{prop}Let $n>3$ and let $P\subset \mathbb H^n$ be an ideal Coxeter polyhedron. Then, a non-zero dihedral angle of $P$ is of the form  $\frac{\pi}{m}$ with $m\leq 5$.  
\label{ideal}
\end{prop}
\begin{proof}
Let $n>3$ and let $P\subset \mathbb H^n$ be an ideal Coxeter polyhedron, that is, all of its vertices are ideal.    
Assume that there exists a subdiagram $\sigma=G_2^{(m)}$ for $m\geq 6$ in the Coxeter diagram $\Sigma$ of $P$.  As previously, this forces  the appearance of a sudiagram $\sigma_F$ of signature $(n-2,1)$ for the corresponding $G_2^{(m)}$-face,   that is disjoint from  $\sigma$.   Therefore,  there exists  a spherical subdiagram of rank $n$ in $\Sigma$,  and  this contradicts the fact that $P$ only has no ordinary vertices. 
\end{proof}

\hip 
Motivated by these angular obstructions,  and inspired by Prokhorov's work \cite{Prk},  we shall develop next a  strategy  to classify Coxeter polyhedra with prescribed dihedral angles.

 \section{The classification of ADEG-polyhedra}
\label{Section3} 

We introduce the subsequent family of hyperbolic Coxeter polyhedra. 

\begin{definition}  
A polyhedron $P\subset \mathbb H^n$ is said to be an   {\it ADEG-polyhedron} if $P$ has finite volume and it satisfies the following properties.  
\begin{enumerate}
\item[i.] All facets of $P$ are mutually intersecting. 
\item[ii.] All dihedral angles of $P$ are equal to $\frac{\pi}{2},  \frac{\pi}{3}$ or $\frac{\pi}{6}$.
\item[iii.] $P$ has at least one dihedral angle $\frac{\pi}{6}$.  
\end{enumerate}
\end{definition}

\hip 
The terminology is inspired from Prokhorov's work \cite{Prk}, as the ADEG-property implies that the spherical and affine subdiagrams that appear and their related root systems are of type $A_k, D_k, E_k$ and $G_2$.

\hop 
The  complete classification of ADEG-polyhedra is stated is the following theorem.  

\begin{theorem} 
Let $P\subset \mathbb H^n$ be an ADEG-polyhedron. 
Then,  $P$ is one of the $24$ Coxeter polyhedra depicted in Table \ref{adeg}.  
In particular,  $P$ is non-compact for $n>2$,  non-simple for $n>3$, and $P$ is of dimension $n\leq 11$.  Furthermore, 
$P$ is either  a triangle or a tetrahedron,  a doubly-truncated simplex in $\mathbb H^5$,  a  pyramid,   or the Coxeter polyhedron $P_{\star} \subset \mathbb H^{9}$  depicted in Figure \ref{N9}.  
\label{thmN9}
\end{theorem}

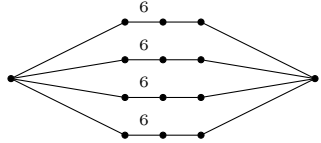
\begin{figure}[!h]
\centering
\begin{tikzpicture}[
every edge/.style = {draw=black},
 vrtx/.style args = {#1/#2}{%
      circle, draw,
      minimum size=1mm, label=#1:#2}
                    ]
\tikzstyle{every node}=[font=\tiny]
\node(A)  [vrtx=above/,scale=0.22,fill=black]  at (0,0) {};
\node(B)  [vrtx=above/,scale=0.22,fill=black]  at (0.5,0) {};
\node(C)  [vrtx=above/,scale=0.22,fill=black]  at (1,0) {};
\node(D)  [vrtx=above/,scale=0.22,fill=black]  at (0,0.5) {};
\node(E)  [vrtx=above/,scale=0.22,fill=black]  at (0.5,0.5) {};
\node(F)  [vrtx=above/,scale=0.22,fill=black]  at (1,0.5) {};
\node(G)  [vrtx=above/,scale=0.22,fill=black]  at (0,1) {};
\node(H)  [vrtx=above/,scale=0.22,fill=black]  at (0.5,1) {};
\node(I)  [vrtx=above/,scale=0.22,fill=black]  at (1,1) {};
\node(J)  [vrtx=above/,scale=0.22,fill=black]  at (0,1.5) {};
\node(K)  [vrtx=above/,scale=0.22,fill=black]  at (0.5,1.5) {};
\node(L)  [vrtx=above/,scale=0.22,fill=black]  at (1,1.5) {};
\node(y)  [vrtx=above/,scale=0.22,fill=black]  at (-1.5,0.75) {};
\node(x)  [vrtx=above/,scale=0.22,fill=black]  at (2.5,0.75) {};
\draw (A) -- (B) node[midway,above] {$6$}; 
\path (C) edge (B);
\draw  (D) -- (E)  node[midway,above] {$6$}; 
\path  (E) edge (F);  
\draw (G) -- (H) node[midway,above] {$6$}; 
\path (I) edge (H);
\draw (J) -- (K) node[midway,above] {$6$}; 
\path (L) edge (K)
 (y) edge (A)
 (y) edge (D)
 (y) edge (G)
 (y) edge (J)
 (x) edge (C)
 (x) edge (F)
 (x) edge (I)
 (x) edge (L);
\end{tikzpicture}
\caption{The Coxeter polyhedron $P_{\star} \subset \mathbb H^{9}$}
\label{N9}
\end{figure}

\begin{table}
\begin{center}
\bigskip
\tabcolsep=13pt%
\renewcommand*{\arraystretch}{1.3} 
\begin{tabular}{c|c}
$n$& \\
\hline
2 & \\ & 
\begin{tikzpicture}\tikzstyle{every node}=[font=\tiny]
\fill[black] (-1,0) circle (0.05cm);
\fill[black] (-1.5,0) circle (0.05cm);
\fill[black] (-2,0) circle (0.05cm);
\draw (-2,0) -- (-1.5,0) node [above,midway] {$6$} ;
\draw (-1.5,0) -- (-1,0)  node [above,midway] {$6$} ;

\fill[black] (0,0) circle (0.05cm);\fill[black] (1,0) circle (0.05cm);\fill[black] (1/2,1/2) circle (0.05cm);\draw (0,0) -- (1,0) node [below,midway] {$6$} ;\draw (1,0) -- (1/2,1/2);\draw (0,0) -- (1/2,1/2) ;\end{tikzpicture} \qquad \begin{tikzpicture}\tikzstyle{every node}=[font=\tiny]\fill[black] (0,0) circle (0.05cm);\fill[black] (1,0) circle (0.05cm);\fill[black] (1/2,1/2) circle (0.05cm);\draw (0,0) -- (1,0) node [below,midway] {$6$} ;\draw (1,0) -- (1/2,1/2) node [right,midway] {$6$} ;\draw (0,0) -- (1/2,1/2)  ;\end{tikzpicture}\qquad \begin{tikzpicture}\tikzstyle{every node}=[font=\tiny]\fill[black] (0,0) circle (0.05cm);\fill[black] (1,0) circle (0.05cm);\fill[black] (1/2,1/2) circle (0.05cm);\draw (0,0) -- (1,0) node [below,midway] {$6$} ;\draw (1,0) -- (1/2,1/2)node [right,midway] {$6$} ;\draw (0,0) -- (1/2,1/2) node [left,midway] {$6$} ;\end{tikzpicture}\\\hline
3 & 
\begin{tikzpicture}
\tikzstyle{every node}=[font=\tiny]
\fill[black] (0,0) circle (0.05cm);
\fill[black] (1/2,0) circle (0.05cm);
\fill[black] (1,0) circle (0.05cm);
\fill[black] (3/2,0) circle (0.05cm);
\draw (0,0) -- (1/2,0) node [above,midway] {$6$} ;
\draw (1,0) -- (1/2,0);
\draw (1,0) -- (3/2,0) ;
\end{tikzpicture}
 \qquad  \quad
\begin{tikzpicture}
\tikzstyle{every node}=[font=\tiny]
\fill[black] (0,0) circle (0.05cm);
\fill[black] (1/2,0) circle (0.05cm);
\fill[black] (1,0) circle (0.05cm);
\fill[black] (3/2,0) circle (0.05cm);
\draw (0,0) -- (1/2,0) node [above,midway] {$6$} ;
\draw (1,0) -- (1/2,0);
\draw (1,0) -- (3/2,0) node [above,midway] {$6$} ;
\end{tikzpicture}
 \qquad  \quad\begin{tikzpicture}
\tikzstyle{every node}=[font=\tiny]
\fill[black] (0,0) circle (0.05cm);
\fill[black] (1/2,0) circle (0.05cm);
\fill[black] (1,0) circle (0.05cm);
\fill[black] (3/2,0) circle (0.05cm);
\draw (0,0) -- (1/2,0)  ;
\draw (1,0) -- (1/2,0)node [above,midway] {$6$};
\draw (1,0) -- (3/2,0) ;
\end{tikzpicture}
\\ 
& 
\begin{tikzpicture}
\tikzstyle{every node}=[font=\tiny]
\fill[black] (0,0) circle (0.05cm);
\fill[black] (1/2,0) circle (0.05cm);
\fill[black] (1,-1/4) circle (0.05cm);
\fill[black] (1,1/4) circle (0.05cm);
\draw (0,0) -- (1/2,0) node [above,midway] {$6$} ; 
\draw (1,1/4) -- (1/2,0);
\draw (1,-1/4) -- (1/2,0) ;
\end{tikzpicture}
 \qquad  \quad
\begin{tikzpicture}
\tikzstyle{every node}=[font=\tiny]
\fill[black] (0,0) circle (0.05cm);
\fill[black] (1/2,0) circle (0.05cm);
\fill[black] (1,1/4) circle (0.05cm);
\fill[black] (1,-1/4) circle (0.05cm);
\draw (0,0) -- (1/2,0) node [above,midway] {$6$} ;
\draw (1,1/4) -- (1/2,0);
\draw (1,-1/4) -- (1/2,0) ;
\draw (1,1/4) -- (1,-1/4);
\end{tikzpicture}  %
 \qquad  \quad 
\begin{tikzpicture}
\tikzstyle{every node}=[font=\tiny]
\fill[black] (0,0) circle (0.05cm);
\fill[black] (1/2,0) circle (0.05cm);
\fill[black] (1/2,-1/2) circle (0.05cm);
\fill[black] (0,-1/2) circle (0.05cm);
\draw (0,0) -- (1/2,0) node [above,midway] {$6$} ;
\draw (1/2,0) -- (1/2,-1/2)   ;
\draw (1/2,-1/2) -- (0,-1/2) ;
\draw (0,0) -- (0,-1/2) ;
\end{tikzpicture} 
 \qquad  \quad
\begin{tikzpicture}
\tikzstyle{every node}=[font=\tiny]
\fill[black] (0,0) circle (0.05cm);
\fill[black] (1/2,0) circle (0.05cm);
\fill[black] (1/2,-1/2) circle (0.05cm);
\fill[black] (0,-1/2) circle (0.05cm);
\draw (0,0) -- (1/2,0) ;
\draw (1/2,0) -- (1/2,-1/2)  node [right,midway] {$6$};
\draw (1/2,-1/2) -- (0,-1/2) ;
\draw (0,0) -- (0,-1/2) node [left,midway] {$6$};
\end{tikzpicture}  
 \\
\hline
5 & 
\begin{tikzpicture}
\tikzstyle{every node}=[font=\tiny]
\fill[black] (0,0) circle (0.05cm);
\fill[black] (1/2,0) circle (0.05cm);
\fill[black] (1,0) circle (0.05cm);
\fill[black] (3/2,0) circle (0.05cm);

\fill[black] (4/2,0) circle (0.05cm);
\fill[black] (5/2,0) circle (0.05cm);
\fill[black] (6/2,0) circle (0.05cm);

\draw (0,0) -- (1/2,0) node [above,midway] {$6$} ;
\draw (1,0) -- (1/2,0);
\draw (1,0) -- (3/2,0) ;

\draw (5/2,0) -- (6/2,0) node [above,midway] {$6$} ;
\draw (5/2,0) -- (4/2,0);
\draw (4/2,0) -- (3/2,0) ;

\tikzstyle{every node}=[font=\tiny]
\fill[black] (4,0) circle (0.05cm);
\fill[black] (9/2,0) circle (0.05cm);
\fill[black] (5,0) circle (0.05cm);
\fill[black] (11/2,0) circle (0.05cm);

\fill[black] (12/2,0) circle (0.05cm);
\fill[black] (13/2,1/4) circle (0.05cm);
\fill[black] (13/2,-1/4) circle (0.05cm);

\draw (4,0) -- (9/2,0) node [above,midway] {$6$} ;
\draw (9/2,0) -- (9/2,0);
\draw (9/2,0) -- (11/2,0) ;

\draw (12/2,0) -- (11/2,0) ;
\draw (13/2,1/4) -- (13/2,-1/4) ; 
\draw (13/2,1/4) -- (12/2,0);
\draw (12/2,0) -- (13/2,-1/4) ;
\end{tikzpicture} 
\\
& 
\begin{tikzpicture}\tikzstyle{every node}=[font=\tiny]\fill[black] (0,0) circle (0.05cm);\fill[black] (1/2,1/2) circle (0.05cm);\fill[black] (1,1/2) circle (0.05cm);\fill[black] (1.5,0) circle (0.05cm);\fill[black] (1.5,-1/2) circle (0.05cm);\fill[black] (0,-1/2) circle (0.05cm);\fill[black] (1/2,-1) circle (0.05cm);\fill[black] (1,-1) circle (0.05cm);

\draw (0,0) -- (1/2,1/2) node [midway,yshift=0.3em,xshift=-0.4em] {$6$} ;
\draw (1,1/2) -- (1/2,1/2);\draw (1.5,0) -- (1,1/2) node [midway,yshift=0.3em,xshift=0.4em] {$6$} ;
\draw (1.5,0) -- (1.5,-1/2);\draw (1.5,0) -- (1.5,-1/2);
\draw (1,-1) -- (1.5,-1/2) node [midway,yshift=-0.3em,xshift=0.4em] {$6$} ;
\draw (1,-1) -- (1/2,-1);\draw (1/2,-1) -- (0,-1/2) node [midway,yshift=-0.3em,xshift=-0.4em] {$6$} ;
\draw (0,0) -- (0,-1/2);

\tikzstyle{every node}=[font=\small]\fill[black] (3,0) circle (0.05cm);\fill[black] (1/2,1/2) circle (0.05cm);\fill[black] (1,1/2) circle (0.05cm);\fill[black] (1.5,0) circle (0.05cm);\fill[black] (1.5,-1/2) circle (0.05cm);\fill[black] (0,-1/2) circle (0.05cm);\fill[black] (1/2,-1) circle (0.05cm);\fill[black] (1,-1) circle (0.05cm);

\tikzstyle{every node}=[font=\tiny]
\fill[black] (3,0) circle (0.05cm);
\fill[black] (7/2,1/2) circle (0.05cm);
\fill[black] (8/2,1/2) circle (0.05cm);
\fill[black] (9/2,0) circle (0.05cm);
\fill[black] (9/2,-1/2) circle (0.05cm);
\fill[black] (3,-1/2) circle (0.05cm);
\fill[black] (7/2,-1) circle (0.05cm);
\fill[black] (8/2,-1) circle (0.05cm);

\draw (3,0) -- (7/2,1/2) node [midway,yshift=0.3em,xshift=-0.4em] {$6$} ;
\draw (4,1/2) -- (7/2,1/2);
\draw (9/2,0) -- (4,1/2) node [midway,xshift=0.4em,yshift=0.3em] {$6$} ;
\draw (9/2,0) -- (9/2,-1/2)node [right,midway] {$6$} ;
\draw (9/2,0) -- (9/2,-1/2);
\draw (8/2,-1) -- (9/2,-1/2) ;
\draw (8/2,-1) -- (7/2,-1)node [below,midway] {$6$} ;
\draw (7/2,-1) -- (3,-1/2);
\draw (3,0) -- (3,-1/2);
\end{tikzpicture} \\ 
\hline
6 & 
\begin{tikzpicture}[
every edge/.style = {draw=black, },
 vrtx/.style args = {#1/#2}{%
      circle, draw,
      minimum size=1mm, label=#1:#2}
                    ]
\node(A)  [vrtx=above/,scale=0.22,fill=black]  at (0,0) {};
\node(B)  [vrtx=above/,scale=0.22,fill=black]  at (0.5,0.25) {};
\node(C)  [vrtx=above/,scale=0.22,fill=black]  at (1,0) {};
\node(D)  [vrtx=above/,scale=0.22,fill=black]  at (0.5,-0.25) {};
\node(E)  [vrtx=above/,scale=0.22,fill=black]  at (1.5,0) {};
\node(F)  [vrtx=above/,scale=0.22,fill=black]  at (2,0) {};
\node(I)  [vrtx=above/,scale=0.22,fill=black]  at (2.5,0) {};
\node(J)  [vrtx=above/,scale=0.22,fill=black]  at (3,0) {};

\path  (A) edge (B)
(C) edge (B)
(C) edge (D)
(A) edge (D)
(E) edge (C)
(E) edge (F)
(I) edge (F);

\tikzstyle{every node}=[font=\tiny]
\draw (I) -- (J)node [above,midway] {$6$} ;
\end{tikzpicture}
 \\
\hline
7
& 
\begin{tikzpicture}[
every edge/.style = {draw=black, },
 vrtx/.style args = {#1/#2}{%
      circle, draw,
      minimum size=1mm, label=#1:#2}
                    ]
\node(A)  [vrtx=above/,scale=0.22,fill=black]  at (0,-0.25) {};
\node(B)  [vrtx=above/,scale=0.22,fill=black]  at (0.5,-0.5) {};
\node(C)  [vrtx=above/,scale=0.22,fill=black]  at (0.25,-1) {};
\node(D)  [vrtx=above/,scale=0.22,fill=black]  at (-0.25,-1) {};
\node(E)  [vrtx=above/,scale=0.22,fill=black]  at (-0.5,-0.5) {};

\node(F) [vrtx=above/,scale=0.22,fill=black]  at (1,-0.5) {};
\node(G) [vrtx=above/,scale=0.22,fill=black]  at (1.5,-0.5) {};
\node(H) [vrtx=above/,scale=0.22,fill=black]  at (2,-0.5) {};
\node(I) [vrtx=above/,scale=0.22,fill=black]  at (2.5,-0.5) {};

\path (A) edge (B)  
(C) edge (B)
(C) edge (D)
(E) edge (D)
(E) edge (A)
(B) edge (F)
(F) edge (G)
(I) edge (G)
;
\tikzstyle{every node}=[font=\tiny]
\draw (I) -- (H)node [above,midway] {$6$} ;

\end{tikzpicture}
\quad 
\begin{tikzpicture}[
every edge/.style = {draw=black, },
 vrtx/.style args = {#1/#2}{%
      circle, draw,
      minimum size=1mm, label=#1:#2}
                    ]
\node(B)  [vrtx=above/,scale=0.22,fill=black]  at (0,0) {};
\node(C)  [vrtx=above/,scale=0.22,fill=black]  at (1,0) {};
\node(D)  [vrtx=above/,scale=0.22,fill=black]  at (0.5,0) {};
\node(E)  [vrtx=above/,scale=0.22,fill=black]  at (1.5,0) {};
\node(F)  [vrtx=above/,scale=0.22,fill=black]  at (2,0) {};
\node(G)  [vrtx=above/,scale=0.22,fill=black]  at (2.5,0) {};
\node(H)  [vrtx=above/,scale=0.22,fill=black]  at (3,0) {};
\node(I)  [vrtx=above/,scale=0.22,fill=black]  at (2.5,0.5) {};
\node(J)  [vrtx=above/,scale=0.22,fill=black]  at (2.5,-0.5) {};

\path 
(C) edge (B)
(C) edge (D)
(E) edge (C)
(E) edge (F)
(F) edge (G)
(H) edge (G)
(I) edge (G) 
(J) edge (G) ;

\tikzstyle{every node}=[font=\tiny]
\draw (B) -- (D)node [above,midway] {$6$} ;
\end{tikzpicture}\quad
\\
& \begin{tikzpicture}
\tikzstyle{every node}=[font=\tiny]
\fill[black] (0,0) circle (0.05cm);
\fill[black] (1/2,0) circle (0.05cm);
\fill[black] (1,0) circle (0.05cm);
\fill[black] (3/2,0) circle (0.05cm);

\fill[black] (3/2,1/2) circle (0.05cm);
\fill[black] (3/2,1) circle (0.05cm);
\fill[black] (3/2,1.5) circle (0.05cm);

\fill[black] (4/2,0) circle (0.05cm);
\fill[black] (5/2,0) circle (0.05cm);
\fill[black] (6/2,0) circle (0.05cm);

\draw (0,0) -- (1/2,0) node [above,midway] {$6$} ;
\draw (1,0) -- (1/2,0);
\draw (1,0) -- (3/2,0) ;

\draw (5/2,0) -- (6/2,0) node [above,midway] {$6$} ;
\draw (5/2,0) -- (4/2,0);
\draw (4/2,0) -- (3/2,0) ;

\draw (3/2,1/2) -- (3/2,0) ;
\draw (3/2,1/2) -- (3/2,1) ;
\draw (3/2,1.5) -- (3/2,1) node [right,midway] {$6$} ;
\end{tikzpicture} \quad 
 \begin{tikzpicture}
\tikzstyle{every node}=[font=\tiny]
\fill[black] (0,0) circle (0.05cm);
\fill[black] (1/2,0) circle (0.05cm);
\fill[black] (1,0) circle (0.05cm);
\fill[black] (3/2,0) circle (0.05cm);

\fill[black] (3/2,1/2) circle (0.05cm);
\fill[black] (4/2,0.75) circle (0.05cm);
\fill[black] (2/2,0.75) circle (0.05cm);

\fill[black] (4/2,0) circle (0.05cm);
\fill[black] (5/2,0) circle (0.05cm);
\fill[black] (6/2,0) circle (0.05cm);

\draw (0,0) -- (1/2,0) node [above,midway] {$6$} ;
\draw (1,0) -- (1/2,0);
\draw (1,0) -- (3/2,0) ;

\draw (5/2,0) -- (6/2,0) node [above,midway] {$6$} ;
\draw (5/2,0) -- (4/2,0);
\draw (4/2,0) -- (3/2,0) ;

\draw (3/2,0) -- (3/2,1/2) ;

\draw (4/2,0.75) -- (3/2,1/2) ;
\draw (1,0.75) -- (3/2,1/2) ;
\draw (2/2,0.75) -- (4/2,0.75) ;
\end{tikzpicture}
\quad 
\begin{tikzpicture}[
every edge/.style = {draw=black, },
 vrtx/.style args = {#1/#2}{%
      circle, draw,
      minimum size=1mm, label=#1:#2}
                    ]
\tikzstyle{every node}=[font=\tiny]
\node(A)  [vrtx=above/,scale=0.22,fill=black]  at (0,0) {};
\node(B)  [vrtx=above/,scale=0.22,fill=black]  at (0.5,0.25) {};
\node(C)  [vrtx=above/,scale=0.22,fill=black]  at (0,0.5) {};

\node(D)  [vrtx=above/,scale=0.22,fill=black]  at (1,0.25) {};

\node(E)  [vrtx=above/,scale=0.22,fill=black]  at (1.5,0.25) {};
\node(F)  [vrtx=above/,scale=0.22,fill=black]  at (2,0.5) {};
\node(G)  [vrtx=above/,scale=0.22,fill=black]  at (2,0) {};

\node(H)  [vrtx=above/,scale=0.22,fill=black]  at (1,0.5) {};
\node(I)  [vrtx=above/,scale=0.22,fill=black]  at (1,0.75) {};
\node(J)  [vrtx=above/,scale=0.22,fill=black]  at (1,1) {};

\path  (A) edge (B)
(C) edge (B)
(A) edge (C)
(B) edge (D)
(E) edge (D)
(E) edge (F)
(F) edge (G)
(E) edge (G)
(H) edge (I)
(J) edge (I)
(D) edge (H);
\draw (J) -- (I) node [right,midway] {$6$} ;
\end{tikzpicture}
\\ 
\hline
9 & 
\begin{tikzpicture}[
every edge/.style = {draw=black, },
 vrtx/.style args = {#1/#2}{%
      circle, draw,
      minimum size=1mm, label=#1:#2}
                    ]
\node(B)  [vrtx=above/,scale=0.22,fill=black]  at (0,0) {};
\node(C)  [vrtx=above/,scale=0.22,fill=black]  at (1,0) {};
\node(D)  [vrtx=above/,scale=0.22,fill=black]  at (0.5,0) {};
\node(E)  [vrtx=above/,scale=0.22,fill=black]  at (1.5,0) {};
\node(F)  [vrtx=above/,scale=0.22,fill=black]  at (2,0) {};
\node(G)  [vrtx=above/,scale=0.22,fill=black]  at (2.5,0) {};
\node(H)  [vrtx=above/,scale=0.22,fill=black]  at (3,0) {};
\node(A)  [vrtx=above/,scale=0.22,fill=black]  at (3.5,0) {};
\node(I)  [vrtx=above/,scale=0.22,fill=black]  at (3,-0.5) {};
\node(J)  [vrtx=above/,scale=0.22,fill=black]  at (3,-1) {};
\node(K)  [vrtx=above/,scale=0.22,fill=black]  at (4,0) {};

\path  (D) edge (B)
(C) edge (B)
(C) edge (D)
(E) edge (C)
(E) edge (F)
(F) edge (G)
(H) edge (G)
(K) edge (A)
(I) edge (H) 
(J) edge (I) 
(A) edge (H) ;

\tikzstyle{every node}=[font=\tiny]
\draw (B) -- (D)node [above,midway] {$6$} ;
\end{tikzpicture}
\quad
\begin{tikzpicture}[
every edge/.style = {draw=black},
 vrtx/.style args = {#1/#2}{%
      circle, draw,
      minimum size=1mm, label=#1:#2}
                    ]
\tikzstyle{every node}=[font=\tiny]
\node(A)  [vrtx=above/,scale=0.22,fill=black]  at (0,0) {};
\node(B)  [vrtx=above/,scale=0.22,fill=black]  at (0.5,0) {};
\node(C)  [vrtx=above/,scale=0.22,fill=black]  at (1,0) {};
\node(D)  [vrtx=above/,scale=0.22,fill=black]  at (0,0.5) {};
\node(E)  [vrtx=above/,scale=0.22,fill=black]  at (0.5,0.5) {};
\node(F)  [vrtx=above/,scale=0.22,fill=black]  at (1,0.5) {};
\node(G)  [vrtx=above/,scale=0.22,fill=black]  at (0,1) {};
\node(H)  [vrtx=above/,scale=0.22,fill=black]  at (0.5,1) {};
\node(I)  [vrtx=above/,scale=0.22,fill=black]  at (1,1) {};
\node(J)  [vrtx=above/,scale=0.22,fill=black]  at (0,1.5) {};
\node(K)  [vrtx=above/,scale=0.22,fill=black]  at (0.5,1.5) {};
\node(L)  [vrtx=above/,scale=0.22,fill=black]  at (1,1.5) {};
\node(y)  [vrtx=above/,scale=0.22,fill=black]  at (-1.5,0.75) {};
\node(x)  [vrtx=above/,scale=0.22,fill=black]  at (2.5,0.75) {};
\draw (A) -- (B) node[midway,above] {$6$}; 
\path (C) edge (B);
\draw  (D) -- (E)  node[midway,above] {$6$}; 
\path  (E) edge (F);  
\draw (G) -- (H) node[midway,above] {$6$}; 
\path (I) edge (H);
\draw (J) -- (K) node[midway,above] {$6$}; 
\path (L) edge (K)
 (y) edge (A)
 (y) edge (D)
 (y) edge (G)
 (y) edge (J)
 (x) edge (C)
 (x) edge (F)
 (x) edge (I)
 (x) edge (L);
\end{tikzpicture}
 \\
\hline
11 & 
\begin{tikzpicture}[
every edge/.style = {draw=black,},
 vrtx/.style args = {#1/#2}{%
      circle, draw,
      minimum size=1mm, label=#1:#2}
                    ]
\node(A) [vrtx=above/,scale=0.22,fill=black] at (0, 0) {};
\node(B) [vrtx=above/,scale=0.22,fill=black] at (0.5, 0) {};
\node(C) [vrtx=above/,scale=0.22,fill=black] at (1,0) {};
\node(D) [vrtx=above/,scale=0.22,fill=black] at (1.5,0) {};
\node(E) [vrtx=above/,scale=0.22,fill=black] at (2,0) {};
\node(F) [vrtx=above/,scale=0.22,fill=black] at (0.5,-0.5) {};
\node(G) [vrtx=above/,scale=0.22,fill=black] at (2.5, 0) {};
\node(L) [vrtx=above/,scale=0.22,fill=black] at (3.5, 0) {};
\node(K) [vrtx=above/,scale=0.22,fill=black] at (3, 0) {};
\node(H) [vrtx=above/,scale=0.22,fill=black] at (-0.5, 0) {};

\node(P) [vrtx=above/,scale=0.22,fill=black] at (4, 0) {};
\node(M) [vrtx=above/,scale=0.22,fill=black] at (4.5,0) {};
\node(N) [vrtx=above/,scale=0.22,fill=black] at (5, 0) {};

\path   (A) edge (B)
        (B) edge (C)
 (C) edge (D)
 (D) edge (E)
 (F) edge (B)

 (E) edge (G)
 (K) edge (G)
 (L) edge (K)
 (A) edge (H)
 (P) edge (M)
 (P) edge (N)
 (P) edge (L)
 (M) edge (N)
;
\tikzstyle{every node}=[font=\tiny]
\draw (M) -- (N)node [above,midway] {$6$} ;
\end{tikzpicture}
 \\
\hline 
\end{tabular}
\end{center}
\caption{The ADEG-polyhedra in $\mathbb H^n$}
\label{adeg}
\end{table}

\noindent
The proof of Theorem \ref{thmN9} is provided in Section \ref{Section4}. 

\hip  
Note that two ADEG doubly-truncated simplices in $\mathbb H^5$ 
first appeared in Im Hof's work   \cite{ImHof} and will be recovered by our constructive procedure. 

\hip 

The Coxeter polyhedron $P_{\star} \subset \mathbb H^{9}$   is a new discovery.   It has $14$ facets and  $134$ vertices with $6$ of them being ideal. 
More precisely, the software CoxIter \cite{Gug} provides the  following $f$-vector for $P_{\star}$ 
 \[f_{\star}=(134, 671, 1480, 1909, 1606, 917, 356, 91, 14) \ ,   \]
containing the number of faces of dimension $k$ for $k=0,\dots, n-1$. 

\hip 
Observe that all of its $G_2$-faces have the same combinatorial structure, which is of a pyramid over a product of three simplices of type $\widetilde G_2$; see Remark \ref{remG2}. More generally, the highly symmetric  nature the Coxeter diagram of $P_{\star}$  turns out to be very helpful in verifying the finite-volume criterion of Vinberg stated in Theorem \ref{thmVin2} for $P_{\star}$.

\begin{remark} \normalfont Apart from the two-dimensional case and one tetrahedron, all ADEG-polyhedra give rise to Coxeter groups which are arithmetic over $\mathbb Q$. 
The four cocompact ADEG-triangle groups are arithmetic as well,  but  their field of definition is $\mathbb Q(\sqrt3)$. 
\hip 
In  Section \ref{Section5}, we will discuss the fact that the non-cocompact Coxeter  group $\Gamma_{\star}=\Gamma(P_{\star})$ is commensurable to the ADE-group associated with $P_2\subset\mathbb H^9$,   as well as to all  Coxeter simplex groups and all  Coxeter pyramid groups in $\hbox{Isom}\mathbb H^{9}$.   We will derive 
that the volume of $P_{\star}$  satisfies 
\begin{equation}
\hbox{vol}(P_{\star})= q\cdot \frac{\zeta(5)}{22,295,347,200} \ \text{ with }\ q\in\mathbb Q_{>1} \ .   
\label{volPstarbis}  
\end{equation}
 \end{remark}

\section{Proof of the main theorem}
\label{Section4}

The proof of Theorem \ref{thmN9} relies on different ingredients that we provide in Section \ref{Section4.1}. 
The section concludes with our constructive procedure, outlined in Section \ref{Section4.1.4} . 
We apply this procedure inductively on the dimension  in Section \ref{Section4.2} to derive our main theorem. 
Some auxiliary data are collected in the Appendices \ref{AppendixA},  \ref{AppendixB} and \ref{AppendixC}.

\subsection{First observations, lemmas and strategy}
\label{Section4.1}

\hip 
Let $P\subset \mathbb H^n$  be an ADEG-polyhedron 
and denote by $\Sigma$ its Coxeter diagram.  

\subsubsection{The simple case} 
\label{Section4.1.1}

\hip 
Assume that $P$ is  simple.   From Theorem \ref{thmFT}, we deduce that $P$ is a simplex,  as Esselmann polyhedra and $P_0\subset\mathbb H^4$ are not ADEG-polyhedra; see \cite{Ess} and Figure \ref{figP0}. 

\hip  Observe that in the compact case, 
there are four ADEG-triangles in total,  and there are no compact ADEG-simplices for $n>2$.
 In the non-compact case,  there are seven ADEG-tetrahedra, and no ADEG-simplices for $n\neq 3$.  

\subsubsection{$G_2$-faces and existence of a subdiagram $\widetilde G_2$} 
\label{Section4.1.2}

From the previous section combined with 
 Corollary \ref{cor}, we assume from now on  that $P\subset \mathbb H^n$ is a non-simple ADEG-polyhedron of dimension $n\geq 5$.  

\hip
 By definition, $P$ admits at least one dihedral angle $\frac{\pi}{6}$. Therefore its Coxeter diagram $\Sigma$ 
contains at least one subdiagram $G_2=G_2^{(6)}$.  

\hip By Theorem \ref{thmAll},  the corresponding $G_2$-face $F$ of codimension $2$ of $P$  is a Coxeter polyhedron. 
More precisely,  in view of Theorem \ref{thmAll2} and  Remark \ref{remG2},  the face $F$ is in fact an ADE- or an ADEG-polyhedron in $\mathbb H^{n-2}$ and its Coxeter diagram $\sigma_F$  is a subdiagram of $\Sigma$.   

\hip
Let us point  out that $F$ is necessarily a non-compact polyhedron,  as ADE-polyhedra and ADEG-polyhedra in dimension $n-2\geq 3$ are all  non-compact; see  Section \ref{Section2.2} and    Section \ref{Section4.1.1}. 
We derive the following lemma. 

\begin{lemma}
Let $n\geq5$, and let $P\subset \mathbb H^n$ be an ADEG-polyhedron. 
Then,  the Coxeter diagram of $P$ contains at least one subdiagram of type $\widetilde G_2$.  
\label{G2}
\end{lemma}

\begin{proof} 
Let $n\geq 5$,  and let $P\subset \mathbb H^n$ be an ADEG-polyhedron with Coxeter diagram  $\Sigma$. Since $P$ is an ADEG-polyhedron, $\Sigma$ contains a subdiagram $\sigma=G_2=[6]$. 

\hip By Theorem \ref{thmAll}, $\sigma$ yields a $G_2$-face $F$ of  dimension $n-2$ which is an ADE- or an ADEG-polyhedron,  and the Coxeter diagram $\sigma_F$ of $F$ is a subdiagram of $\Sigma$ that is disjoint from $\sigma$.  
As $n-2\geq 3$, the face $F$ is necessarily non-compact. Therefore,  $\sigma_F$ contains an affine subdiagram $\sigma_F^{\infty}$ of rank $n-3$ which appears as a component in an affine diagram of rank $n-1$ in $\Sigma$. Now the complement $\Sigma\setminus\sigma_F$  contains $\sigma$,  therefore  the affine subdiagram  of $\Sigma\setminus\sigma_F^{\infty}$ is necessarily a  diagram of type $\widetilde G_2$. 

\end{proof}

By Lemma \ref{G2}, there exists an affine subdiagram $\sigma_{\infty}$ of rank $n-1$ in $\Sigma$ of the form 
\begin{equation}
\sigma_{\infty}=\sigma_1\cup \sigma_2 \cup \dots \cup \sigma_m = \widetilde G_2 \cup \sigma_2 \cup \dots \cup \sigma_m  \ , \ m \geq 2 \ .  
\label{sigmainfty}
\end{equation}
Denote by $r_i$ the rank of $\sigma_i$ so that  
$\sum_{i=1}^m r_i = n-1$. Note that $r_i\geq 2$ for all $i=1,\ldots, m$ as   $P$ does not admit any pair of disjoint facets. 

\hop 

In what follows, we consider $F$ to be the $G_2$-face associated with the subdiagram 
$\sigma=G_2$ of $\sigma_1=\widetilde G_2 \subset \sigma_{\infty}$. 
By what precedes, $F$ is an ADE- or ADEG-polyhedron of dimension $n-2$ and
 its Coxeter diagram $\sigma_F\subset \Sigma\setminus \sigma_1$ 
contains all the affine components $\sigma_2, \ldots, \sigma_m$.

\hip We emphasize that when $n\geq 20$, such a $G_2$-face $F$  is necessarily an ADEG-polyhedron, 
since ADE-polyhedra do not exist anymore in dimensions $\geq 18$; see Theorem \ref{thmPrk}.

\subsubsection{From a $G_2$-face to an admissible set of vectors} \label{Section4.2.3}
\label{Section4.1.3}

The  subdiagram $\sigma_{\infty}$ given in  \eqref{sigmainfty} corresponds to a non-simple ideal vertex $v_{\infty}$ of $P$, which we treat as a lightlike vector in $\mathbb R^{n,1}$.  
Furthermore,  each affine component $\sigma_i$  can be interpreted as the extended Dynkin diagram of a root system $R_i$ of type $A,  D,   E$ or $G$; see \cite{Bourbaki}.  In Appendix \ref{AppendixB}, we provide a short description of these root systems with relevant data for what follows.

\hip 
Let us denote  by $e_1^i, \dots,  e_{r_i}^i$ the natural simple roots of $R_i$ and  by $e_{r_i+1}^i$ the opposite of the highest root of  $R_i$,   with their prescribed lengths as given in Appendix \ref{AppendixB}. 
For instance, for the affine component $\sigma_1=\widetilde G_2$,  the root $e_1^1$, as indexed below,  is taken of square length $2$ and the roots   $ e_2^1$ and $e_3^1$ are taken of  square norm $6$; see Appendix \ref{AppendixB.4}.  

$$
\begin{tikzpicture}
[
every edge/.style = {draw=black},
every vrtx/.style = {fill=black},
 vrtx/.style args = {#1/#2}{%
      circle, draw,
      minimum size=1mm, label=#1:#2}
                    ]
\tikzstyle{every node}=[font=\scriptsize]
\node(A) [fill=black,vrtx=below/$e_{1}^1$,scale=0.22] at (0, 0) {};
\node(B) [fill=black,vrtx=below/$e_2^1$,scale=0.22] at (1,0) {};
\node(C) [fill=black,vrtx=below/$e_3^1$,scale=0.22] at (2,0) {};
\draw   (A) -- (B) node[midway,above] {$6$}; 
\path        (B) edge (C); 
\end{tikzpicture}
$$

\noindent For any affine component of type $A, D$ or $E$,  all roots are all taken of square length $2$; see Appendix \ref{AppendixB}. 

\hip 
Every vector $e_j^i$, for $i=1,\ldots,m, j=1,\ldots r_{i+1}$, 
 can be interpreted as  a spacelike vector normal to a  hyperplane, denoted by  $H_{e_j^i}$,  that passes through $v_{\infty}$.  
 The vertex $v_{\infty}$ appears therefore as the apex of a polyhedral cone $C_{\infty}\subset \mathbb H^n$ over a product of several simplices, each of dimension $\geq 2$, formed by the intersection of the  hyperplanes 
\[
 \{ H_{e_j^i} \  \mid    \  1\leq i \leq m \ ,  \ 1\leq j \leq r_i+1 \   \} \ . 
\] 
The cone $C_{\infty}$ has {\it infinite} volume in $\mathbb H^n$ and has to be truncated by additional hyperplanes  
 not passing through $v_{\infty}$  in order to yield the finite-volume polyhedron  $P\subset \mathbb H^n$. 

\hip 
The first additional hyperplane(s) that we will detect belong to the boundary of the $G_2$-face $F$ described above. 
This set will be completed by additional hyperplanes subject to an admissibility condition,  inspired from the work of Nikulin \cite{Nikulin2} and  Prokhorov \cite{Prk}; see 
 Definition \ref{admissible}.  We provide the technical details below.  

\hip 
Observe that $P$ is a pyramid with apex $v_{\infty}$ if and only if the boundary of $P$ is obtained by adding a {\it unique} additional hyperplane to the set of hyperplanes $\{H_{e_j^i}\}_{i,j}$. 
As all pyramids are classified (see \cite{Tum3} and \cite{Mcleod}), we discuss in what follows the case where $P$ is not a pyramid. 
Note however that we will recover all ADEG-pyramid through our method; see item 4.i. of Section \ref{Section4.1.4}. 
Therefore, we assume next that the set $\{H_{e_j^i}\}_{i,j}$ 
has to be completed  by at least two additional hyperplanes that do not pass through the vertex $v_{\infty}$.

\hip
As $v_{\infty}$ is Lorentz-orthogonal to all vectors $e_j^i$, and it can be expressed in terms of certain lightlike vectors $v_{\infty}^i$ related to the affine components  $\sigma_i$ for $i=1,\dots m$ as follows. 

\begin{lemma}
For each $1 \leq i \leq m$,   there exist positive integers $c^i_1,\dots, c^i_{r_i+1}$ with $c^i_{r_i+1}=1$ such that the vector $v_{\infty}^i:=c_1^ie_1^i+\dots + c_{r_i+1}^i e_{r_i+1}^i \in \mathbb R^{n,1}$ is collinear with  $v_{\infty}$.  
\label{vi}
\end{lemma}

\begin{proof}
Let $c^i_1,\dots,c^i_{r_i}\in \mathbb Z$  be the integers such that 
\begin{equation}
-e_{r_i+1}^i  = c_1^ie_1^i+\ldots+ c_{r_i}^i e_{r_i}^i \ \in \mathbb R^{r_i} \ ;
\label{ci}
\end{equation}
see Appendix  \ref{AppendixB} for their explicit form.    
The vector $v_{\infty}^i=c_1^ie_1^i+\ldots + c_{r_i}^i e_{r_i}^i + e_{r_i+1}^i$ (suitably extended to a non-zero vector in  $\mathbb R^{n,1}$) satisfies  $\langle v_{\infty}^i, v_{\infty}^i \rangle=0$ by properties of the highest root. 
 In addition,  $\langle v_{\infty}^i, v_{\infty} \rangle=0$ as $\langle e_j^i,  v_{\infty}\rangle =0$ for all $j=1, \ldots, r_i+1$. 
As any two  Lorentz-orthogonal lightlike vectors are collinear, this finishes the proof. 
\end{proof}

\hip As a consequence of Lemma \ref{vi}, 
  the vectors $v_{\infty}$ and  $v_{\infty}^i$  
define the same lightlike ray, which we denote by  $ v_{\infty} \sim    v_{\infty}^i$. 
 In particular, for $i=1$, we derive that 
\begin{equation}
 v_{\infty} \sim v_{\infty}^1=3e_1^1+2e_2^1+e_3^1  \ ; 
 \label{v1}
 \end{equation}
see Appendix \ref{AppendixB.4}.

\hop 

Assume that $H_x$ and $H_y$ are two  bounding hyperplanes of $P$  not passing through $v_{\infty}$. 
These hyperplanes are characterized by two spacelike normal vectors $x, y \in \mathbb R^{n,1}$  such that  
 $\langle x,v_{\infty} \rangle \neq 0$ and $\langle y ,v_{\infty}\rangle \neq 0$.   
In this setting, the Lorentzian product $\langle x,y \rangle$ can be expressed in terms  of the Lorentzian products  $k_j^i:=\langle x, e_j^i \rangle$ and $ l_j^i:= \langle y,  e_j^i \rangle$  with the vectors $e_j^i$ for $i=1,\dots,m$, $j=1,\dots,r_i+1$; see \cite{Prk}.  

\hip 
We consider  $x$ and $y$ of squared  norm $2$, and encode them by the following strings
\begin{equation}
\begin{split}
x&\leftrightarrow (k_1^1,k_2^1,k_3^1; k_1^2,\ldots, k_{r_2+1}^2;\dots;k_1^m, \dots, k_{r_m+1}^m ) \\
y&\leftrightarrow (l_1^1, l_2^1, l_{3}^1;l_1^2,\dots ,  l_{r_2+1}^2 ; \dots; l_1^m, \dots, l_{r_m+1}^m)   
\end{split} 
\label{strings}
\end{equation}
where $k_j^i=\langle x, e_j^i \rangle$  and $l_j^i=\langle y, e_j^i \rangle$ for $i=1,\ldots, m$ and $j=1,\ldots,r_i+1$.

\hop 
Since the dihedral angles $\measuredangle (H_x, H_{e_j^i})$  all vary within the set $\{\frac{\pi}{2}, \frac{\pi}{3}, \frac{\pi}{6}\}$,  the quantities  $k_j^i=\langle x, e_j^i \rangle$  are as follows.  
Similar expressions hold for $l_j^i=\langle y, e_j^i \rangle$.

\begin{equation}
\langle x, e_j^i \rangle = \left\{ 
\begin{array}{ll}
0 & \text{ if } \measuredangle (H_x,  H_{e_j^i}) = \frac{\pi}{2}\\ 
-1 & \text{ if } \measuredangle (H_x,  H_{e_j^i}) = \frac{\pi}{3} \text{ and } ||e_j^i||^2=2\\ 
-\sqrt 3  & \text{ if } \measuredangle ( H_x,  H_{e_j^i}) = \frac{\pi}{6} \text{ and } ||e_j^i||^2=2\\ 
-\sqrt 3  & \text{ if } \measuredangle (H_x,  H_{e_j^i}) = \frac{\pi}{3} \text{ and } ||e_j^i||^2=6\\ 
-3  & \text{ if } \measuredangle ( H_x,  H_{e_j^i}) = \frac{\pi}{6} \text{ and } ||e_j^i||^2=6\\  
\end{array}
\right.  
\label{length}
\end{equation}

\hop Consider  the quantities $\Lambda$ and $\Lambda_i$ for $1\leq i \leq m$ defined as follows
\[
\Lambda= \Lambda_{xy} := \frac{\langle x,v_{\infty}\rangle}{\langle y,v_{\infty}\rangle} \ \text{ and } \ \Lambda_i := 
\frac{\langle x, v_{\infty}^i \rangle}{\langle y, v_{\infty}^i \rangle}=\frac{\langle x, \sum_{k=1}^{r_i+1} c_k^ie_k^i\rangle }{\langle y, \sum_{k=1}^{r_i+1} c_k^ie_k^i\rangle } \ , 
\] 
where the lightlike vectors $v_{\infty}^i$ are defined in Lemma \ref{vi}; see also  \cite{Prk}.  
In particular, in view of  the fact that $\sigma_1=\widetilde G_2$, one has  
\begin{equation}
\Lambda_1  =  \frac{3k_1^1+2k_2^1+k_3^1}{3l_1^1+2l_2^1+l_3^1} \ . 
\label{req1}
\end{equation}

\hip 
From  Lemma \ref{vi}, one immediately deduces that     
\begin{equation}
\Lambda_i =\Lambda \ \text{ for all }  i=1,\ldots,m \ . 
\label{lambda}
\end{equation}

\noindent
We are now able to write down Prokhorov's formula \cite{Prk} for the Lorentzian product $\langle x, y \rangle$.  Note that the signs are changed according to our setting.

\begin{lemma}[Prokhorov \cite{Prk}]
\begin{equation}
\langle x , y\rangle =\frac{\langle x,x \rangle}{2\Lambda} +\frac{\langle y,y \rangle}{2}\Lambda  - (\Delta_1+\ldots+\Delta_m) \ ,
\label{xy}
\end{equation}
where, 
 for $p\in \{1, \ldots, m\}$,  
\begin{equation}
\Delta_p = \sum\limits_{1\leq i<j\leq r_p+1}(k_i^pl_j^p+k_j^pl_i^p-\frac{k_i^pk_j^p}{\Lambda} - l_i^pl_j^p\Lambda) s_{ij}^p \ . 
\label{Delta}
\end{equation} 
\noindent
The quantities $s_{ij}^p$  in \eqref{Delta} depend on the fundamental weights $w_1,\dots,w_{r_p}$  of the root system $R_p$ as follows

 \begin{equation}\hspace{-1em}\left\{\begin{array}{lll}s_{ij}^p=&\langle w_i,w_j\rangle-\dfrac{\langle  w_i,w_i\rangle c_j^p}{2c_i^p} - \dfrac{\langle w_j,w_j\rangle c_i^p}{2c_j^p} \quad &\text{ for } 1\leq i, j\leq r_p \ ,  \\  s_{ir_p}^p=&\dfrac{\langle  w_i,w_i\rangle}{2c_i^p}  \quad &\text{ for } i=1, \dots, r_p \ , \\s_{ii}^p =&0  \quad &\text{ for } i=1, \dots, r_p+1 \ ,  \end{array}\right.
 \label{sisj}\end{equation} 
where 
 $c_j^p$ are the integers according to \eqref{ci}. 
\label{Prkformula}
\end{lemma}

\hop  The quantities $s_{ij}=s_{ij}^p$   defined in \eqref{sisj}  
 for a root system $R=R_p$ of type $A_k,  D_k,  E_k$ or $ G_2$  
are given in their explicit form in Appendix \ref{AppendixB}.

\hop 
Since $P$ is an ADEG-polyhedron, we define the following admissibility condition for the pair of vectors $\{x,y\}$, which is necessary for the hyperplanes $H_x$ and $H_y$ to intersect  in a dihedral angle $\frac{\pi}{2}, \frac{\pi}{3}$ or $\frac{\pi}{6}$; see \eqref{lambda} and Lemma \ref{Prkformula}.

\begin{definition}[Admissibility condition] 
For $x, y$ of squared length $2$, the pair of vector $\{x,y\}$  is said to be a    {\it admissible } if 
\begin{equation}
\Lambda_1  =  \frac{3k_1^1+2k_2^1+k_3^1}{3l_1^1+2l_2^1+l_3^1}= \frac{\sum_j c_j^2k_j^2}{\sum_j c_j^2l_j^2}
= \dots = \frac{\sum_j c_j^mk_j^m}{\sum_j c_j^ml_j^m} =\Lambda_m \, , 
\label{req1}
\end{equation}  
and  
\begin{equation}
\langle x,y\rangle=\Lambda+\frac1{\Lambda} - (\Delta_1 +\ldots + \Delta_m)  \in \{0,-1,-\sqrt3\} \ ,  
\label{req2}
\end{equation}
\hip 
where $\Lambda$ and $\Delta_1,  \ldots, \Delta_m$ are given according to \eqref{lambda} and  \eqref{Delta}.  
\label{admissible}
\end{definition}

 \hip 
More generally, we will say that a set  of vectors $\{z_1,  \ldots, z_t\}$ is   {\it admissible} if its elements are pairwise admissible,  that is,  $\{z_i,z_j\}$ is an admissible pair for all $i, j\in\{1,\ldots, t\}$.

\subsubsection{The strategy}
\label{Section4.1.4}

Considering all the above, we will work inductively on the dimension $n\geq 5$ using the constructive procedure described below. We  take into account the knowledge of all ADE-polyhedra, and of all ADEG-polyhedra of dimensions $\leq 4$; see Section \ref{Section4.1.1}.

\begin{enumerate}
\item 
Fix an affine subdiagram $\sigma_{\infty}=\sigma_1\cup\ldots\cup\sigma_m$ of rank $n-1$ such that $\sigma_1=\widetilde G_2$; see Lemma \ref{G2}. The diagram $\sigma_{\infty}$ corresponds to a  non-simple vertex $v_{\infty}$ of a polyhedral candidate $P\subset \mathbb H^n$. 

 \item Fix an ADE- or ADEG-polyhedron $F\subset \mathbb H^{n-2}$ with Coxeter diagram  $\sigma_F$,    to play the role of the  $\sigma$-face $F$ of $P$, for $\sigma:=G_2\subset\sigma_1$.  

\noindent
 If there does not exist an ADE- or ADEG-polyhedron in $\mathbb H^{n-2}$ whose Coxeter diagram contains the affine components  $\sigma_2, \ldots, \sigma_m$, discard this case  and consider a different affine subdiagram $\sigma_{\infty}$ of rank $n-1$; see Section \ref{Section4.1.2}. 

\item Denote by  $x_1,\dots,x_t$ the nodes which,  together with $\sigma_2 \cup \dots \cup \sigma_m$,  constitute the diagram $\sigma_F$.  For simplicity, denote by   $x_1, \ldots, x_t$
 the spacelike vectors,  chosen of squared norm $2$, normal to the corresponding bounding hyperplanes  $H_{x_1},\ldots, H_{x_t}$ of $P$.    

\hip
Observe that if $F$ is a simplex (in which case $m=2$)  or a pyramid (in which case $m>2$),   
one necessarily has $t=1$, and  one has $t\geq 2$ otherwise.

\hip  Each node $x_1, \ldots, x_t$ is encoded by a string of the form 
\begin{equation}
x_i \leftrightarrow (0,0,a_i; k_1^2,\dots k_{r_2+1}^2;\dots;k_1^m, \dots, k_{r_m+1}^m ) \ , 
\label{stringx}
\end{equation}
where we abbreviate  $a_i:=k_3^1= \langle x_i, e_3^1 \rangle  \in \{-\sqrt3, -3\}$; see  \eqref{length}. 
Note that one has  $\langle x_i, e_1^1\rangle = \langle x_i, e_2^1\rangle=0$ by construction, and that 
$a_i=\langle x_i, e_1^1\rangle \neq 0 $ as the complement of $\sigma_F$ in $\Sigma$ is not equal to $\sigma$ otherwise.  
 All remaining coefficients in \eqref{stringx} are fixed by the choice of $F$.  

\item    Assume that $t=1$. 
\begin{enumerate}[i.]
\item 
If for a fixed value of $a:=a_1\in \{-\sqrt3, -3\}$, the hyperplane $H_{x_1}$ combined with the hyperplanes 
$\{H_{e_j^i}\}_{i,j}$ 
bound a finite-volume polyhedron in $\mathbb H^n$,  
then $P$ is necessarily a {\it pyramid} with apex $v_{\infty}$.  
One recovers all ADEG-pyramids in this way.  
\item 
If $P$ is not a pyramid, there is at least one additional hyperplane $H_y=H_{y_1}$ that does not  pass through $v_{\infty}$.
Let $y$ be of squared norm $2$,   set $l_j^i=\langle y, e_j^i \rangle$, and write 
\vspace{-0.3em}
\begin{equation}
y\leftrightarrow (l_1^1, l_2^1, l_{3}^1;l_1^2,\dots ,  l_{r_2+1}^2 ; \dots; l_1^m, \dots, l_{r_m+1}^m)   \ . 
\label{stringy}
\end{equation}
 The corresponding node in $\Sigma$ is a {\it bad} neighbour of $G_2\subset\sigma_1$,  therefore  $l_1^1$ and $l_2^1$ cannot be simultaneously zero.  
\end{enumerate}

\item 
 Assume that $t>1$. 
\begin{enumerate}[i.]
\item The values $a_1,\ldots, a_t$ are completely determined by \eqref{lambda}, since the set of vectors $\{x_1, \ldots, x_s\}$ must be admissible; see Definition \ref{admissible}. 

\hip 
If the hyperplanes $H_{x_1}, \dots, H_{x_t}$ combined with the hyperplanes $\{H_{e_j^i}\}_{i,j}$ 
bound a finite-volume polyhedron in $\mathbb H^n$,  
we have identified all bounding hyperplanes of an ADEG-polyhedron $P$, and its combinatorial-metrical structure is explicit. For this verification, see Theorem \ref{thmVin2} and   Remark \ref{obstructions}. 
\item 
If not, there is at least one additional hyperplane $H_y=H_{y_1}$ that does not  pass through $v_{\infty}$, which can be  encoded as in \eqref{stringy}. 

\end{enumerate}

\item
By means of Prokhorov's formula stated in Lemma \ref{Prkformula},  and with the help of the software Mathematica \cite{Mathematica},  
we search for the set of strings 
\[
\mathcal S_y:=\{ y\leftrightarrow (l_1^1, l_2^1, l_{3}^1;l_1^2,\dots ,  l_{r_2+1}^2 ; \dots; l_1^m, \dots, l_{r_m+1}^m) \  \mid \ \{x_1,\dots, x_t, y\} \text{ admissible} \} \ . 
\]

\item  
For every pairwise admissible vectors $y_1, \dots, y_q \in \mathcal S_y$, 
 we verify if the arrangement of hyperplanes resulting of the admissible set  $\{x_1,\dots,x_t, y_1, \ldots, y_q\}$ bounds a finite-volume polyhedron in $\mathbb H^n$; see Theorem \ref{thmVin2} and  Remark \ref{obstructions}.     
\end{enumerate}

\begin{remark}\normalfont
In practice, the following {\it obstructions}  to the hyperbolicity or the finite-volume of the polyhedral candidates $P$ with Coxeter diagram $\Sigma$  are used repeatedly. 
\begin{enumerate}[i.]
\item $\Sigma$ contains a spherical subdiagram of rank $>n$. 
\item  $\Sigma$ contains a superhyperbolic subdiagram, that is, the signature of its Gram matrix  is $(k,p)$ with $p>1$. 
\item   $\Sigma$ contains an affine subdiagram that is not a component of an affine subdiagram of rank $n-1$.
\item  The set of good neigbours of a subdiagram of type $G_2$ does not constitute the diagram of a finite-volume Coxeter polyhedron in $\mathbb H^{n-2}$; see Theorem \ref{thmAll}. 
\item  Any obstruction to Vinberg's finite-volume criterion stated in Theorem \ref{thmVin2}. 
\end{enumerate}
\label{obstructions}
\end{remark}

\noindent  Observe that when starting the procedure for $n=5$, there are only finitely many possible $G_2$-faces, since ADE-polyhedra exist in finite number, which is also the case of ADEG-polyhedra of dimensions $\leq 4$; see Table \ref{adeg} and Section \ref{Section4.1.1}. Therefore, the constructive procedure  described above yields all but finitely many polyhedral candidates in all dimensions $n\geq 5$, as there are finitely many affine subdiagrams of rank $n-1$, finitely many possible $G_2$-faces, and finitely many possible strings of type \eqref{stringy} to constitue admissible sets of vectors.

\hip 
 A priori,  the procedure has to be performed at least up to dimension $n=19$,  
as there exists an ADE-pyramid in $\mathbb H^{17}$ which which could serve as a $G_2$-face of an ADEG-polyhedron in $\mathbb H^{19}$. 
The  procedure will finish in dimension $n\geq 20$, 
 once in both $\mathbb H^{n-1}$ and $\mathbb H^{n-2}$,  there exist no ADEG-polyhedra.

\subsection{Proof of the main Theorem}  
\label{Section4.2}

\hop We apply the inductive procedure described in Section \ref{Section4.1.4}.  

\hop 
{\bf  Notation.}  From now on and for better readability, we replace $k_j^i$ by its opposite and write $k_j^i =-\langle x, e_j^i \rangle$. Similarly, we denote $l_j^i =-\langle y, e_j^i \rangle$ and $a_i=-\langle x_i, e_1^1\rangle$.  

\subsubsection{Assume that  $\mathbf{n=5}$}

By Lemma \ref{G2}, $\Sigma$ contains an affine subdiagram  $\sigma_{\infty}=\sigma_1\cup\sigma_2$ of rank $4$  %
that is  either   $\widetilde G_2 \cup \widetilde G_2$  or   $\widetilde G_2 \cup \widetilde A_2$. 
The $G_2$-face $F$ of dimension $3$  is  a finite-volume Coxeter polyhedron whose Coxeter diagram $\sigma_F$ contains $\sigma_2$.   
Therefore, $F$ is one the ten ADE- or ADEG-tetrahedra in $\mathbb H^3$; see Section \ref{Section4.1.1},  Table \ref{adeg} and Table \ref{ADE}. 

\hip Since $F$ is a simplex,  a unique node $x$ is added to $\sigma_2$ to constitue $\sigma_F$. 
Let $x\in \mathbb R^{5,1}$ be the corresponding vector,  $||x||^2=2$, so that $\{H_x,H_{e_1^2},H_{e_2^2}, H_{e_3^2}\}$ constitute the boundary of  $F$. We write 
\begin{equation}
x \leftrightarrow (0,0,a; k_1^2, k_2^2, k_3^2), 
\label{string5}
\end{equation}
where  $a=-\langle x, e_3^1\rangle \in\{\sqrt3,3\}$ and $k_j^2=- \langle x, e_i^2\rangle\in\{0,1,\sqrt3,3\}$.

\hip 
Table \ref{cases5} contains all strings to be considered for the vector $x$,  written as in \eqref{string5}, for all possible $G_2$-faces.  
Notice that in some cases,  the coefficient $a=\sqrt3$ can be excluded.   
 The reason will become clear in part (i) below; see Remark \ref{nopyr}.
We also omit some cases that lead to symmetric configurations.

\begin{table}[!h]
\footnotesize
\begin{center}
\bigskip
\tabcolsep=12pt%
\renewcommand*{\arraystretch}{1.4} 
\begin{tabular}{|c|c|| c| c |}
\hline 
$\sigma_{\infty}$  &  String  & $\sigma_{\infty}$  & String \\ 
\hline
\hline
$\widetilde G_2 \cup \widetilde G_2$ &  $(0,0,3 ;0,0,\sqrt3)$  & $\widetilde G_2 \cup \widetilde A_2$ &  $(0,0,3;1,0,0)$\\
& $(0,0,a ;1,0,0)$ & &  $(0,0,a ;1,0,1)$  \\
&  $(0,0,3 ;0,0,3)$ & &  $(0,0,a ;1,1,1)$ \\
&  $(0,0,a ;0,\sqrt3,0)$ & &  $(0,0,a ;\sqrt3,0,0)$   \\
&  $(0,0,a ;1,0,\sqrt3)$& &   \\
&  $(0,0,a ;1,0,3)$  & & \\
\hline
\end{tabular}
\end{center}
\caption{Strings for dimension $n=5$}
\label{cases5}
\end{table}

\noindent
$\diamond \ \ $  Suppose that   $\sigma_{\infty}=  \widetilde G_2\cup \widetilde G_2$.  Then,    $F$ is one of  the seven ADEG-tetrahedra in $\mathbb H^3$. 
In what follows,  we describe  the two most  important cases, only.  For all other cases,  we give just a few details but display all the admissible pairs in  Table \ref{admiss5}.

\hop 
(i) Assume that $F$ is the Coxeter tetrahedron 
with Coxeter symbol $[3,3,6]$; see the following diagram. 
$$
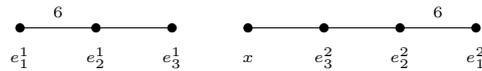
\begin{tikzpicture}[
every edge/.style = {draw=black},
every vrtx/.style = {fill=black},
 vrtx/.style args = {#1/#2}{%
      circle, draw,
      minimum size=1mm, label=#1:#2}
                    ]
\tikzstyle{every node}=[font=\tiny]
\node(A) [fill=black,vrtx=below/$e_{1}^1$,scale=0.3] at (0, 0) {};
\node(B) [fill=black,vrtx=below/$e_2^1$,scale=0.3] at (1,0) {};
\node(C) [fill=black,vrtx=below/$e_3^1$,scale=0.3] at (2,0) {};
\node(x) [fill=black,vrtx=below/$x$,scale=0.3] at (3,0) {};
\node(D) [fill=black,vrtx=below/$e_3^2$,scale=0.3] at (4,0) {};
\node(E) [fill=black,vrtx=below/$e_2^2$,scale=0.3] at (5,0) {};
\node(F) [fill=black,vrtx=below/$e_1^2$,scale=0.3] at (6,0) {};
\draw   (A) -- (B) node[midway,above] {6}
       (B) -- (C)
       (D) -- (E)
       (D) -- (x)
 (E) -- (F) node[midway,above] {6};
\end{tikzpicture}
$$

\hip  
We write  
$x \leftrightarrow (0,0, a;0,0,\sqrt3)$ for $a\in \{\sqrt3,3\}$.  

\hop 
A vector $y \leftrightarrow (l_1^1, l_2^1, l_{3}^1; l_1^2, l_{1}^2 , l_3^2)$,  
 $||y||^2=2$,   forms an  admissible pair with $x$ if  
\begin{equation*}
\Lambda_1  =  
\frac{ a}{3l_1^1+2l_2^1+l_3^1}=\frac{\sqrt3}{3l_1^2+3l_2^2+l_3^2} =\Lambda_2 = \Lambda \ ,  
\end{equation*}
\noindent  and  if $
\langle x,y\rangle=\Lambda+\frac1{\Lambda} - (\Delta_1 + \Delta_2)  \in \{0,-1,-\sqrt3\}$
with  $\Delta_1$ and  $\Delta_2$ as in \eqref{Delta}.

\begin{remark}\normalfont 
For $a=\sqrt3$,  the resulting Coxeter diagram corresponds to a  $5$-pyramid of finite-volume.  
In view of Proposition \ref{propsubdiagram}, we can now assume that $a=3$.  
\label{nopyr} 
\end{remark}

\noindent   
As a consequence of Remark \ref{nopyr},  let us assume that $a=3$.  
There is a unique solution so that $\{x,y\}$ is admissible,  and it is  given by 

\[
y\leftrightarrow(l_1^1,l_2^1,l_3^1;l_1^2,l_2^2,l_3^2) = (\sqrt3,0,0; 1,0,0)
\]
for which 
$\Lambda=\frac{1}{\sqrt3}$,  
$\Delta_1=\sqrt3$ and $\Delta_2=\frac{1}{\sqrt3}$, and  
$\langle x,y\rangle=0$. 

\hip 
This configuration corresponds to the Coxeter diagram depicted below 
wherein the red nodes correspond to the vectors  $x$ and $y$.

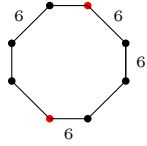
$$
\begin{tikzpicture}
\tikzstyle{every node}=[font=\tiny]
\fill[black] (0,0) circle (0.05cm);
\fill[black] (1/2,1/2) circle (0.05cm);
\fill[red] (1,1/2) circle (0.05cm);
\fill[black] (1.5,0) circle (0.05cm);
\fill[black] (1.5,-1/2) circle (0.05cm);
\fill[black] (0,-1/2) circle (0.05cm);
\fill[red] (1/2,-1) circle (0.05cm);
\fill[black] (1,-1) circle (0.05cm);

\draw (0,0) -- (1/2,1/2) node [midway,yshift=0.3em,xshift=-0.4em] {$6$} ;
\draw (1,1/2) -- (1/2,1/2);
\draw (1.5,0) -- (1,1/2) node [midway,xshift=0.4em,yshift=0.3em] {$6$} ;
\draw (1.5,0) -- (1.5,-1/2)node [right,midway] {$6$} ;
\draw (1.5,0) -- (1.5,-1/2);
\draw (1,-1) -- (1.5,-1/2) ;
\draw (1,-1) -- (1/2,-1)node [below,midway] {$6$} ;
\draw (1/2,-1) -- (0,-1/2);
\draw (0,0) -- (0,-1/2);
\end{tikzpicture}$$

\hip This diagram corresponds to a finite-volume Coxeter polyhedron in $\mathbb H^5$; see Theorem \ref{thmVin2} and  \cite{ImHof}.

\hop
(ii) Assume that  $F$  is the tetrahedron with  Coxeter symbol $[3,6,3]$.  
The situation can be described with the following diagram. 

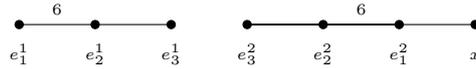
\begin{figure}[!h]\centering\begin{tikzpicture}[every edge/.style = {draw=black},every vrtx/.style = {fill=black}, vrtx/.style args = {#1/#2}{       circle, draw,      minimum size=1mm, label=#1:#2}                    ]
\tikzstyle{every node}=[font=\tiny]
\node(A) [fill=black,vrtx=below/$e_{1}^1$,scale=0.3] at (0, 0) {};\node(B) [fill=black,vrtx=below/$e_2^1$,scale=0.3] at (1,0) {};\node(C) [fill=black,vrtx=below/$e_3^1$,scale=0.3] at (2,0) {};\node(D) [fill=black,vrtx=below/$e_3^2$,scale=0.3] at (3,0) {};\node(E) [fill=black,vrtx=below/$e_2^2$,scale=0.3] at (4,0) {};\node(F) [fill=black,vrtx=below/$e_1^2$,scale=0.3] at (5,0) {};\node(x) [fill=black,vrtx=below/$x$,scale=0.3] at (6,0) {};\tikzstyle{every node}=[font=\tiny]\draw   (A) -- (B) node[midway,above] {6}       (B) -- (C)     (D) -- (E)     (D) -- (x)(E) -- (F) node[midway,above] {6};\end{tikzpicture}\end{figure}

\hop
For $x\leftrightarrow(0,0, \sqrt3;1,0,0)$,  there is a unique vector  $y$ such that the pair  $\{x,y\}$ is  admissible.  It is given by $y\leftrightarrow(0, \sqrt3,\sqrt3;1,0,0)$ and yields $\langle x,y\rangle=
0$. 
This configuration, depicted below, does not correspond to  a finite-volume $5$-polyhedron.   

$$
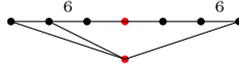
\begin{tikzpicture}\tikzstyle{every node}=[font=\tiny]\fill[black] (0,0) circle (0.05cm);\fill[black] (1/2,0) circle (0.05cm);\fill[black] (1,0) circle (0.05cm);\fill[red] (3/2,0) circle (0.05cm);\fill[red] (3/2,-1/2) circle (0.05cm);\fill[black] (4/2,0) circle (0.05cm);\fill[black] (5/2,0) circle (0.05cm);\fill[black] (6/2,0) circle (0.05cm);\draw (0,0) -- (1/2,0) ;\draw (1,0) -- (1/2,0) node [above,midway] {$6$} ;\draw (4/2,0) -- (3/2,0) node [above,midway] {$$} ;\draw (3/2,0) -- (2/2,0) node [above,midway] {$$} ;\draw (4/2,0) -- (5/2,0);\draw (6/2,0) -- (5/2,0) node [above,midway] {$6$} ;\draw (3/2,-1/2) -- (1/2,0);\draw (3/2,-1/2) -- (6/2,0);\draw (3/2,-1/2) -- (0,0);\end{tikzpicture}$$

\hip  
In fact, this can be easily deduced from  the Coxeter diagram as  it  contains an affine subdiagram $\widetilde A_2$ which cannot be completed to yield an affine subdiagram of rank $4$.  This contradicts Theorem \ref{thmVin2}.

\hop 
For $ x\leftrightarrow(0,0,3;1,0,0)$, there is again a unique solution  $y\leftrightarrow(1,0,0;0,0,3)$ which satisfies   $\langle x,y\rangle=0$. This  gives rise to the finite-volume Coxeter polyhedron  depicted below. 

$$
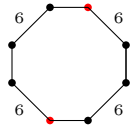
\begin{tikzpicture}\tikzstyle{every node}=[font=\tiny]
\fill[black] (0,0) circle (0.05cm);
\fill[black] (1/2,1/2) circle (0.05cm);
\fill[red] (1,1/2) circle (0.05cm);
\fill[black] (1.5,0) circle (0.05cm);
\fill[black] (1.5,-1/2) circle (0.05cm);
\fill[black] (0,-1/2) circle (0.05cm);
\fill[red] (1/2,-1) circle (0.05cm);
\fill[black] (1,-1) circle (0.05cm);
\draw (0,0) -- (1/2,1/2) node [midway,yshift=0.3em,xshift=-0.4em] {$6$} ;\draw (1,1/2) -- (1/2,1/2);\draw (1.5,0) -- (1,1/2) node [midway,yshift=0.3em,xshift=0.4em] {$6$} ;\draw (1.5,0) -- (1.5,-1/2);\draw (1.5,0) -- (1.5,-1/2);\draw (1,-1) -- (1.5,-1/2) node [midway,yshift=-0.3em,xshift=0.4em] {$6$} ;\draw (1,-1) -- (1/2,-1);\draw (1/2,-1) -- (0,-1/2) node [midway,yshift=-0.3em,xshift=-0.4em] {$6$} ;\draw (0,0) -- (0,-1/2);\end{tikzpicture}$$

\hop
For all of the remaining cases, the admissible pairs are listed in Table \ref{admiss5}. However,  none of them gives rise to a $5$-polyhedron of finite volume.  
Furthermore, there exists no admissible set of cardinality $>2$.

\begin{table}[!h]
\footnotesize
\begin{center}
\bigskip
\tabcolsep=12pt%
\renewcommand*{\arraystretch}{1.3} 
\begin{tabular}{|c|c|c|c|}
\hline
$\sigma_{\infty}$  & $x$ & $y$ & $\langle x, y \rangle $ \\ 
\hline
\hline
$\widetilde G_2 \cup \widetilde G_2$ & $(0,0, 3;0,0,\sqrt3)$ & $(\sqrt3,0,0;1,0,0)$ & $0$  \\
&$(0,0, \sqrt3;1,0,0)$ & $(1,0,0;0, \sqrt3, \sqrt3)$ & $0$ \\
&$(0,0,3;1,0,0)$ & $(1,0,0;0,0,3)$ & $0$ \\
&$(0,0,3;0,0,3)$ & $(1,0,0;1,0,0)$ & $0$ \\
&$(0,0,3;0,\sqrt3,\sqrt3)$ & $(\sqrt3,0,0;1,3,0)$ & $0$ \\ 
\hline 
\hline 
$\widetilde G_2 \cup \widetilde A_2$ & $(0,0,3;1,0,0)$  & $(\sqrt3,0,0;0,0,\sqrt3)$ & 0 \\
&$(0,0,3;\sqrt3,0,0)$  & $(\sqrt3,0,0;0,0,\sqrt3)$  &  0 \\ 
&$(0,0,3;\sqrt3,0,0)$  &   $(\sqrt3,0,0;0,\sqrt3,0)$ &  0 \\ 
\hline
\end{tabular}
\end{center}
\caption{Admissible pairs $\{x,y\}$ for $n=5$}
\label{admiss5}
\end{table}

\hop
$\diamond \ \ $  In the case where $\sigma_{\infty}=\widetilde G_2\cup \widetilde A_2$,  none of the   admissible pairs yields a finite-volume polyhedron; see Table \ref{admiss5}. 
In addition, we find one admissible set formed by three vectors, encoded by the strings  
$
 \{(0,0,3;1,0,0), (\sqrt3,0,0;0,0,\sqrt3), (0,0,3;\sqrt3,0,0) \}$. 
However,  this configuration does not yield a finite-volume hyperbolic polyhedron. 

\hip  This finishes the proof for $n=5$.

\subsubsection{Assume that $\mathbf{n=6}$}

By means of Lemma \ref{G2},  there is a unique case to consider given by  $\sigma_{\infty}=\widetilde G_2\cup\widetilde A_3$.  
Hence, 
$F$ is one of the two ADE-simplices in $\mathbb H^4$ as there is no ADEG-polyhedra in $\mathbb H^4$; see Section \ref{Section4.1.1} and  Table \ref{ADE}. 

\hop 
Considering Remark \ref{nopyr},  we  assume that  $x\leftrightarrow(0,0,3;1,0,0,0)$ or     $x\leftrightarrow(0,0,a;1,0,1,0)$ for $a\in \{\sqrt3,3\}$.     In both cases,  we find that there is no admissible pair.  
 Hence, apart from a single pyramid,  there are no further  ADEG-polyhedra in $\mathbb H^6$ .

\subsubsection{Assume that  $\mathbf{n=7}$}

From the previous sections, there are eight possibilites for $F$. Namely, there are three ADE-simplices,  one ADE-pyramid,  two ADEG-pyramids and two ADEG-doubly-truncated simplices in  $\mathbb H^5$; see Tables \ref{adeg} and \ref{ADE}.  

\hip 
If  $F$ is a simplex or a pyramid, there is only one vector $x$ to be added to the set  $\{e_j^i\}_{i,j}$ in order to yield $F$. 
 If $F$ is a doubly-truncted simplex, two vectors $x_1$ and $x_2$ are added to $\{e_j^i\}_{i,j}$  to form  $\sigma_F$.  
By taking into account Remark \ref{nopyr},  
we list all corresponding strings to consider  in Table \ref{case7}.

\begin{table}[!h]
\footnotesize
\begin{center}
\bigskip
\tabcolsep=12pt%
\renewcommand*{\arraystretch}{1.3}  
\begin{tabular}{|c|c|}
\hline
$\sigma_{\infty}$ & Strings  \\
\hline
\hline
 $\widetilde G_2\cup \widetilde A_4$    & 
(0,0,3;1,0,0,0,0)
\\ 
 $\widetilde G_2\cup \widetilde D_4$   & (0,0,3;1,0,0,0,0)
\\ 
 & $(0,0,a;0,1,0,0,0)$
\\
 $\widetilde G_2\cup \widetilde A_2 \cup \widetilde A_2$  & (0,0,3;1,0,0;1,0,0)
\\ 
 $\widetilde G_2\cup \widetilde G_2 \cup \widetilde A_2$   &  $(0,0,3;0,0,\sqrt3;1,0,0)$
\\ 
 $\widetilde G_2\cup \widetilde G_2 \cup \widetilde G_2$  & $(0,0,3;0,0,\sqrt3;0,0,\sqrt3)$\\
& $\{x_1\leftrightarrow(0,0,a_1;0,0,3;1,0,0),$  
\\ 
& $x_2\leftrightarrow(0,0,a_2;1,0,0;\sqrt3,0,0)\}$ \\ 
  &$\{x_1\leftrightarrow(0,0,a_1;0,0,\sqrt3;0,0,3),$ 
\\
& $x_2\leftrightarrow(0,0,a_2;0,0,3;1,0,0)\}$ \\ 
\hline
\end{tabular}
\end{center}
\caption{Strings for dimension $n=7$} 
\label{case7}
\end{table}

\hip 
 In what follows,  we only give the details for the case where  $\sigma_{\infty}=\widetilde G_2 \cup \widetilde G_2 \cup \widetilde G_2$. $F$ is then the pyramid with the Coxeter symbol $[6,3,3,3,3,6]$,  or one of the two doubly-truncated simplices.    For all the other cases,   the admissible pairs are summarized in Table \ref{pairs7}.

\hop 
(i) Assume that $\sigma_F=[6,3,3,3,3,6]$.  
By Remark \ref{nopyr},  
we consider  the vector $x$ with the string 
\[
x\leftrightarrow(0,0,3;0,0,\sqrt3;0,0,\sqrt3) \ .  
\]   
There are two  vectors $y_i$ such that  $\{x,y_i\}$  is an admissible pair,  namely
\[
 y_1\leftrightarrow(\sqrt3,0,0;0,0,3;1,0,0) \ ,    \,  y_2\leftrightarrow(\sqrt3,0,0;1,0,0;0,0,3), 
\]
and one has   $\langle x,y_1 \rangle=\langle x,y_2 \rangle=0$.  
Moreover,   we verify that  $\{y_1,  y_2\}$  is also  an admissible pair.  
However,  none of the admissible sets gives rise to a finite-volume $7$-polyhedron.  

\hop
(ii)  Assume that     
$F$ is the doubly truncated simplex depicted below.  

$$
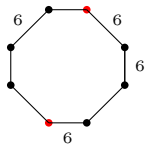
\begin{tikzpicture}
\tikzstyle{every node}=[font=\tiny]
\fill[black] (0,0) circle (0.05cm);
\fill[black] (1/2,1/2) circle (0.05cm);
\fill[red] (1,1/2) circle (0.05cm);
\fill[black] (1.5,0) circle (0.05cm);
\fill[black] (1.5,-1/2) circle (0.05cm);
\fill[black] (0,-1/2) circle (0.05cm);
\fill[red] (1/2,-1) circle (0.05cm);
\fill[black] (1,-1) circle (0.05cm);

\draw (0,0) -- (1/2,1/2) node [midway,yshift=0.3em,xshift=-0.4em] {$6$} ;
\draw (1,1/2) -- (1/2,1/2);
\draw (1.5,0) -- (1,1/2) node [midway,xshift=0.4em,yshift=0.3em] {$6$} ;
\draw (1.5,0) -- (1.5,-1/2)node [right,midway] {$6$} ;
\draw (1.5,0) -- (1.5,-1/2);
\draw (1,-1) -- (1.5,-1/2) ;
\draw (1,-1) -- (1/2,-1)node [below,midway] {$6$} ;
\draw (1/2,-1) -- (0,-1/2);
\draw (0,0) -- (0,-1/2);
\end{tikzpicture}$$

\vspace{-0.5em}

\noindent
In this case, we can write  for the red nodes 
\[
 x_1\leftrightarrow(0,0,a_1;0,0, \sqrt3;0,0,3)  \text{ and } x_2\leftrightarrow(0,0,a_2;1,0,0;\sqrt3,0,0) \ .
\]

\noindent
Then,  we   determine the coefficients $a_1, a_2 \in\{\sqrt3,3\}$ so that $\{x_1,x_2\}$ is an admissible pair.  As $\Lambda_1=\frac{a_1}{a_2}$
and  $\Lambda_2=\Lambda_3=\frac{1}{\sqrt3}$,  this yields $a_1=\frac{a_2}{\sqrt3}$,  that is,  $a_1=\sqrt3$ and $a_2=3$.  It follows that   $\langle x_1, x_2\rangle=0$.  However, the resulting polyhedron does not have  finite volume.

\hip 
Next, we search for all vectors $y$ such that  $\{x_1,x_2,y\}$  is an admissible set.    The only solution is given by  $y\leftrightarrow(1,0,0;0,0,3;\sqrt3,0,0)$,    and again,  we do not get a finite-volume polyhedron.

\hop (iii) Assume   that  $F$ is the doubly-truncated simplex depicted below.

$$
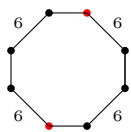
\begin{tikzpicture}\tikzstyle{every node}=[font=\tiny]
\fill[black] (0,0) circle (0.05cm);
\fill[black] (1/2,1/2) circle (0.05cm);
\fill[red] (1,1/2) circle (0.05cm);
\fill[black] (1.5,0) circle (0.05cm);
\fill[black] (1.5,-1/2) circle (0.05cm);
\fill[black] (0,-1/2) circle (0.05cm);
\fill[red] (1/2,-1) circle (0.05cm);
\fill[black] (1,-1) circle (0.05cm);
\draw (0,0) -- (1/2,1/2) node [midway,yshift=0.3em,xshift=-0.4em] {$6$} ;\draw (1,1/2) -- (1/2,1/2);\draw (1.5,0) -- (1,1/2) node [midway,yshift=0.3em,xshift=0.4em] {$6$} ;\draw (1.5,0) -- (1.5,-1/2);\draw (1.5,0) -- (1.5,-1/2);\draw (1,-1) -- (1.5,-1/2) node [midway,yshift=-0.3em,xshift=0.4em] {$6$} ;\draw (1,-1) -- (1/2,-1);\draw (1/2,-1) -- (0,-1/2) node [midway,yshift=-0.3em,xshift=-0.4em] {$6$} ;\draw (0,0) -- (0,-1/2);\end{tikzpicture}$$

\noindent 
We write $
x_1\leftrightarrow(0,0,a_1;1,0,0;0,0,3)  \text{ and } x_2\leftrightarrow(0,0,a_2;0,0,3;1,0,0)$.  

\hip As above, since  $\Lambda_2=\Lambda_3=1$,  this amounts to  $a_1=a_2\in \{\sqrt3,3\}$.  
As a consequence,  $\Delta_1=0$,  and  for each $a:=a_1=a_2\in \{\sqrt3,3\}$,  $\{x_1,x_2\}$ is admissible and such that $\langle x_1, x_2 \rangle = 0$.    
None of the resulting configurations yields a  finite-volume polyhedron.  

\hip
For $a=\sqrt3$,   
we find one vector  $y\leftrightarrow(1,0,0;\sqrt3,0,0;\sqrt3,0,0)$ such that $\{x_1,x_2,y\}$ is admissible. Also, this configuration does not give rise to a finite-volume polyhedron.

\hip
For $a=3$, 
we find all the vectors $z_i$ such that $\{x_1,x_2,z_i\}$ is admissible. They are given by  
\begin{equation*}
\begin{split}
z_1\leftrightarrow(1,0,0;0,0,3;0,0,3) \ , \ &  \, z_2\leftrightarrow(1,0,0;1,0,0;1,0,0) \ ,  \\ 
z_3\leftrightarrow(\sqrt3,0,\sqrt3;&\sqrt3,0,\sqrt3;\sqrt3,0,\sqrt3) \ .  
\end{split}
\end{equation*}
In addition,  $\{z_1,z_2\}$ and  $\{z_1,z_3\}$ are admissible pairs.  
However,  for each of the different admissible sets,  we do not obtain a  finite-volume polyhedron.

\begin{table}[!h]
\footnotesize
\begin{center}
\bigskip
\tabcolsep=2pt%
\renewcommand*{\arraystretch}{1.3}  
\begin{tabular}{|c|c|c|c|}
\hline
$\sigma_{\infty}$ & $x$ & $y$ & $\langle x,y \rangle$ \\  
\hline
\hline 
 $\widetilde G_2\cup \widetilde D_4$ 
 & $(0,0,\sqrt3;0,1,0,0,0)$ & $(0,\sqrt3,\sqrt3;1, 1, 1, 1, 1) $ & 0  \\
\hline
 $\widetilde G_2\cup \widetilde A_2 \cup \widetilde A_2$  & (0,0,3;1,0,0;1,0,0) & $y_1\leftrightarrow(\sqrt3,0,0;0,0,\sqrt3;\sqrt3,0,0) $ & 0 \\ 
& & $y_2\leftrightarrow(\sqrt3,0,0;0,\sqrt3,0;\sqrt3,0,0)$ & 0  \\ 
\hline
 $\widetilde G_2\cup \widetilde G_2 \cup \widetilde A_2$   &  $(0,0,3;0,0,\sqrt3;1,0,0)$ & $(\sqrt3,0,0;1,0,0;\sqrt3,0,0)$ & 0 
\\ 
\hline
\end{tabular}
\end{center}
\caption{The remaining admissible pairs $\{x,y\}$ for $n=7$}
\label{pairs7}
\end{table}

\noindent   For the remaining cases listed in 
Table \ref{pairs7}, let us mention that there is no admissible pair for  $\sigma_{\infty}=\widetilde G_2\cup \widetilde A_4$ and that the pair  $\{y_1,y_2\}$  as given in Table \ref{pairs7} for $\sigma_{\infty}=\widetilde G_2\cup \widetilde A_2\cup \widetilde A_2$  is not admissible. 
Again,  none of those cases  yields a finite-volume polyhedron in $\mathbb H^7$.

\subsubsection{Assume that  $\mathbf{n=8}$}

From the previous sections,   the face $F$ is either one of the two ADE-simplices or one of two ADEG-pyramids existing in $\mathbb H^6$. 
The corresponding strings are listed in Table \ref{cases8}.  

\begin{table}[!h]\footnotesize\begin{center}\bigskip\tabcolsep=12pt\renewcommand*{\arraystretch}{1.2}  \begin{tabular}{|c|c|}\hline$\sigma_{\infty}$ & Strings \\\hline\hline $\widetilde G_2\cup \widetilde A_5$   & $(0,0,a;1,0,0,0,0,0)$\\  $\widetilde G_2\cup \widetilde D_5$  & $(0,0,a;1,0,0,0,0,0)$\\  $\widetilde G_2 \cup \widetilde A_2  \cup \widetilde A_3 $  &  $(0,0,a;1,0,0;1;0,0,0)$ \\  $\widetilde G_2 \cup \widetilde G_2 \cup \widetilde A_3$   & $(0,0,a;0,0,\sqrt3;1;0,0,0)$ \\ \hline\end{tabular}\end{center}\caption{Strings for dimension $n=8$} \label{cases8}\end{table}

\hip
All the admissible pairs are summarized in Table \ref{pairs8}, and we give some details for the case where  $\sigma_{\infty}=\widetilde G_2\cup \widetilde A_5$.  
As there are six admissible pairs of the form $\{x,y_i\}$, we determine which of the vectors $y_1, \dots,  y_6$  form  admissible pairs and  give their Lorentzian products $\langle y_i, y_j\rangle$  in Table \ref{pairwise}.    %

\begin{table}[!h]
\footnotesize
\begin{center}
\bigskip
\tabcolsep=2pt%
\renewcommand*{\arraystretch}{1.3}  
\begin{tabular}{|c|c|c|c|}
\hline
$\sigma_{\infty}$ & $x$ & $y$ & $\langle x,y \rangle$ \\ 
\hline
\hline
 $\widetilde G_2\cup \widetilde A_5$  & $(0,0,\sqrt3;1,0,0,0,0,0)$ &  $y_1\leftrightarrow(0,\sqrt3,\sqrt3;0,0,0,1,1,1) $&  -1  \\
&& $ y_2\leftrightarrow(0,\sqrt3,\sqrt3;0,1,1,1,0,0)$ &  0  \\ 
&& $ y_3\leftrightarrow(0,\sqrt3,\sqrt3;1,1,0,0,0,1)$ &  0  \\ 
&& $y_4\leftrightarrow(1,3,3;\sqrt3,0,0,\sqrt3,\sqrt3,\sqrt3)$  &$0$ \\ 
&& $ y_5\leftrightarrow(1,3,3;\sqrt3,\sqrt3,\sqrt3,\sqrt3,0,0)$  & $0$ \\ 
&& $y_6\leftrightarrow(\sqrt3,\sqrt3,0;1,1,1,0,0,1)$ & 0 \\ 
\hline
 $\widetilde G_2\cup \widetilde D_5$  & $(0,0,\sqrt3;1,0,0,0,0,0)$ & $(0, \sqrt3, \sqrt3;1, 1, 0, 0, 0, 0)$ & 0  \\
& & $(\sqrt3,\sqrt3,0;1, 1, 0, 1, 1, 0)$ & 0 \\ 
\hline 
 $\widetilde G_2 \cup \widetilde A_2  \cup \widetilde A_3 $  &  $(0,0,3;1,0,0;1;0,0,0)$  & $(\sqrt3,0,0;0,0,\sqrt3;\sqrt3,0,0,0)$ &  0 \\ 
&& $(\sqrt3,0,0;0,\sqrt3,0;\sqrt3,0,0,0)$ &  0 \\ 
\hline 
 $\widetilde G_2 \cup \widetilde G_2 \cup \widetilde A_3$   & $(0,0,\sqrt3;0,0,\sqrt3;1;0,0,0)$ & $(0,\sqrt3,\sqrt3;\sqrt3,0,0;1,1,0,1)$ & -1 \\ 
&& $(\sqrt3,0,0;0,\sqrt3,\sqrt3;1,1,0,1)$ & -1 \\ 
\cline{2-4}
& $(0,0,3;0,0,\sqrt3;1;0,0,0)$  & $(\sqrt3,0,0;1,0,0;\sqrt3,0,0,0)$  & 0 \\ 
\hline
\end{tabular}
\end{center}
\caption{Admissible pairs $\{x,y\}$ for $n=8$} 
\label{pairs8}
\end{table}

\hip 
Empty boxes correspond to inadmissible pairs.   By looking at each admissible set,  we see  that none of them gives rise to a finite-volume polyhedron in $\mathbb H^8$.

\begin{table}[!h]
\footnotesize
\begin{center}
\begin{tabular}{|c|c|c|c|c|c|}
\hline
& $y_2$ & $y_3$ & $y_4$ & $y_5$ & $y_6$ \\ 
\hline 
$y_1$ &  0 & 0 && ${\scriptstyle -\sqrt3}$ & 0\\ 
\hline
$y_2$ & & 0 & ${\scriptstyle -\sqrt3}$ &  &0  \\ 
\hline
$y_3$ & &  & 0 & 0 &  \\ 
\hline
$y_4$ & &  & &  &  0\\ 
\hline
$y_5$ & &  & &  &  0\\ 
\hline
\end{tabular}
\end{center}
\caption{The Lorentzian products $\langle y_i, y_j\rangle$ for $\sigma_{\infty}=\widetilde G_2\cup \widetilde A_5$}
\label{pairwise}
\end{table}

\hip
In all other cases, there are no admissible sets  of cardinality $>2$, and none of the admissible pairs yield a finite-volume polyhedron in $\mathbb H^8$. 

\hop 

\noindent 
{\bf Interlude.} $ \ $ For $n\geq 9$,   the number of admissible pairs increases and admissible sets of big cardinality show up. 
Despite the considerable amount of cases to consider, we find that a single one will 
 provide a  hyperbolic Coxeter polyhedron of finite volume. As announced in Theorem \ref{thmN9}, it will be the polyhedron $P_{\star}\subset \mathbb H^9$. 

\hip 
We will  summarize  most of our findings, apart from a few cases, in Appendix \ref{AppendixC}. 
All admissible pairs $\{x,y_i\}$ together with the Lorentzian products $\langle x,y_j\rangle$ for a given $\sigma_{\infty}$ are listed, as well as  the Lorentzian products $\langle y_i,y_j\rangle$ when the pair $\{y_i, y_j\}$ is admissible and of further relevance; see Table \ref{pairwise}, for example.

\hip  In the following,  we give details for the cases where the $G_2$-face $F$  is the pyramid whose apex $v_{\infty}$ has a vertex link of type  $\widetilde G_2\cup \widetilde G_2 \cup \widetilde G_2$ yielding $P_{\star}\subset \mathbb H^9$, 
and for the cases where $F$ is neither a simplex nor a pyramid for $n> 9$.

\subsubsection{Assume that  $\mathbf{n=9}$}

By means of Lemma \ref{G2} and considering the previous sections,   there are twelve possibilities for  the face $F$.  Namely,  there are three simplices,  five pyramids over a product of two simplices,  and four pyramids over a product of three simplices.  
 The corresponding strings are listed in Table \ref{cases9}; see Remark \ref{nopyr}.

\begin{table}[!h]
\footnotesize
\begin{center}
\bigskip
\tabcolsep=12pt%
\renewcommand*{\arraystretch}{1.2} 
\begin{tabular}{|c|c|}
\hline
$\sigma_{\infty}$ &  Strings \\ 
\hline
\hline
 $\widetilde G_2\cup \widetilde A_6$  &  $(0,0,\sqrt3;1,0,0,0,0,0,0)$   \\ 
 $\widetilde G_2\cup \widetilde D_6$  & $(0,0,a;1,0,0,0,0,0,0)$ \\ 
 $\widetilde G_2\cup \widetilde E_6$  & $(0,0,a;1,0,0,0,0,0,0)$ \\ 
 $\widetilde G_2\cup \widetilde A_3\cup \widetilde A_3$  & $(0,0,a;1,0,0,0;1,0,0,0)$ \\ 
 $\widetilde G_2\cup \widetilde A_2\cup \widetilde A_4$ & $(0,0,a;1,0,0;1,0,0,0)$  \\ 
 $\widetilde G_2\cup \widetilde G_2\cup \widetilde A_4$  & $(0,0,a;0,0,\sqrt3;1,0,0,0,0)$ \\ 
 $\widetilde G_2\cup \widetilde A_2\cup \widetilde D_4$ & $(0,0,a;1,0,0;1,0,0,0)$  \\ 
 $\widetilde G_2\cup \widetilde G_2\cup \widetilde D_4$  & $(0,0,a;0,0,\sqrt3;1,0,0,0,0)$ \\ 
 $\widetilde G_2\cup \widetilde A_2\cup \widetilde A_2\cup \widetilde A_2$  & $(0,0,a; 0,0,1;0,0,1;0,0,1)$  \\ 
$\widetilde G_2\cup \widetilde G_2\cup \widetilde A_2\cup \widetilde A_2$   & $(0,0,a; 0,0,\sqrt3;0,0,1;0,0,1)$ \\ 
 $\widetilde G_2\cup \widetilde G_2\cup \widetilde G_2\cup \widetilde A_2$  & $(0,0,a; 0,0,\sqrt3;0,0,\sqrt3;1,0,0)$   \\ 
 $\widetilde G_2\cup \widetilde G_2\cup \widetilde G_2\cup \widetilde G_2$ & $(0,0,a; 0,0,\sqrt3;0,0,\sqrt3;0,0,\sqrt3)$  \\ 
\hline
\end{tabular}
\end{center}
\caption{Strings for  dimension $n=9$} 
\label{cases9}
\end{table}

\hop 
Remarquable is the case where  $\sigma_{\infty}=\widetilde G_2\cup \widetilde G_2\cup \widetilde G_2\cup \widetilde G_2$, and where  $F$ is the $7$-pyramid over a product of three simplices depicted below.  

$$
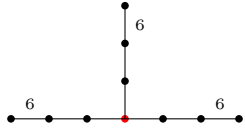
\begin{tikzpicture}
\tikzstyle{every node}=[font=\tiny]
\fill[black] (0,0) circle (0.05cm);
\fill[black] (1/2,0) circle (0.05cm);
\fill[black] (1,0) circle (0.05cm);
\fill[red] (3/2,0) circle (0.05cm);
\fill[black] (3/2,1/2) circle (0.05cm);
\fill[black] (3/2,1) circle (0.05cm);
\fill[black] (3/2,1.5) circle (0.05cm);
\fill[black] (4/2,0) circle (0.05cm);
\fill[black] (5/2,0) circle (0.05cm);
\fill[black] (6/2,0) circle (0.05cm);

\draw (0,0) -- (1/2,0) node [above,midway] {$6$} ;
\draw (1,0) -- (1/2,0);
\draw (1,0) -- (3/2,0) ;

\draw (5/2,0) -- (6/2,0) node [above,midway] {$6$} ;
\draw (5/2,0) -- (4/2,0);
\draw (4/2,0) -- (3/2,0) ;

\draw (3/2,1/2) -- (3/2,0) ;
\draw (3/2,1/2) -- (3/2,1) ;
\draw (3/2,1.5) -- (3/2,1) node [right,midway] {$6$} ;
\end{tikzpicture} 
$$

\hip
 As before, we denote $x\leftrightarrow(0,0,a; 0,0,\sqrt3;0,0,\sqrt3;0,0,\sqrt3)$  for $a\in\{\sqrt3,3\}$.   

\hop By applying Theorem \ref{thmAll} to one of the subdiagrams $G_2$ in (the symmetric) diagram $\sigma_F$, one immediately deduces that $a=\sqrt3$.    
We find that  there is a unique and beautiful  admissible pair $\{x, y\}$ where $y$ is given by 
\[
 y\leftrightarrow(1,0,0;1,0,0;1,0,0;1,0,0) \ , 
\]
\noindent
and we have $\langle x, y\rangle=0$.  
The corresponding diagram is depicted below.  


$$ 
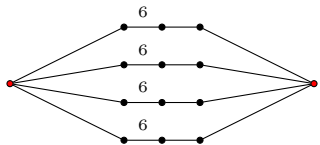
\begin{tikzpicture}[
every edge/.style = {draw=black},
 vrtx/.style args = {#1/#2}{%
      circle, draw,
      minimum size=1mm, label=#1:#2}
                    ]
\tikzstyle{every node}=[font=\tiny]
\node(A)  [vrtx=above/,scale=0.22,fill=black]  at (0,0) {};
\node(B)  [vrtx=above/,scale=0.22,fill=black]  at (0.5,0) {};
\node(C)  [vrtx=above/,scale=0.22,fill=black]  at (1,0) {};
\node(D)  [vrtx=above/,scale=0.22,fill=black]  at (0,0.5) {};
\node(E)  [vrtx=above/,scale=0.22,fill=black]  at (0.5,0.5) {};
\node(F)  [vrtx=above/,scale=0.22,fill=black]  at (1,0.5) {};
\node(G)  [vrtx=above/,scale=0.22,fill=black]  at (0,1) {};
\node(H)  [vrtx=above/,scale=0.22,fill=black]  at (0.5,1) {};
\node(I)  [vrtx=above/,scale=0.22,fill=black]  at (1,1) {};
\node(J)  [vrtx=above/,scale=0.22,fill=black]  at (0,1.5) {};
\node(K)  [vrtx=above/,scale=0.22,fill=black]  at (0.5,1.5) {};
\node(L)  [vrtx=above/,scale=0.22,fill=black]  at (1,1.5) {};
\node(y)  [vrtx=above/,scale=0.22,fill=red]  at (-1.5,0.75) {};
\node(x)  [vrtx=above/,scale=0.22,fill=red]  at (2.5,0.75) {};
\draw (A) -- (B) node[midway,above] {$6$}; 
\path (C) edge (B);
\draw  (D) -- (E)  node[midway,above] {$6$}; 
\path  (E) edge (F);  
\draw (G) -- (H) node[midway,above] {$6$}; 
\path (I) edge (H);
\draw (J) -- (K) node[midway,above] {$6$}; 
\path (L) edge (K)
 (y) edge (A)
 (y) edge (D)
 (y) edge (G)
 (y) edge (J)
 (x) edge (C)
 (x) edge (F)
 (x) edge (I)
 (x) edge (L);
\end{tikzpicture}
$$

\hip By means of Theorem \ref{thmVin2}, we verify that this polyhedron is a Coxeter polyhedron of finite volume in $ \mathbb H^9$, denoted by $P_{\star}$.  For further results about $P_{\star}$ and its associated Coxeter group,  see Section \ref{Section5}.
For all other cases, see  Tables \ref{admissibleG2A6},  \ref{admissibleG2D6},
  \ref{admissibleG2E6}, \ref{admissibleG2G2G2G2}, 
 \ref{admissibleG2G2G2A2},   \ref{admissibleG2G2A2A2},    \ref{admissibleG2A2A2A2}, \ref{admissibleG2G2A4},  \ref{admissibleG2A2A4}, \ref{admissibleG2G2D4},  \ref{admissibleG2A2D4} and  \ref{admissibleG2A3A3} in   Appendix \ref{AppendixC}.

\subsubsection{Assume that  $\mathbf{n=10}$}

There are seven possibilities for $F$.  Namely,  there are three ADE-simplices, three ADEG-pyramids and the ADE-polyhedron $P_1\subset \mathbb H^8$ depicted in Figure \ref{Prk}.  
All corresponding strings are listed in Table \ref{cases10}.

\begin{table}[!h]
\footnotesize
\begin{center}
\bigskip
\tabcolsep=12pt%
\renewcommand*{\arraystretch}{1.2} 
\begin{tabular}{|c|c|}
\hline
 $\sigma_{\infty}$ & Strings \\
\hline
\hline
 $\widetilde G_2\cup \widetilde A_7$ & $(0,0,a;1,0,0,0,0,0,0,0)$  \\ 
 & $\{x_1\leftrightarrow(0,0,a_1;1,1,0,0,0,0,0,0)$, \\ 
& $x_2\leftrightarrow(0,0,a_2;0,0,1,1,0,0,0,0)$,  \\ 
& $ x_3\leftrightarrow(0,0,a_3;0,0,0,0,1,1,0,0)$, \\ 
& $ x_4\leftrightarrow(0,0,a_4;0,0,0,0,1,1,0,0)\}$ \\ 
 $\widetilde G_2\cup \widetilde D_7$  & $(0,0,a;1,0,0,0,0,0,0,0)$ \\ 
 $\widetilde G_2\cup \widetilde E_7$  & $(0,0,a;1,0,0,0,0,0,0,0)$ \\ 
 $\widetilde G_2\cup \widetilde A_3\cup \widetilde D_4$  & $(0,0,a;1,0,0,0;1,0,0,0,0)$  \\ 
 $\widetilde G_2\cup \widetilde A_2\cup \widetilde D_5$  & $(0,0,a;1,0,0;1,0,0,0,0)$ \\ 
 $\widetilde G_2\cup \widetilde G_2\cup \widetilde D_5$  & $(0,0,a;0,0,\sqrt3;1,0,0,0,0)$
 \\ 
\hline
\end{tabular}
\end{center}
\caption{Strings for dimension $n=10$} 
\label{cases10}
\end{table}

\hop
Consider the case $\sigma_{\infty}=\widetilde G_2\cup \widetilde A_7$.  When the face $F$ is  an  ADE-simplex, we refer to Table \ref{admissibleG2A7} in Appendix  \ref{AppendixC}. 
 Of interest here is the case where the face $F$ is given by the ADE-polyhedron $P_1\subset \mathbb H^8$. 
Let us write  
\begin{equation*}
\begin{split}
x_1\leftrightarrow(0,0,a_{1};1,1,0,0,0,0,0,0) \ , & \,  \  x_2\leftrightarrow(0,0,a_{2};0,0,1,1,0,0,0,0) \ ,  \\ 
x_3\leftrightarrow(0,0,a_{3};0,0,0,0,1,1,0,0) \ , & \,  \ x_4\leftrightarrow(0,0,a_{4};0,0,0,0,0,0,1,1) \ , 
\end{split}
\end{equation*}
where  $a_i\in \{ \sqrt3,3\}$ for $i=1,2,3,4$ for the nodes depicted in red in the following diagram.

$$
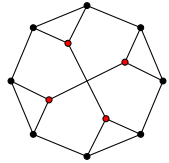
\begin{tikzpicture}[
every edge/.style = {draw=black, },
 vrtx/.style args = {#1/#2}{%
      circle, draw,
      minimum size=1mm, label=#1:#2}
                    ]
\node(x)  [vrtx=above/,scale=0.22,fill=red] at (1/2,1/4) {};
\node(y)  [vrtx=above/,scale=0.22,fill=red] at (1/4,-1/2) {};
\node(z)  [vrtx=above/,scale=0.22,fill=red] at (-1/2,-1/4) {};
\node(t)  [vrtx=above/,scale=0.22,fill=red] at (-1/4,1/2) {};
\tkzDefPoint(0,0){O}\tkzDefPoint(1,0){A}
\tkzDefPointsBy[rotation=center O angle 360/8](A,B,C,D,E,F,G){B,C,D,E,F,G,H}
\tkzDrawPoints[fill =black,size=2.2,color=black](A,B,C,D,E,F,G,H)
\tkzDrawPolygon[thin](A,B,C,D,E,F,G,H)
\path   (A) edge (x)
   (B) edge (x)
  (z) edge (x)
(C) edge (t)
   (D) edge (t)
  (t) edge (y)
(E) edge (z)
   (F) edge (z)
(G) edge (y)
   (H) edge (y);
\end{tikzpicture}$$

\noindent
From \eqref{lambda}, the set  $\{x_1,x_2,x_3,x_4\}$ is admissible if and only if  $a_1=a_2=a_3=a_4$.  
Set $a:=a_i$, $1\leq i \leq 4$.   
In both cases,  for  $a=\sqrt3$ and $a=3$,  the corresponding diagram does not encode a finite-volume polyhedron.  
Furthermore,  we do not find any additional vector $y\neq x_i$  such that $\{x_1, x_2, x_3, x_4, y\}$ is admissible.   

\hip 
For all other cases,  see Tables  \ref{admissibleG2A7},  
\ref{admissibleG2D7},   \ref{admissibleG2E7},  \ref{admissibleG2G2D5},   \ref{admissibleG2A2D5} and 
\ref{admissibleG2A3D4} in Appendix \ref{AppendixC}.

\subsubsection{Assume that  $\mathbf{n=11}$}

There are eight possibilities for $F$.  Namely,  there are three ADE-simplices,  three ADEG-pyramids,  the ADE-polyhedron $P_2\subset \mathbb H^9$  and the ADEG-polyhedron $P_{\star}\subset\mathbb H^9$.  
The corresponding strings are listed in Table \ref{cases11}.  

\hip 
In what follows, we provide details in the cases where $F=P_2$ and $F=P_{\star}$.  
For all other cases,  we refer to  Tables   \ref{admissibleG2D8},  \ref{admissibleG2E8},   \ref{admissibleG2D4D4},  \ref{admissibleG2D4D4bis},  \ref{admissibleG2A2E6} and \ref{admissibleG2G2E6}.

\begin{table}[!h]
\footnotesize
\begin{center}
\bigskip
\tabcolsep=6pt%
\renewcommand*{\arraystretch}{1.2} 
\begin{tabular}{|c|c|c|}
\hline
 $\sigma_{\infty}$ & Strings \\
\hline
\hline
 $\widetilde G_2\cup \widetilde A_8$ & $(0,0,a;1,0,0,0,0,0,0,0,0)$\\ 
 & $\{x_1\leftrightarrow(0,0,a_1;1,0,1,0,0,0,0,0,0)$, \\ 
& $x_2\leftrightarrow(0,0,a_2;0,0,0,1,0,1,0,0,0)$,  \\ 
& $ x_3\leftrightarrow(0,0,a_3;0,0,0,0,0,0,1,0,1)\}$, \\ 
  $\widetilde G_2\cup \widetilde D_8$& $(0,0,a;1,0,0,0,0,0,0,0,0)$ \\ 
  $\widetilde G_2\cup \widetilde E_8$ & $(0,0,a;1,0,0,0,0,0,0,0,0)$  \\ 
  $\widetilde G_2\cup \widetilde D_4\cup \widetilde D_4$ & $(0,0,a;1,0,0,0,0;1,0,0,0,0)$   \\ 
  $\widetilde G_2\cup \widetilde A_2\cup \widetilde E_6$   & $(0,0,a;1,0,0;1,0,0,0,0,0,0)$\\ 
  $\widetilde G_2\cup \widetilde G_2\cup \widetilde E_6$  & $(0,0,a;0,0,\sqrt3;1,0,0,0,0,0,0)$. \\ 
 $\widetilde G_2\cup \widetilde G_2\cup \widetilde G_2\cup \widetilde G_2\cup \widetilde G_2$  & 
$\{x_1\leftrightarrow(0,0,a_1;0,0,\sqrt3; 0,0,\sqrt3;0,0,\sqrt3;0,0,\sqrt3)$  \\ 
 & $x_2\leftrightarrow(0,0,a_2;1,0,0;1,0,0;1,0,0;1,0,0)\}$   \\ 
\hline
\end{tabular}
\end{center}
\caption{Strings for dimension $n=11$}
\label{cases11}
\end{table}

\noindent 
(i) Assume that   $\sigma_{\infty}=\widetilde G_2\cup \widetilde A_8$ and that $F=P_2$.   
For $a_1,  a_2, a_3\in\{\sqrt3,3\}$, encode the red nodes as depicted below with the following strings.   
\begin{equation*}
\begin{split}
x_1\leftrightarrow \ &(0,0,a_{1};1,0,1,0,0,0,0,0,0),  \\
x_2\leftrightarrow \ &(0,0,a_{2};0,0,0,1,0,1,0,0,0), \\ 
x_3\leftrightarrow \ &(0,0,a_{3};0,0,0,0,0,0,1,0,1), \\ 
\end{split}
\end{equation*}

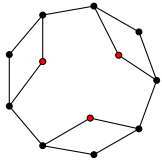
$$
\begin{tikzpicture}[
every edge/.style = {draw=black, },
 vrtx/.style args = {#1/#2}{%
      circle, draw,
      minimum size=1mm, label=#1:#2}
                    ]
\node(x)  [vrtx=above/,scale=0.22,fill=red] at (1/2,1/3) {};
\node(y)  [vrtx=above/,scale=0.22,fill=red] at (1/8,-1/2) {};
\node(z)  [vrtx=above/,scale=0.22,fill=red] at (-1/2,1/4) {};

\tkzDefPoint(0,0){O}\tkzDefPoint(1,0){A}
\tkzDefPointsBy[rotation=center O angle 360/9](A,B,C,D,E,F,G,H){B,C,D,E,F,G,H,I}
\tkzDrawPoints[fill =black,size=2.2,color=black](A,B,C,D,E,F,G,H,I)
\tkzDrawPolygon[thin](A,B,C,D,E,F,G,H,I)

\path   (A) edge (x)
   (C) edge (x)
(D) edge (z)
   (F) edge (z)
(G) edge (y)
   (I) edge (y)
;
\end{tikzpicture}$$

\noindent
As we require $\{x_1,x_2,x_3\}$ to be an admissible set,  we derive that $a:=a_1=a_2=a_3$.   
For both $a=\sqrt3$ and $a=3$, it turns out that the resulting polyhedron does not have finite volume.  

\hip 
  For $a=\sqrt3$,  we find eight additional vectors $y_i$ such that the set $\{x_1, x_2, x_3, y\}$ is admissible; see Table \ref{tableP21}.     In addition,  we give the Lorentzian products $\langle y_i,y_j\rangle$ when $\{y_i,y_j\}$ is an admissible pair in Table \ref{yy}.  
However,   none of  the admissible sets gives rise to a finite-volume polyhedron. 

\hip 
 For $a=3$,  there are six additional vectors; see Table \ref{tableP21}.  None of the admissible sets $\{x_1,x_2,x_3,y_i\}$  leads to a hyperbolic polyhedron of finite volume, and there is no admissible set of cardinality $>4$.

\begin{table}[!h]
\footnotesize
\begin{center}
\tabcolsep=12pt%
\renewcommand*{\arraystretch}{1.2} 
\begin{tabular}{|c|c|c|c|}
\hline
Admissible vectors for $a=\sqrt3$ & $\langle x_1,y_i\rangle$ & $\langle x_2,y_i\rangle$ & $\langle x_3,y_i\rangle$  \\ 
\hline
\hline
$y_1\leftrightarrow(1, 0, 0;0, 0, 0, 0, \sqrt3, 0, 0, \sqrt3, 0)$ & $-\sqrt3$  & 0 & 0  \\ 
$y_2\leftrightarrow(1, 0, 0; 0, 0, 0, \sqrt3, 0, 0, 0, 0, \sqrt3)$  & 0 & 0 & 0 \\ 
$y_3\leftrightarrow(1, 0, 0;0, 0, \sqrt3, 0, 0, 0,  \sqrt3, 0, 0)$ & 0 & 0 & 0 \\ 
$ y_4\leftrightarrow(1, 0, 0;0, \sqrt3, 0, 0, 0, 0, 0, \sqrt3, 0)$& 0 & $-\sqrt3$ & 0 \\ 
$y_5\leftrightarrow(1, 0, 0;0, \sqrt3, 0, 0, \sqrt3, 0, 0, 0, 0)$ & 0 & 0 & $-\sqrt3$ \\
$y_6\leftrightarrow(1, 0, 0;  \sqrt3, 0, 0, 0, 0, \sqrt3, 0, 0, 0)$ & 0 & 0 &0  \\ 
$y_7\leftrightarrow(\sqrt3, 0, 0;0, 1, 1, 0,  1, 1, 0, 1, 1)$& 0 & 0 &  0\\ 
$y_8\leftrightarrow(\sqrt3, 0, 0;1, 1, 0, 1, 1, 0, 1, 1, 0)$& 0 & 0 & 0 \\ 
\hline 
\hline
Admissible vectors for $a=3$ & $\langle x_1,y\rangle$ & $\langle x_2,y\rangle$ & $\langle x_3,y\rangle$  \\ 
\hline
\hline
$ (1, 0, 0; 0, 0, 0, 0, 1, 0, 0, 1, 0)$ & -1 & 0 & 0 \\ 
 $(1, 0, 0; 0, 0, 0, 1, 0, 0, 0,  0, 1)$ & 0 & 0 & 0 \\ 
 $(1, 0, 0;0, 0, 1, 0, 0, 0, 1, 0, 0)$ & 0 & 0 &0 \\ 
 $(1, 0, 0;0, 1, 0, 0, 0, 0, 0, 1, 0)$ & 0& -1 & 0\\ 
 $(1, 0, 0; 0, 1, 0, 0, 1, 0, 0, 0, 0)$ & 0 & 0 &-1 \\ 
 $(1, 0, 0;1, 0, 0, 0, 0, 1, 0, 0, 0)$ & 0 & 0 &0\\ 
\hline 
\end{tabular}
\end{center}
\caption{Admissible sets $\{x_1,x_2,x_3,y_i\}$ for $\sigma_F=P_2$}
\label{tableP21}
\end{table}

\vspace{-1em}

\begin{table}[!h]
\footnotesize
\begin{center}
\begin{tabular}{|c|c|c|c|c|c|c|c|}
\hline
& $y_2$ & $y_3$ & $y_4$ & $y_5$ & $y_6$ & $y_7$ & $y_8$ \\ 
\hline 
$y_1$ &- 1&-1&-1&-1&-1&0&0 \\ 
\hline
$y_2$ &&-1&-1&-1&-1&0&0\\ 
\hline
$y_3$ & & &-1&-1&-1&0&0 \\ 
\hline
$y_4$  & & & &- 1&-1&0&0 \\ 
\hline
$y_5$ &&&&&-1&0&0 \\ 
\hline
$y_6$ &&&&&&0&0 \\ 
\hline
\end{tabular}
\end{center}
\caption{Lorentzian products  $\langle y_i,y_j\rangle$}
\label{yy}
\end{table}

\noindent (ii) Assume that   $\sigma_{\infty}=\widetilde G_2\cup \widetilde G_2\cup \widetilde G_2\cup \widetilde G_2\cup \widetilde G_2$ and  that $F=P_{\star}$.  
For $a_1, a_2  \in \{\sqrt3, 3\}$,  write 
\begin{equation*}
\begin{split}
x_1&\leftrightarrow(0,0,a_1;0,0,\sqrt3;0,0,\sqrt3;0,0,\sqrt3;0,0,\sqrt3) \ ,  \\
  x_2&\leftrightarrow(0,0,a_2;1,0,0;1,0,0;1,0,0;1,0,0) \ , 
\end{split}
\end{equation*}
for the red nodes as depicted below. 

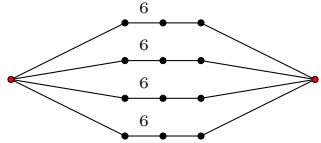
\vspace{-1em}

$$ \begin{tikzpicture}[
every edge/.style = {draw=black},
 vrtx/.style args = {#1/#2}{%
      circle, draw,
      minimum size=1mm, label=#1:#2}
                    ]
\tikzstyle{every node}=[font=\tiny]
\node(A)  [vrtx=above/,scale=0.22,fill=black]  at (0,0) {};
\node(B)  [vrtx=above/,scale=0.22,fill=black]  at (0.5,0) {};
\node(C)  [vrtx=above/,scale=0.22,fill=black]  at (1,0) {};
\node(D)  [vrtx=above/,scale=0.22,fill=black]  at (0,0.5) {};
\node(E)  [vrtx=above/,scale=0.22,fill=black]  at (0.5,0.5) {};
\node(F)  [vrtx=above/,scale=0.22,fill=black]  at (1,0.5) {};
\node(G)  [vrtx=above/,scale=0.22,fill=black]  at (0,1) {};
\node(H)  [vrtx=above/,scale=0.22,fill=black]  at (0.5,1) {};
\node(I)  [vrtx=above/,scale=0.22,fill=black]  at (1,1) {};
\node(J)  [vrtx=above/,scale=0.22,fill=black]  at (0,1.5) {};
\node(K)  [vrtx=above/,scale=0.22,fill=black]  at (0.5,1.5) {};
\node(L)  [vrtx=above/,scale=0.22,fill=black]  at (1,1.5) {};
\node(y)  [vrtx=above/,scale=0.22,fill=red]  at (-1.5,0.75) {};
\node(x)  [vrtx=above/,scale=0.22,fill=red]  at (2.5,0.75) {};
\draw (A) -- (B) node[midway,above] {$6$}; 
\path (C) edge (B);
\draw  (D) -- (E)  node[midway,above] {$6$}; 
\path  (E) edge (F);  
\draw (G) -- (H) node[midway,above] {$6$}; 
\path (I) edge (H);
\draw (J) -- (K) node[midway,above] {$6$}; 
\path (L) edge (K)
 (y) edge (A)
 (y) edge (D)
 (y) edge (G)
 (y) edge (J)
 (x) edge (C)
 (x) edge (F)
 (x) edge (I)
 (x) edge (L);
\end{tikzpicture}
$$

\noindent
Since $\{x_1,  x_2\}$ is  admissible pair,   this forces  $a_2$ to be equal to $\sqrt3 a_1$,  that is,  $a_1=\sqrt3$ and $a_2=3$.  We obtain  $\langle x_1, x_2 \rangle=0$.   However,  these data do not yield a finite-volume polyhedron. 

\hip 
Furthermore, we find the following pairwise admissible vectors $y_1, \ldots, y_4$
 such that the triplet  $\{x_1,x_2, y_i\}$ 
is admissible and 
 $\langle x_1,y_i\rangle=\langle x_2, y_i\rangle=0$ for all $i=1,\ldots,4$.  

\begin{equation*}
\begin{split}
&y_1\leftrightarrow(1,0,0;0,0,3;1,0,0;1,0,0;1,0,0),  \, y_2\leftrightarrow (1,0,0;1,0,0;0,0,3;1,0,0;1,0,0),\\
&y_3\leftrightarrow (1,0,0;1,0,0;1,0,0;0,0,3;1,0,0), \, y_4\leftrightarrow(1,0,0;1,0,0;1,0,0;1,0,0;0,0,3) \ 
\end{split}
\end{equation*}
However,  none of the  admissible sets  gives  rise to  a finite-volume hyperbolic polyhedron.

\subsubsection{Assume that $\mathbf{12\leq n \leq 18}$} 

Observe that there is nothing to verify in dimension $n=12$ as there exists neither an ADE-polyhedron nor an ADEG-polyhedron in $\mathbb H^{10}$ which can serve as a $G_2$-face.

\begin{table}[!h]
\footnotesize
\begin{center}
\bigskip
\tabcolsep=10pt%
\renewcommand*{\arraystretch}{1.2} 
\begin{tabular}{|c|c|c|}
\hline
$n$ & $\sigma_{\infty}$ &  Strings \\
\hline
\hline
$13$ & $\widetilde G_2\cup \widetilde A_2\cup \widetilde E_8$  & $(0,0,a;1,0,0;0,0,0,0,0,0,0,0,1)$\\ 
$$ & $\widetilde G_2\cup \widetilde G_2\cup \widetilde E_8$& $(0,0,a;0,0,\sqrt3;0,0,0,0,0,0,0,0,1)$ \\ 
\hline
$14$ & $\widetilde G_2\cup \widetilde A_3\cup \widetilde E_8$ & $(0,0,a;1,0,0,0;0,0,0,0,0,0,0,0,1)$ \\ 
\hline
$15$ & $\widetilde G_2\cup \widetilde D_4\cup \widetilde E_8$ & $(0,0,a;1,0,0,0,0;0,0,0,0,0,0,0,0,1)$ \\ 
\hline
\end{tabular}
\end{center}
\caption{Strings for dimensions $12\leq n \leq 18$} 
\label{allcases}
\end{table}

 \hip    
For dimension $n=13$,  there are two possibilities for  the face $F$: an ADE- or an ADEG-pyramid.  
The corresponding strings are listed in  Table \ref{allcases}.

\hip  If $\sigma_{\infty}=\widetilde G_2\cup \widetilde A_2\cup \widetilde E_8$,  we find no admissible pair, and 
the admissible pairs in the case where  $\sigma_{\infty}=\widetilde G_2\cup \widetilde G_2 \cup \widetilde E_8 $ are listed in Table \ref{admissibleG2G2E8}. None of them yields a finite-volume polyhedron. 
 
\begin{table}[!h] 
\footnotesize
\begin{center}
\renewcommand*{\arraystretch}{1.3} 
\begin{tabular}{|c|c|c|}\hline
  $x$ & $y$  & $\langle x, y \rangle$ \\ 
\hline 
 $ (0,0,\sqrt3;0,0,\sqrt3;0,0,0,0,0,0,0,0,1)$ 
& $(\sqrt3,0,0;\sqrt3,0,0;0,1,0,0,0,0,0,0,0)$ & 0   \\
& $(0,\sqrt3,\sqrt3;\sqrt3,0,0;0,0,0,0,0,0,0,1,1)$ & -1 \\
& $(\sqrt3,0,0;0,\sqrt3,\sqrt3;0,0,0,0,0,0,0,1,1)$ & -1 \\
\hline
 $ (0,0,3;0,0,\sqrt3;0,0,0,0,0,0,0,0,1)$ 
& $(\sqrt3,0,0;1,0,0;0,0,0,0,0,0,0,0,\sqrt3)$ & 0   \\
\hline
\end{tabular}
\caption{Admissible pairs $\{x,y\}$ for  $\sigma_{\infty}=\widetilde G_2\cup \widetilde G_2 \cup \widetilde E_8 $ }
\label{admissibleG2G2E8}
\end{center}
\end{table}

\noindent
For dimensions $n=14$ and $n=15$,  $F$ has to be the unique ADE-pyramid in $\mathbb H^{n-2}$. 
The corresponding strings are given  in  Table \ref{allcases}, and in both cases, we do not find any admissible pair. Hence, there is no ADEG-polyhedron in $\mathbb H^n$.

\hip For dimensions $n=16, 17$  and $18$,  there is neither an ADE-polyhedron nor an ADEG-polyhedron in $\mathbb H^{n-2}$ serving as a $G_2$-face $F$, and this finishes the proof for $n\leq 18$.

\subsubsection{Assume that $\mathbf{n\geq19}$}

In dimension $n=19$,  the single possibility for the $G_2$-face $F$ is given by the unique ADE-polyhedron in $ \mathbb H^{17}$.  

$$
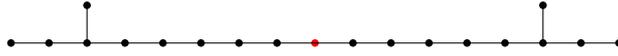
\begin{tikzpicture}
\tikzstyle{every node}=[font=\tiny]
\fill[black] (0,0) circle (0.05cm);
\fill[black] (1/2,0) circle (0.05cm);
\fill[black] (1,0) circle (0.05cm);
\fill[black] (3/2,0) circle (0.05cm);

\fill[black] (4/2,0) circle (0.05cm);
\fill[black] (5/2,0) circle (0.05cm);
\fill[black] (6/2,0) circle (0.05cm);

\fill[black] (7/2,0) circle (0.05cm);
\fill[red] (8/2,0) circle (0.05cm);
\fill[black] (9/2,0) circle (0.05cm);
\fill[black] (10/2,0) circle (0.05cm);

\fill[black] (11/2,0) circle (0.05cm);
\fill[black] (12/2,0) circle (0.05cm);
\fill[black] (13/2,0) circle (0.05cm);
\fill[black] (14/2,0) circle (0.05cm);
\fill[black] (15/2,0) circle (0.05cm);
\fill[black] (16/2,0) circle (0.05cm);

\fill[black] (2/2,1/2) circle (0.05cm);
\fill[black] (14/2,1/2) circle (0.05cm);
\draw (14/2,1/2) -- (14/2,0);
\draw (2/2,1/2) -- (2/2,0);
\draw (0,0) -- (16/2,0);
\end{tikzpicture} 
$$

\hop
In this case, $\sigma_{\infty}=\widetilde G_2 \cup \widetilde E_8 \cup \widetilde E_8$,  and 
we write  
\[
x\leftrightarrow(0,0,a;0,0,0,0,0,0,0,0,1;0,0,0,0,0,0,0,0,1) \ , \ a\in\{\sqrt3,3\} \ . 
\]

\hip 
For $a=3$,  there is a unique admissible pair which, however, does not yield a finite-volume polyhedron; see Table \ref{pairs19}.  

\hip 
For $a=\sqrt3$, there are $26$ vectors $y$ such that $\{x,y\}$ is an admissible pair, listed in Table \ref{pairs19}.  
Apart from a few exceptions,  each admissible set gives rise to a spherical subdiagram of rank $>19$, and all remaining ones do not yield a finite-volume polyhedron.   This finishes the proof for $n=19$.

\hop 
Finally, the inductive procedure stops since there are no  
ADE-polyhedra and no ADEG-polyhedra both in dimension $18$ and dimension $19$. 
This finishes the proof of Theorem \ref{thmN9}.

$\hfill\qed$

\section{Arithmeticity and commensurability}
\label{Section5}

To conclude this work, we  disccus some more aspects of ADEG-polyhedra and their associated geometric Coxeter groups which we will call ADEG-groups; see \eqref{presentation}.  

\hip 
For the broad theory of arithmetic  groups which we apply here in the restricted context of hyperbolic Coxeter groups, we refer to \cite{Maclachlan, MR, V0}. 
Let us recall that two discrete groups $\Gamma_1, \Gamma_2$ in $ \hbox{Isom}\mathbb H^n$ are said to be  {\it commensurable (in the wide sense)}  if  there is an element $\gamma\in\hbox{Isom}\mathbb H^n$
such that $\,\Gamma_1\cap \gamma\Gamma_2\gamma^{-1}$ has finite index in both $\Gamma_1$ and $\gamma\Gamma_2\gamma^{-1}$. 
By a fundamental result of  Margulis, it is known that a cofinite group  $\Gamma$ in $\hbox{Isom}\mathbb H^n$ for $n\ge3$ is {\it arithmetic} if and only if its commensurator
\begin{equation*}
\hbox{Comm}(\Gamma)=\{\,\gamma\in\Isom\mathbb H^n\,\mid\, \Gamma \hbox{ and } \gamma\Gamma\gamma^{-1} \hbox{ are commensurable}\,\}
\end{equation*}
is dense in $\hbox{Isom}\mathbb H^n$; see \cite{Zimmer} for example.

\hip 
In dimension $n=2$,  due to the work of Takeuchi \cite[Appendix 13.3]{MR}, 
all discrete triangle-groups are classified up to arithmeticity, and the four compact ADEG-triangle groups are arithmetic defined over $\mathbb Q(\sqrt3)$.

\hip 
For $n\geq 3$, all ADEG-groups in  $\hbox{Isom}\mathbb H^n$ are non-cocompact; see Sections \ref{Section3} and \ref{Section4}. In such a setting, 
the arithmeticity property can be characterized in an easy way thanks to Vinberg's arithmeticity criterion; see \cite{V2}. 
In fact, a non-cocompact Coxeter group $\Gamma$ is arithmetic (and defined over $\mathbb Q$) if and only if all {\it cycles} in the matrix  $2Gr(\Gamma)$ are rational integers. 
We therefore easily verify that all (ADE- and) ADEG-groups in $\hbox{Isom}\mathbb H^n$ for $n\geq 3$ are arithmetic.

\hip 
In addition, from 
 Emery's work  \cite{Emery}, for non-cocompact arithmetic groups acting on a hyperbolic space of {\it odd} dimension $\geq 5$, their covolume is a rational multiple of Riemann's zeta function evaluated at $5$, denoted $\zeta(5)$. 
In particular, this applies to the ADE- and ADEG-groups  $\Gamma_{\star}=\Gamma(P_{\star})$ and  $\Gamma(P_2)$ associated with the polyhedra $P_{\star}$ and $P_2$ in $\mathbb H^9$.

\hop 

Let us consider the Coxeter simplex group $\Gamma_9$ in $\hbox{Isom}\mathbb H^9$ depicted in Figure \ref{Gamma9}. 
 
\vspace{0.3em} 
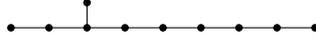
\begin{figure}[!h]
\centering
\begin{tikzpicture}
\fill[black] (1/2,0) circle (0.05cm);
\fill[black] (2/2,0) circle (0.05cm);
\fill[black] (3/2,0) circle (0.05cm);
\fill[black] (4/2,0) circle (0.05cm);
\fill[black] (5/2,0) circle (0.05cm);
\fill[black] (6/2,0) circle (0.05cm);
\fill[black] (7/2,0) circle (0.05cm);
\fill[black] (8/2,0) circle (0.05cm);
\fill[black] (9/2,0) circle (0.05cm);
\fill[black] (3/2,1/3) circle (0.05cm);
\draw (1/2,0) -- (3/2,0) ;
\draw (4/2,0) -- (3/2,0) ;
\draw (4/2,0) -- (5/2,0) ;
\draw (6/2,0) -- (5/2,0) ;
\draw (6/2,0) -- (7/2,0) ;
\draw (7/2,0) -- (8/2,0) ;
\draw (9/2,0) -- (8/2,0) ;
\draw (3/2,1/3) -- (3/2,0) ;
\end{tikzpicture}
\caption{The Coxeter simplex group $\Gamma_9\subset \hbox{Isom}\mathbb H^9$}
\label{Gamma9}
\end{figure}

\hip 
The Coxeter group $\Gamma_9$  is arithmetic (and defined over $\mathbb Q$) as well, 
and it is distinguished by 
the property that among {\it all} cusped hyperbolic $9$-orbifolds,  
it has minimal covolume and it is unique as such; see the work of Hild \cite{Hild}.  
More precisely, its covolume has been computed  in \cite{JKRT}  and is  given by 
\begin{equation}
\hbox{covol}(\Gamma_9)= \frac{\zeta(5)}{22,295,347,200} \ . 
\label{volGamma9}
 \end{equation}

\vspace{0.3em} 
\begin{prop}
The Coxeter groups $\Gamma_{\star}$,  $\Gamma(P_2)$  and  $\Gamma_9$  in $\hbox{Isom}\mathbb H^9$ are pairwise commensurable. 
\label{commens}
\end{prop}

\begin{proof}
By Maclachlan's work \cite{Maclachlan}, there is a {\it complete} set of commensurability invariants for arithmetic subgroups of $\hbox{Isom}\mathbb H^n$. 
This work has been exploited by Guglielmetti,  Jacquemet and  Kellerhals  \cite{GJK} who established these invariants for Coxeter simplex groups and Coxeter pyramid groups, 
and in particular  for the group $\Gamma_9\subset \hbox{Isom}\mathbb H^9$.  
Hence, it is sufficient to show that the sets of invariants for 
$\Gamma_{\star}$ and  $\Gamma(P_2)$ coincide with the one of $\Gamma_9$. 

\hip 
In this odd-dimensional setting, the complete set of commensurability invariants is given by $\{\mathbb Q,  1, \emptyset\}$  where $\mathbb Q$ is the field of definition of the group,  $1$ is the signed determinant  in $\mathbb Q/\mathbb Q^2$ of its Vinberg form $q$, and  $\hbox{Ram}_q(B)=\emptyset$ is the ramification set for a quaternion algebra  representing the Witt invariant  $c(q)$ of $q$. 
We refer to this author's thesis  \cite{Bredon} for the computational details of the Vinberg forms as well as of these invariants.  
We conclude that  all three groups $\Gamma_9$, $\Gamma_{\star}$ and  $\Gamma_{P_2}$ are pairwise commensurable.
\end{proof}

\hip 
From Proposition \ref{commens},  \eqref{volGamma9} and minimality of the covolume of $\Gamma_9$,  we  derive that 
\begin{equation}
\hbox{vol}(P_{\star})= q\cdot \frac{\zeta(5)}{22,295,347,200}    \ \text{ with }\ q\in\mathbb Q_{>1} \ . 
\label{volPstarbis}  
\end{equation}
A similar expression holds for $\hbox{vol}(P_2)$.

\begin{remark}\normalfont 
From the works  \cite{GJK,JKRT2} and Propostion \ref{commens}, we derive that the commensurability class of $\Gamma_{\star}$  contains in fact all  Coxeter simplex groups and all Coxeter pyramid groups  in $\hbox{Isom}\mathbb H^9$.   
\end{remark}

\clearpage

\appendix

\section{-- The ADE-polyhedra in $\mathbb H^n$}
\label{AppendixA}

\begin{table}[!h]
\footnotesize
\begin{tabular}{c|c}
$n$ &\\
\hline
3 &  
\begin{tikzpicture}
\tikzstyle{every node}=[font=\tiny]
\fill[black] (0,0) circle (0.05cm);
\fill[black] (1/2,1/4) circle (0.05cm);
\fill[black] (1/2,-1/4) circle (0.05cm);
\fill[black] (-1/2,0) circle (0.05cm);
\draw (0,0) -- (1/2,1/4) ;
\draw (1/2,-1/4) -- (0,0)  ;
\draw (1/2,-1/4) -- (1/2,1/4)  ;
\draw (0,0) -- (-1/2,0) ;
\draw (0,0.5) -- (0,0.51) ;
\end{tikzpicture} 
\qquad \qquad
 \begin{tikzpicture}\tikzstyle{every node}=[font=\tiny]
\fill[black] (0,0) circle (0.05cm);
\fill[black] (1/2,0) circle (0.05cm);
\fill[black] (0,1/2) circle (0.05cm);
\fill[black] (1/2,1/2) circle (0.05cm);
\draw (0,0) -- (1/2,1/2) ;
\draw (1/2,0) -- (0,1/2)  ;
\draw (1/2,1/2) -- (1/2,0)  ;
\draw (1/2,1/2) -- (0,1/2)  ;
\draw (0,0) -- (0,1/2)  ;
\draw (0,0) -- (1/2,0)  ;
\end{tikzpicture} 
\qquad  \quad
 \begin{tikzpicture}\tikzstyle{every node}=[font=\tiny]
\fill[black] (0,0) circle (0.05cm);
\fill[black] (1/2,0) circle (0.05cm);
\fill[black] (0,1/2) circle (0.05cm);
\fill[black] (1/2,1/2) circle (0.05cm);
\draw (0,0) -- (1/2,1/2) ;
\draw (1/2,1/2) -- (1/2,0)  ;
\draw (1/2,1/2) -- (0,1/2)  ;
\draw (0,0) -- (0,1/2)  ;
\draw (0,0) -- (1/2,0)  ;
\end{tikzpicture}  \\
\hline
\end{tabular}
\caption{The ADE-polyhedra in $\mathbb H^n$}
\label{ADE}
\end{table}

\clearpage

\section{-- Data for roots systems of types $A, D, E$ and $G$}
\label{AppendixB}

We provide a short description of the root systems  of type  
$ A,   D,  E$ and $ G$, together with 
the quantities  $s_{ij}:=s_is_j$ involved in  Prokhorov's formula  stated in  Lemma \ref{Prkformula}. 

\hip For the $ A,  D$ and $ E$ types, they can be found in \cite{Prk}, and they are easily computed in view of  \eqref{sisj} for $\widetilde G_2$ (and for other types).  For general data about the root systems involved, we refer to Bourbaki \cite{Bourbaki}. 

\hip 
Denote by  $\epsilon_1, \ldots, \epsilon_{n+1}$ the canonical basis of  $\mathbb R^{n+1}$. 

\subsection{The root system $ A_n$} 

Let $V$ be the subspace of vectors in $\mathbb R^{n+1}$ whose coordinates sum up to $0$.  
The $A_n$-root system is the set of vectors in $V$ of length $\sqrt2$ with integer coordinates given by

\[
\mathcal A_n=\{ \epsilon_i-\epsilon_j \mid 1\leq i \neq j \leq n+1\} \ . 
\]
The natural simple roots are  $
\alpha_1=\epsilon_1-\epsilon_2,  \ \dots \ ,  \  \alpha_n=\epsilon_n-\epsilon_{n+1}$,  
and the  highest root can be expressed as 
$\overline{\alpha}=\alpha_1+\ldots+\alpha_n  = \epsilon_{n+1}-\epsilon_1=:-\alpha_{n+1}$.  The notation $\alpha_{n+1}=-\bar\alpha$ will also be used in the subsequent cases. 
\hip 
The extended Dynkin diagram is given below, where each node is indexed by the corresponding root.   
$$
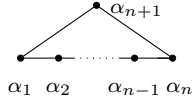
\begin{tikzpicture}[
every edge/.style = {draw=black},
 vrtx/.style args = {#1/#2}{%
      circle, draw,
      minimum size=1mm, label=#1:#2}
                    ]
\tikzstyle{every node}=[font=\scriptsize]
\node(A) [vrtx=right/$\alpha_{n+1}$,scale=0.22,fill=black] at (0, -0.3) {};
\node(B) [vrtx=below/$\alpha_1$,scale=0.22,fill=black] at (-1, -1) {};
\node(C) [vrtx=below/$\alpha_2$,scale=0.22,fill=black] at (-0.5,-1) {};
\node(D) [vrtx=below/$\alpha_{n-1}$,scale=0.22,fill=black] at (0.5,-1) {};
\node(E) [vrtx=below/$\ \ \alpha_n$,scale=0.22,fill=black] at (1,-1) {};
\path   (A) edge (B)
        (B) edge (C)
 (E) edge (A)
 (E) edge (0.3,-1)
;
\draw (-0.3,-1) -- (0.3,-1) [dotted];
\draw (1,-1) -- (0.8,-1); 
\draw (-0.5,-1) -- (-0.3,-1); 
\end{tikzpicture}$$
\noindent 
The fundamental weights of $A_n$ are given by 
\[
w_i=\epsilon_1+\ldots+\epsilon_i-\frac{i}{n+1}\sum_{j=1}^{n+1}\epsilon_j  \ \text{ for } \ 1\leq i \leq n \ . 
\]

\hip 
In view of \eqref{sisj}, the quantities $s_{ij}$ involved in Lemma \ref{thmPrk} for a component of  type $\widetilde A_n$ are given by 
\[
s_{ij}= \frac{(j-i)(n+1-(j-i))}{2(n+1)} \ . 
\]

\hip 
Note that for the proof of Theorem \ref{thmN9}, these quantities are needed  for $n\leq 8$ only.

\hip

\subsection{The root system $ D_n$} 

Let $V=\mathbb R^n$.  
The  $D_n$-root system is the set of vectors in $V$ of length $\sqrt2$ with integer coordinates given by 
 \[
\mathcal D_n=\{ \epsilon_i-\epsilon_j \mid 1\leq i \neq j \leq n\}\cup \{ \pm(\epsilon_i+\epsilon_j)  \mid 1 \leq i<j\leq n \}\ . 
\]
The natural simple roots are given by $\alpha_1=\epsilon_1-\epsilon_2,  \dots,  \alpha_{n-1}=\epsilon_{n-1}-\epsilon_{n} \ , \ \alpha_{n}$,
and the  highest root  satisfies  
$
\overline{\alpha}=\alpha_1+2\alpha_2+\ldots+2\alpha_{n-2}+\alpha_{n-1}+\alpha_n = \epsilon_1+ \epsilon_2 \ 
$. 

\hip 
The extended Dynkin diagram is given below, where the nodes are indexed by the roots.  
$$
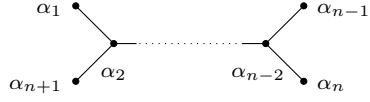
\begin{tikzpicture}[
every edge/.style = {draw=black},
 vrtx/.style args = {#1/#2}{%
      circle, draw,
      minimum size=1mm, label=#1:#2}
                    ]
\tikzstyle{every node}=[font=\scriptsize]
\node(A) [vrtx=left/$\alpha_{1}$,scale=0.22, fill=black] at (0, 1) {};
\node(B) [vrtx=left/$\alpha_{n+1}$,scale=0.22,fill=black] at (0, 0) {};
\node(C) [vrtx=below/$\alpha_2$,scale=0.22,fill=black] at (0.5,0.5) {};
\node(F) [vrtx=below/$\alpha_{n-2} \ \ $,scale=0.22,fill=black] at (2.5,0.5) {};
\node(G) [vrtx=right/$\alpha_{n-1}$,scale=0.22,fill=black] at (3, 1) {};
\node(H) [vrtx=right/$\alpha_{n}$,scale=0.22,fill=black] at (3, 0) {};

\path   (A) edge (C)
        (B) edge (C)
 (F) edge (G)
 (F) edge (H);

\draw (C) -- (0.8,0.5) ; 
\draw (2.2,0.5) -- (0.8,0.5) [dotted];
\draw (2.2,0.5) -- (F) ;
\end{tikzpicture}
$$

\noindent The fundamental weights are given by
\[
\left\{
\begin{array}{ll}
w_i&=\epsilon_1+\ldots+ \epsilon_i \  \text{ for } i<n-1 \ ,  \\ 
w_{n-1}&=\frac{1}{2}( \epsilon_1+\ldots+ \epsilon_{n-1}-  \epsilon_n) \ ,  \\ 
w_n&=\frac{1}{2}( \epsilon_1+\ldots+ \epsilon_{n-1}+  \epsilon_n )\ . 
\end{array}
\right.
\]

\hip The  quantities $s_{ij}:=s_is_j$ defined in  \eqref{sisj} are given below. 
\[
\left\{ 
\begin{array}{llll}
s_is_j=\frac{j-i}{2} & \text{ for } { 2\leq i\leq j \leq n-2}  \\
s_{n+1}s_j=s_1s_j=\frac{j}{4} & \text{ for }  { j\leq n-2} \\ 
s_1s_{n-1}=s_1s_n=s_{n-1}s_{n+1}=s_ns_{n+1}=\frac{n}{8} \\
s_{n-1}s_n=s_1s_{n+1}=\frac1{2} ; s_js_{n-1}=s_js_n=\frac{n-j}{4} & \text{ for }  { j\leq n-2}
\end{array}
\right.
\]

\hip Note that for the proof of Theorem \ref{thmN9},  these quantities are needed only for $n\leq 6$.  

\subsection{The root systems $ E_8,  E_7$ and $ E_6$}

$\bullet$ 
Let $V=\mathbb R^8$, and consider the lattice $L=L_0+\mathbb Z(\frac{1}{2}\sum\limits_{1\leq i \leq 8}\epsilon_i)$ where $L_0$ consists of all vectors $\sum\limits_{1\leq i \leq 8} c_i\epsilon_i$ with $c_i\in \mathbb Z$ and $\sum\limits_{1\leq i \leq 8} c_i$ even.  
\hip 
The   $E_8$-root system is the set of vectors of length $\sqrt2$  in $L$, with 
 the following natural simple roots  
\begin{equation}
\left\{
\begin{array}{ll}
\alpha_1&=\frac{1}{2}(\epsilon_1-\epsilon_2 -\dots -\epsilon_7 + \epsilon_8) \\ 
\alpha_2&= \epsilon_1 + \epsilon_2  \\ 
\alpha_i&= \epsilon_{i-1}-\epsilon_{i-2} \ \text{ for } 3\leq i \leq 8 
\end{array}
\right.  
\label{rootsE8}
\end{equation}
and the highest root can be expressed as $$
\overline{\alpha}=2\alpha_1+3\alpha_2+4\alpha_3+6\alpha_4+5\alpha_5+4\alpha_6+3\alpha_7+2\alpha_8 \ . $$ 

\hip
The extended Dynkin diagram is given below. 
$$
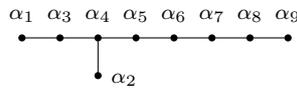
\begin{tikzpicture}[
every edge/.style = {draw=black,},
 vrtx/.style args = {#1/#2}{%
      circle, draw,
      minimum size=1mm, label=#1:#2}
                    ]
\tikzstyle{every node}=[font=\scriptsize]
\node(A) [vrtx=above/$\alpha_{3}$,scale=0.22, fill=black] at (0, 0) {};
\node(B) [vrtx=above/$\alpha_{4}$,scale=0.22, fill=black] at (0.5, 0) {};
\node(C) [vrtx=above/$\alpha_5$,scale=0.22, fill=black]at (1,0) {};
\node(D) [vrtx=above/$\alpha_6$,scale=0.22, fill=black] at (1.5,0) {};
\node(E) [vrtx=above/$\alpha_7$,scale=0.22, fill=black] at (2,0) {};
\node(F) [vrtx=right/$\alpha_{2}$,scale=0.22, fill=black] at (0.5,-0.5) {};
\node(G) [vrtx=above/$\alpha_{8}$,scale=0.22, fill=black] at (2.5, 0) {};
\node(H) [vrtx=above/$\alpha_{9}$,scale=0.22, fill=black] at (3, 0) {};
\node(K) [vrtx=above/$\alpha_{1}$,scale=0.22, fill=black] at (-0.5, 0) {};
\path   (A) edge (B)
   (A) edge (K)
        (B) edge (C)
 (C) edge (D)
 (D) edge (E)
 (G) edge (E)
 (F) edge (B)
 (H) edge (G);

\end{tikzpicture}$$
The fundamental weights are given by 
\[
\left\{
\begin{array}{ll}
w_1=2\epsilon_8   
\\ 
w_2= \frac{1}{2}(\epsilon_1+\epsilon_2+\epsilon_3+\epsilon_4+\epsilon_5+\epsilon_6+\epsilon_7+5\epsilon_8) 
\\  
w_3= \frac{1}{2}(-\epsilon_1+\epsilon_2+\epsilon_3+\epsilon_4+\epsilon_5+\epsilon_6+\epsilon_7+7\epsilon_8) 
\\   
w_4= \epsilon_3+\epsilon_4+\epsilon_5+\epsilon_6+\epsilon_7+5\epsilon_8 
\\   
w_5=  \epsilon_4+\epsilon_5+\epsilon_6+\epsilon_7+4\epsilon_8
\\   
w_6= \epsilon_5+\epsilon_6+\epsilon_7+3\epsilon_8    
\\   
w_7= \epsilon_6+\epsilon_7+2\epsilon_8  
\\   
w_8= \epsilon_7+\epsilon_8
\\   
\end{array}
\right.  
\]

\hip and the quantities $s_{ij}:=s_is_j$ defined in  \eqref{sisj} are listed in the following table.

$$
 \setlength\extrarowheight{-2pt}
\addtolength{\tabcolsep}{-2pt}
\begin{tabular}{|c|c |c |c |c |c |c | c | c | c}
\hline
& $s_2$ & $s_3$ & $s_4$& $s_5$& $s_6$ &  $s_7$&  $s_8$&  $s_9$  \\
\hline
$s_1$ & $\nicefrac2{3}$ & $ \nicefrac{1}{2}$ & $1$ & $1 $& $1$ & $1$ &$1$ & $1$ \\
\hline
$s_2$ && $\nicefrac{7}{12}$ & $ \nicefrac1{2}$ & $ \nicefrac2{3}$ & $ \nicefrac{5}{6}$ & $1$ & $ \nicefrac7{6}$ &$\nicefrac4{3}$\\
\hline
$s_3$ & & & $\nicefrac{1}{2}$ & $\nicefrac{3}{4}$ & $ 1$ & $ \nicefrac{5}{4}$ & $\nicefrac3{2}$& $ \nicefrac{7}{4}$ \\
\hline
$s_4$ & & & & $\nicefrac1{2}$ & $1$ & $\nicefrac{3}{2}$& $2$& $\nicefrac{5}{2}$ \\
\hline
$s_5$ & & & & & $\nicefrac1{2}$ & $1$ & $\nicefrac3{2}$ & $2$ \\
\hline
$s_6$ && & & & & $\nicefrac1{2}$ & $1$& $\nicefrac3{2}$  \\
\hline
$s_7$ && & & & &  & $\nicefrac1{2}$ & 1 \\
\hline
$s_8$ && & & & & & & $\nicefrac1{2}$  \\
\hline
\end{tabular}$$

\hop

\noindent  $\bullet$ 
 Let $V$ be the hyperplane in $\mathbb R^8$ generated by the first seven simple roots  $\alpha_1,\dots, \alpha_7$ of $\mathcal E_8$. 
Then, $V$ is orthogonal to $w_8$. 

\hip The $E_7$-root system is defined as  $\mathcal E_7= \mathcal E_8\cap V$.
It has natural simple roots $\alpha_1, \dots, \alpha_7$; see \eqref{rootsE8}. 
The highest root is given by 
\[
\overline{\alpha}=2\alpha_1+2\alpha_2+3\alpha_3+4\alpha_4+3\alpha_5+2\alpha_6+\alpha_7 = \epsilon_8 - \epsilon_7\ ,  
\]
and the extended Dynkin diagram is given as follows. 
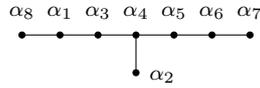
\vspace{-1em}
$$\begin{tikzpicture}[
every edge/.style = {draw=black,},
 vrtx/.style args = {#1/#2}{%
      circle, draw,
      minimum size=1mm, label=#1:#2}
                    ]
\tikzstyle{every node}=[font=\scriptsize]
\node(A) [vrtx=above/$\alpha_{1}$,scale=0.22, fill=black] at (0, 0) {};
\node(B) [vrtx=above/$\alpha_{3}$,scale=0.22,fill=black] at (0.5, 0) {};
\node(C) [vrtx=above/$\alpha_4$,scale=0.22,fill=black] at (1,0) {};
\node(D) [vrtx=above/$\alpha_5$,scale=0.22,fill=black] at (1.5,0) {};
\node(E) [vrtx=above/$\alpha_6$,scale=0.22,fill=black] at (2,0) {};
\node(F) [vrtx=right/$\alpha_{2}$,scale=0.22,fill=black] at (1,-0.5) {};
\node(G) [vrtx=above/$\alpha_{7}$,scale=0.22,fill=black] at (2.5, 0) {};
\node(H) [vrtx=above/$\alpha_{8}$,scale=0.22,fill=black] at (-0.5, 0) {};
\path   (A) edge (B)
        (B) edge (C)
 (C) edge (D)
 (D) edge (E)
 (F) edge (C)

 (E) edge (G)
 (A) edge (H);
\end{tikzpicture}
$$

\vspace{-2em}

\noindent
The fundamental weights are 
\[
\left\{
\begin{array}{ll}
w_1= -\epsilon_7+\epsilon_8 
\\   
w_2=  \frac{1}{2}(\epsilon_1+\epsilon_2+\epsilon_3+\epsilon_4+\epsilon_5+\epsilon_6-2\epsilon_7+2\epsilon_8) %
\\   
w_3= \frac{1}{2}(- \epsilon_1+\epsilon_2+\epsilon_3+\epsilon_4+\epsilon_5+\epsilon_6-3\epsilon_7+3\epsilon_8) 
\\   
w_4=\epsilon_3+\epsilon_4+\epsilon_5+\epsilon_6-2\epsilon_7+2\epsilon_8
\\ 
w_5= \epsilon_4+\epsilon_5+\epsilon_6-\frac3{2}\epsilon_7+\frac3{2}\epsilon_8
\\  
w_6= \epsilon_5+\epsilon_6-\epsilon_7+\epsilon_8
\\   
w_7= \epsilon_6-\frac1{2}\epsilon_7+\frac1{2}\epsilon_8
\\  
\end{array}
\right.  
\]
\hip

\hip The quantities $s_{ij}:=s_is_j$ defined in \eqref{sisj} are given in the following table. 

$$
 \setlength\extrarowheight{-2pt}
\addtolength{\tabcolsep}{-2pt}
 \begin{tabular}{|c|c |c |c |c |c |c | c | c}
\hline
& $s_2$ & $s_3$ & $s_4$& $s_5$& $s_6$ &  $s_7$&  $s_8$  \\
\hline
$s_1$ & $\nicefrac{3}{4}$ & $ \nicefrac{1}{2}$ & 1 & 1 & 1 & 1 &$ \nicefrac1{2}$ \\
\hline
$s_2$ && $\nicefrac{5}{8}$ & $ \nicefrac1{2}$ & $ \nicefrac5{8}$ & $ \nicefrac{3}{4}$ & $ \nicefrac{7}{8}$ & $ \nicefrac7{8} $\\
\hline
$s_3$ & & & $\nicefrac{1}{3}$ & $\nicefrac{3}{4}$ & $ 1$ & $ \nicefrac{5}{4}$ & $1$ \\
\hline
$s_4$ & & & & $\nicefrac1{2}$ & $1$ & $\nicefrac{3}{2}$& $\nicefrac{3}{2}$ \\
\hline
$s_5$ & & & & & $\nicefrac1{2}$ & $1$ & $\nicefrac5{4}$  \\
\hline
$s_6$ && & & & & $\nicefrac1{2}$ & $1$  \\
\hline
$s_7$ && & & & &  & $\nicefrac3{4}$  \\
\hline
\end{tabular}$$

\hop

\noindent $\bullet$ 
Let $V$ be the hyperplane generated by $\alpha_1,\dots, \alpha_6$ in $\mathbb R^8$.  Then, $V$ is orthogonal to $w_7$ and $w_8$.   
The $E_6$-root system is defined as  $\mathcal E_6=\mathcal E_8 \cap V$ and has the natural simple roots $\alpha_1, \dots, \alpha_6$; see \eqref{rootsE8}. The highest root is given by 
\[
\overline{\alpha}=\alpha_1+2\alpha_2+2\alpha_3+4\alpha_4+2\alpha_5+\alpha_6 \ . 
\]
\hip 
The extended Dynkin diagram is given below. 
$$
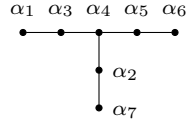
\begin{tikzpicture}[
every edge/.style = {draw=black,},
 vrtx/.style args = {#1/#2}{%
      circle, draw,
      minimum size=1mm, label=#1:#2}
                    ]
\tikzstyle{every node}=[font=\scriptsize]
\node(A) [vrtx=above/$\alpha_{1}$,scale=0.22,fill=black] at (0, 0) {};
\node(B) [vrtx=above/$\alpha_{3}$,scale=0.22,fill=black] at (0.5, 0) {};
\node(C) [vrtx=above/$\alpha_4$,scale=0.22,fill=black] at (1,0) {};
\node(D) [vrtx=above/$\alpha_5$,scale=0.22,fill=black] at (1.5,0) {};
\node(E) [vrtx=above/$\alpha_6$,scale=0.22,fill=black] at (2,0) {};
\node(F) [vrtx=right/$\alpha_{2}$,scale=0.22,fill=black] at (1,-0.5) {};
\node(G) [vrtx=right/$\alpha_{7}$,scale=0.22,fill=black] at (1, -1) {};

\path   (A) edge (B)
        (B) edge (C)
 (C) edge (D)
 (D) edge (E)
 (F) edge (C)
 (F) edge (G);

\end{tikzpicture}
$$ 

\hip 
The fundamental weights are given as follows. 
\[
\left\{
\begin{array}{ll}
w_1=  \frac{2}{3}(-\epsilon_6-\epsilon_7+\epsilon_8) 
\\   
w_2=  \frac{1}{2}(\epsilon_1+\epsilon_2+\epsilon_3+\epsilon_4+\epsilon_5-\epsilon_6-\epsilon_7+\epsilon_8) 
\\   
w_3=   \frac{1}{2}(-\epsilon_1+\epsilon_2+\epsilon_3+\epsilon_4+\epsilon_5)  + \frac{5}{6}(-\epsilon_6-\epsilon_7+\epsilon_8) 
\\   
w_4= \epsilon_3+\epsilon_4+\epsilon_5-\epsilon_6-\epsilon_7+\epsilon_8
\\ 
w_5= \epsilon_4+\epsilon_5 +  \frac{2}{3}(-\epsilon_6-\epsilon_7+\epsilon_8) 
\\  
w_6=\epsilon_5 + \frac{1}{3}(-\epsilon_6-\epsilon_7+\epsilon_8) 
\\   
\end{array}
\right.  
\]

\hip The quantities $s_{ij}:=s_is_j$ defined in \eqref{sisj} are given in the following table. 

\begin{center}
\addtolength{\tabcolsep}{-2pt}
 \setlength\extrarowheight{-2pt}
\begin{tabular}{|c |c |c |c |c |c | c|}
\hline
& $s_2$ & $s_3$ & $s_4$& $s_5$& $s_6$ &  $s_7$  \\
\hline
$s_1$ & $\nicefrac{5}{6}$ & $\nicefrac{1}{2}$ & $1$ & $  \nicefrac{5}{6} $ & $ \nicefrac{2}{3}$ &$ \nicefrac{2}{3}$ \\
\hline
$s_2$ & & $ \nicefrac{2}{3}$ & $ \nicefrac1{2}$ & $ \nicefrac{2}{3}$ & $ \nicefrac{5}{6}$ & $ \nicefrac1{2} $\\
\hline
$s_3$ & & & $ \nicefrac1{2}$ & $ \nicefrac{2}{3}$ & $ \nicefrac{5}{6}$ & $ \nicefrac{5}{6}$ \\
\hline
$s_4$ & & & & $\nicefrac1{2}$ & $1$ & $1$ \\
\hline
$s_5$ & & & &  & $\nicefrac1{2}$ & $\nicefrac5{6}$  \\
\hline
$s_6$ && & & &  & $\nicefrac1{3}$  \\
\hline
\end{tabular}
\end{center}

  \hip 
\subsection{ The root system $\widetilde G_2$} 
\label{AppendixB.4}

Let $V$ be the hyperplane in $\mathbb R^3$ consisting of all vectors whose coordinates add up to $0$. 
The $G_2$-root system  is  the set of vectors in $V$ with integer coordinates and of length  $\sqrt2$ or $\sqrt6$
 given by 
 \[
\mathcal G_2=\{ \pm(\epsilon_i-\epsilon_j) \mid  { 1\leq i < j \leq 3}\}\cup \{ \pm(2\epsilon_i-\epsilon_j-\epsilon_k)  \mid { 1 \leq i, j,k \leq 3 }\}\ . 
\]
It has natural simple roots given by 
$\alpha_1=\epsilon_1-\epsilon_2$ and $\alpha_2=-2\epsilon_{1}+\epsilon_2+\epsilon_3$.  The highest root can be written  according to  $
\overline{\alpha}=3\alpha_1+2\alpha_2 = -\epsilon_1 - \epsilon_2 +2\epsilon_3\ $. 

\hip The extended Dynkin diagram is given below. 

$$
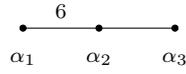
\begin{tikzpicture}
[
every edge/.style = {draw=black},
every vrtx/.style = {fill=black},
 vrtx/.style args = {#1/#2}{%
      circle, draw,
      minimum size=1mm, label=#1:#2}
                    ]
\tikzstyle{every node}=[font=\scriptsize]
\node(A) [fill=black,vrtx=below/$\alpha_{1}$,scale=0.22] at (0, 0) {};
\node(B) [fill=black,vrtx=below/$\alpha_2$,scale=0.22] at (1,0) {};
\node(C) [fill=black,vrtx=below/$\alpha_3$,scale=0.22] at (2,0) {};
\draw   (A) -- (B) node[midway,above] {$6$}; 
\path        (B) edge (C); 
\end{tikzpicture}
$$
\noindent 
The fundamental weights are given by 
$w_1= 2\alpha_1+\alpha_2$ and $w_2=3\alpha_1+2\alpha_2$.  

\hip   In this case, from  \eqref{sisj},  we compute  
\[
s_{12}= \frac{13}{6} \ ,  \  s_{23}= \frac1{3} \  \text{  and } \ s_{13} =\frac{3}{2} \ .
\]

 \newpage 

\section{-- Admissible pairs in higher dimensions} 
\label{AppendixC}

All the admissible pairs for the cases that are not 
detailled throughout Section  \ref{Section4.2} are listed in the following tables.  
For each pair $\{x,y_i\}$, we  indicate the Lorentzian product $\langle x,y_i\rangle$, and, 
when relevant, we also provide the Lorentzian products $\langle y_i,y_j\rangle$; see Section \ref{Section4.1.4}. 
 
\bigskip

\begin{table*}[!h]
\footnotesize
\begin{minipage}[t]{\textwidth}
\centering

\caption{Admissible pairs for $\sigma_{\infty}=\widetilde G_2\cup \widetilde A_6$}
\label{admissibleG2A6}
\end{minipage}

\vspace{2em}

\begin{minipage}[t]{\textwidth}
\centering
%
\caption{Admissible pairs  for $\sigma_{\infty}=\widetilde G_2\cup \widetilde D_6$}
\label{admissibleG2D6}
\end{minipage}

\vspace{2em}

\begin{minipage}[t]{\textwidth}
\centering
%
\caption{Admissible pairs for $\sigma_{\infty}=\widetilde G_2\cup \widetilde E_6$}
\label{admissibleG2E6}
\end{minipage}
\end{table*}

\begin{table*}[!h]
\footnotesize
\begin{minipage}[t]{\textwidth}
\centering
%
\caption{Admissible pairs for $\sigma_{\infty}=\widetilde G_2\cup \widetilde G_2\cup \widetilde G_2\cup \widetilde G_2$}
\label{admissibleG2G2G2G2}
\end{minipage}

\vspace{2em}

\begin{minipage}[t]{\textwidth}
\centering
%
\caption{Admissible pairs  for $\sigma_{\infty}=\widetilde G_2\cup \widetilde G_2\cup \widetilde G_2\cup \widetilde A_2$}
\label{admissibleG2G2G2A2}
\end{minipage}

\vspace{2em}

\begin{minipage}[t]{\textwidth} 
\centering
%
\caption{Admissible pairs  for $\sigma_{\infty}= \widetilde G_2\cup \widetilde G_2\cup \widetilde A_2\cup \widetilde A_2$}
\label{admissibleG2G2A2A2}
\end{minipage}
\end{table*}

\begin{table*}
\footnotesize
\begin{minipage}[t]{\columnwidth}
\centering
%
\caption{Admissible pairs  for $\sigma_{\infty}= \widetilde G_2\cup \widetilde A_2\cup \widetilde A_2\cup \widetilde A_2$}
\label{admissibleG2A2A2A2}
\end{minipage}
\end{table*}

\begin{table*}
\footnotesize
\begin{minipage}[t]{\textwidth}
\centering
\begin{tabular}{|c|c|c|}\hline
  $x$ & $y$  & $\langle x, y \rangle$ \\ 
\hline 
\hline
 $(0,0,\sqrt3;0,0,\sqrt3;1,0,0,0,0), $ 
& $ (0,\sqrt3,\sqrt3;\sqrt3,0,0;1,1,0,0,1) $ &-1  \\
& $(\sqrt3,0,0;0,\sqrt3,\sqrt3;1,1,0,0,1) $  &-1 \\ 
\cline{1-3}
 $(0,0,3;0,0,\sqrt3;1,0,0,0,0) $ 
& $(\sqrt3,0,0;1,0,0;\sqrt3,0,0,0,0) $ & 0   \\
\hline
\end{tabular}
\caption{Admissible pairs  for  $\sigma_{\infty}=\widetilde G_2\cup \widetilde G_2\cup \widetilde A_4$}
\label{admissibleG2G2A4}
\end{minipage} 

\bigskip 
\bigskip

\begin{minipage}[t]{\textwidth}
\centering
\begin{tabular}{|c|c|c|}\hline
  $x$ & $y$  & $\langle x, y \rangle$ \\ 
\hline 
\hline
 $(0,0,3;1,0,0;1,0,0,0) $ 
& $(\sqrt3, 0,0;0,0,\sqrt3; \sqrt3,0,0,0,0) $ & 0  \\
& $ (\sqrt3, 0,0;0,\sqrt3,0; \sqrt3,0,0,0,0)$  & 0 \\ 
\hline
\end{tabular}
\caption{Admissible pairs  for  $\sigma_{\infty}= \widetilde G_2\cup \widetilde A_2\cup \widetilde A_4$}
\label{admissibleG2A2A4}
\end{minipage} 

\bigskip 
\bigskip 

\begin{minipage}[t]{\textwidth}
\centering
\begin{tabular}{|c|c|c|}\hline
  $x$ & $y$  & $\langle x, y \rangle$ \\ 
\hline 
\hline
 $(0,0,\sqrt3,0,0,\sqrt3,1,0,0,0,0),$ 
& $(0,\sqrt3,\sqrt3,\sqrt3,0,0,1,1,0,0,0) $ & 0  \\
& $(\sqrt3,0,0,0,\sqrt3,\sqrt3,1,1,0,0,0) $  & 0 \\ 
\cline{1-3}
 $(0,0,3,0,0,\sqrt3,1,0,0,0,0)$ 
& $ (\sqrt3,0,0;1,0,0;\sqrt3, 0,0,0,0)$ & 0   \\
\hline
\end{tabular}
\caption{Admissible pairs  for $\sigma_{\infty}=\widetilde G_2\cup \widetilde G_2\cup \widetilde D_4$}
\label{admissibleG2G2D4}
\end{minipage} 

\bigskip 
\bigskip

\begin{minipage}[t]{\textwidth}
\centering
\begin{tabular}{|c|c|c|}\hline
  $x$ & $y$  & $\langle x, y \rangle$ \\ 
\hline 
\hline
 $ (0,0,\sqrt3;1,0,0;1,0,0,0)$ 
& $y\leftrightarrow(0,0,\sqrt3;1,1,1;0,0,1,1,1) $ & 0  \\
\cline{1-3}
 $(0,0,3;1,0,0;1,0,0,0) $ 
& $ (\sqrt3,0,0,0,0,\sqrt3,\sqrt3,0,0,0,0)$ & 0   \\
 & $ (\sqrt3,0,0,0,\sqrt3,0,\sqrt3,0,0,0,0)$  & 0 \\ 
\hline
\end{tabular}
\caption{Admissible pairs  for  $\sigma_{\infty}=\widetilde G_2\cup \widetilde A_2\cup \widetilde D_4$}
\label{admissibleG2A2D4}
\end{minipage}

\bigskip 
\bigskip

\begin{minipage}[t]{\textwidth}
\centering
\begin{tabular}{|c|c|c|}\hline
  $x$ & $y$  & $\langle x, y \rangle$ \\ 
\hline 
\hline
 $(0,0,\sqrt3;1,0,0,0;1,0,0,0)$ 
& $(1,0,0;0,0,\sqrt3,0;0,0,\sqrt3,0) $ & 0  \\
\cline{1-3}
 $(0,0,3;1,0,0,0;1,0,0,0) $ 
& $(1,0,0;0,0,1,0;0,0,1,0) $ & 0   \\
\hline
\end{tabular}
\caption{Admissible pairs for $\sigma_{\infty}=\widetilde G_2\cup \widetilde A_3\cup \widetilde A_3$}
\label{admissibleG2A3A3}
\end{minipage}
\end{table*}

\begin{table*}
\footnotesize
\begin{minipage}[t]{\textwidth}
\centering
%
\caption{Admissible pairs  for  $\sigma_{\infty}=\widetilde G_2\cup \widetilde A_7$ }
\label{admissibleG2A7}
\end{minipage} 

\bigskip
\bigskip

\begin{minipage}[t]{\textwidth}
\centering
%
\caption{Admissible pairs for  $\sigma_{\infty}=\widetilde G_2\cup \widetilde D_7$ }
\label{admissibleG2D7}
\end{minipage} 
\end{table*}

\begin{table*}
\footnotesize
\begin{minipage}[t]{\textwidth}
\centering
%
\caption{Admissible pairs  for  $\sigma_{\infty}=\widetilde G_2\cup \widetilde E_7$ }
\label{admissibleG2E7}
\end{minipage} 

\bigskip
\bigskip 

\begin{minipage}[t]{\textwidth}
\centering
\begin{tabular}{|c|c|c|}\hline
  $x$ & $y$  & $\langle x, y \rangle$ \\ 
\hline 
\hline
 $(0,0,\sqrt3;0,0,\sqrt3;1,0,0,0,0,0)$ 
& $(0, \sqrt3, \sqrt3;\sqrt3, 0, 0;1, 1, 0, 0, 0, 0)$ & -1  \\
& $(\sqrt3, 0, 0;0, \sqrt3, \sqrt3;1, 1, 0, 0, 0, 0)$ & -1 \\ 
\cline{1-3}
 $(0,0,3;0,0,\sqrt3;1,0,0,0,0,0) $ 
& $(\sqrt3, 0, 0;1, 0, 0;\sqrt3, 0, 0, 0, 0, 0)$ & 0   \\
\hline
\end{tabular}
\caption{Admissible pairs  for $\sigma_{\infty}=\widetilde G_2\cup \widetilde G_2\cup \widetilde D_5$}
\label{admissibleG2G2D5}
\end{minipage}

\bigskip
\bigskip 

\begin{minipage}[t]{\textwidth}
\centering
\begin{tabular}{|c|c|c|}\hline
  $x$ & $y$  & $\langle x, y \rangle$ \\ 
\hline 
\hline
 $(0, 0, \sqrt3;1, 0, 0;1, 0, 0, 0, 0, 0)$ 
& $(0, \sqrt3,\sqrt3; 1, 1, 1;0, 0, 1, 0, 0, 1)$ & -1  \\
\cline{1-3}
 $(0,0,3;1,0,0;1,0,0,0,0,0) $ 
& $(\sqrt3, 0, 0;0, 0, \sqrt3;\sqrt3, 0, 0, 0, 0, 0)$ & 0   \\
& $(\sqrt3, 0, 0;0, \sqrt3, 0;\sqrt3, 0, 0, 0, 0, 0)$ & 0   \\
\hline
\end{tabular}
\caption{Admissible pairs $\{x,y\}$ for $\sigma_{\infty}=\widetilde G_2\cup \widetilde A_2\cup \widetilde D_5$}
\label{admissibleG2A2D5}
\end{minipage} 

\bigskip
\bigskip 

\begin{minipage}[t]{\textwidth}
\centering
%
\caption{Admissible pairs for $\sigma_{\infty}=\widetilde G_2\cup \widetilde A_3\cup \widetilde D_4$  }
\label{admissibleG2A3D4}
\end{minipage}
\end{table*}


\begin{table*}
\vspace{-20pt}
\footnotesize
\begin{minipage}[t]{\textwidth}
\footnotesize
\centering
%
\caption{Admissible pairs for $\sigma_{\infty}=\widetilde G_2\cup \widetilde D_8$}
\label{admissibleG2D8}
\end{minipage}

\bigskip

\begin{minipage}[t]{\textwidth}
\centering
%
\caption{Admissible pairs $\{x,y_i\}$ for $\sigma_{\infty}=\widetilde G_2\cup \widetilde E_8$}
\label{admissibleG2E8}
\end{minipage} 
\end{table*}

\begin{table*}
\begin{minipage}[t]{\textwidth}
\centering
\begin{tabular}{|c|c|c|}\hline
  $x$ & $y$  & $\langle x, y \rangle$ \\ 
\hline 
\hline
 $ (0,0,3;1,0,0,0,0;1,0,0,0,0)$ 
& $(1,0,0;0,0,0,0,1;0,0,0,0,1)$ &  0 \\
& $ (1,0,0;0,0,0,0,1;0,0,0,1,0)$ & 0\\
& $ (1,0,0;0,0,0,0,1;0,0,1,0,0)$ & 0\\
& $ (1,0,0;0,0,0,1,0;0,0,0,0,1)$ &   0\\
& $(1,0,0;0,0,0,1,0;0,0,0,1,0)$ &  0\\
& $ (1,0,0;0,0,0,1,0;0,0,1,0,0)$ &  0\\
& $ (1,0,0;0,0,1,0,0;0,0,0,0,1) $ &  0\\
\hline
\end{tabular}
\caption{Admissible pairs $\{x,y_i\}$ for  $\sigma_{\infty}=\widetilde G_2\cup \widetilde D_4\cup \widetilde D_4$ }
\label{admissibleG2D4D4}
\end{minipage} 
\end{table*}

\begin{table*}
\vspace{-42pt}
\footnotesize
\begin{minipage}[t]{\textwidth}
\centering
%
\caption{Lorentzian products   $\langle y_i,y_j\rangle$ for  $\sigma_{\infty}=\widetilde G_2\cup \widetilde D_4\cup \widetilde D_4$}
\label{admissibleG2D4D4bis}
\end{minipage}
 
\bigskip 

\begin{minipage}[t]{\textwidth}
\centering
%
\caption{Admissible pairs $\{x,y_i\}$  for $\sigma_{\infty}=\widetilde G_2\cup \widetilde A_2\cup \widetilde E_6$ }
\label{admissibleG2A2E6}
\end{minipage}
\end{table*}

\begin{table*}
\footnotesize
\begin{minipage}[t]{\textwidth}
\centering
%
\caption{Admissible pairs $\{x,y_i\}$ for  $\sigma_{\infty}=\widetilde G_2\cup \widetilde G_2 \cup \widetilde E_6 $   }
\label{admissibleG2G2E6}
\end{minipage}

\bigskip 
\bigskip

\begin{minipage}[t]{\textwidth}
\centering
%
 
}\\
\hline 
\end{tabular}
\caption{Admissible pairs for  $\sigma_{\infty}=\widetilde G_2\cup \widetilde E_8 \cup \widetilde E_8 $}
\label{pairs19}
\end{minipage}
\end{table*}

\noindent
{\small
 \textsc{Department of Mathematics, University of Luxembourg, L-4364 Esch-sur-Alzette, Luxembourg.} \\
  \textit{E-mail address}, \texttt{naomi.bredon@uni.lu}
}


\clearpage 


\begin{thebibliography}{90}

\footnotesize 

\bibitem{All}
D. Allcock, 
{\em {Infinitely many hyperbolic Coxeter groups through dimension 19}},  Geom. Topol. 10 (2006), 737--758.  


\bibitem{Andreev}
E. M.  Andreev,
{\em {On convex polyhedra of finite volume in Lobachevskii spaces}},  Math. USSR Sbornik 12 (1970), 255--259.


\bibitem{Bor}
R. E. Borcherds,
{\em {Automorphism groups of Lorentzian lattices}},  J. Algebra  111 (1987), 133--153.


\bibitem{Bourbaki}
N. Bourbaki, 
{\em Groupes et alg\`ebres de Lie}, 
Ch. 4-6. Hermann, Paris, 1968.

\bibitem{Bredon}
N. Bredon 
{\em On the existence and growth properties of hyperbolic Coxeter groups}, PhD thesis n. 8046, Université de Fribourg. 

\bibitem{Chein}
M. Chein, 
{\em {Recherche des graphes des matrices des Coxeter hyperboliques d'ordre $\leq 10$}},
Modélisation Mathématique et Analyse Numérique, Vol. 3 (1969),  3--16.


\bibitem{Emery}
V. Emery,  
{\em   On volumes of quasi-arithmetic hyperbolic lattices},
Selecta Math. 23 (2017), 2849--2862. 


\bibitem{Ess}
F.  Esselmann, 
{\em The classification of compact hyperbolic Coxeter $d$-polytopes with $d+2$ facets},  
Comment. Math. Helvetici 71 (1996), 229--242.


\bibitem{Fweb}
A. Felikson,
{\em Hyperbolic  Coxeter polytopes},\\ 
www.maths.dur.ac.uk/users/anna.felikson/Polytopes/polytopes.html

\bibitem{FT1} 
A. Felikson,  P. Tumarkin, 
{\em On hyperbolic Coxeter polytopes with mutually intersecting facets}, 
J. Combin. Theory Ser. A 115 (2008), 121--146.



\bibitem{FT2} 
A. Felikson,  P. Tumarkin, 
{\em Essential hyperbolic Coxeter polytopes}, 
Israel Journal of Mathematics 199 (2014), 113--161.  


\bibitem{GJK}
R. Guglielmetti,  M.  Jacquemet,  R.  Kellerhals,
{\em  Commensurability of hyperbolic Coxeter groups: theory and computation}, 
RIMS Kôkyûroku Bessatsu B66 (2017), 57-113.

\bibitem{Gug}
R. Guglielmetti, 
{\em CoxIterWeb},
https://coxiterweb.rafaelguglielmetti.ch/


\bibitem{Hild}
T. Hild,
{\em The cusped hyperbolic orbifolds of minimal volume in dimensions less than ten}, 
J. Algebra 313 (2007), 208--222. 



\bibitem{ImHof} 
H.  Ch.  Im Hof, 
{\em Napier cycles and hyperbolic Coxeter groups}, 
Bull. Soc. Math. de Belg. Serie A (1990), 523--545.


\bibitem{JT}
M.  Jacquemet,  S.  Tschantz, 
{\em  All hyperbolic Coxeter $n$-cubes}, 
J. Combin. Theory A 158 (2018), 387--406.  


\bibitem{JKRT} 
N. Johnson, R. Kellerhals, J. Ratcliffe, and S. Tschantz, 
{\em The size of a hyperbolic Coxeter simplex}, 
Transform. Groups 4 (1999), 329--353.

\bibitem{JKRT2} 
N. Johnson, R. Kellerhals, J. Ratcliffe, and S. Tschantz, 
{\em Commensurability classes of hyperbolic Coxeter simplex reflection groups},  
 Linear Algebra Appl. 345 (2002), 119--147.

\bibitem{Kaplinskaya} 
I. M. Kaplinskaja, 
{\em The discrete groups generated by reﬂections in the faces of simplicial prisms in Lobachevskii spaces},
Math. Notes 15 (1974),88--91.  


\bibitem{Khov}
A. G.  Khovanskii,   
{\em Hyperplane sections of polyhedra, toric varieties and discrete groups in Lobachevsky space, Functional Analysis and its applications},
V. 20, N 1, 50--61, 1986; translation in Funct. Anal. Appl. V. 20 no. 1 (1986),  41--50. 


\bibitem{Lanner} F. Lann\'er,  {\em On complexes with transitive groups of automorphisms}, Comm. Sem. Math. Univ. Lund 11 (1950), 1--71. 


\bibitem{Maclachlan}
C.  Maclachlan,   
{\em Commensurability classes of discrete arithmetic hyperbolic groups},
Groups Geom. Dynamics 5 (2011), 767--785. 



\bibitem{MR}
C.  Maclachlan,  A. W.  Reid, 
{\em The arithmetic of hyperbolic $3$-manifolds},
Graduate Texts in Mathematics Vol.  219 (2002) Springer,  Berlin.


\bibitem{Mcleod}
J.  Mcleod,  
{\em Hyperbolic Coxeter pyramids},  
Advances in Pure Mathematics 3 (2013), 78--82.


\bibitem{Nikulin} 
V. V. Nikulin,
{\em Factor groups of groups of automorphisms of hyperbolic forms with respect to subgroups generated by $2$-reflections}, 
Dokl. Akad. Nauk SSSR.  Vol. 248 (1979),  1307--1309.  


\bibitem{Nikulin2} 
V. V. Nikulin,
{\em Discrete reflection groups in Lobachevsky spaces and algebraic surfaces}, 
Proc.  Intern.  Congress of Math. Berkeley,  1986.   


\bibitem{Poin}
H.  Poincaré, 
{\em Théorie des groupes fuchsiens},  Acta Math. 1 (1882), 1--62.


\bibitem{Prk0}  M. N. Prokhorov,
 {\em The absence of discrete reflection groups with non-compact fundamental polyhedron of finite volume in Lobachevskij spaces of large dimension}, 
Math. USSR Izv. 28 (1987), 401--411.


\bibitem{Prk}  M. N. Prokhorov,
 {\em On polyhedra of finite volume in Lobachevskij spaces with dihedral angle $\pi/2$ and $\pi/3$},   Lecture in mathematics and its applications, Vol 2, No.2 (Russian), Inst. Mat. im. Steklova (1988), 151--187. 

 \bibitem{Tum1} P. Tumarkin,  {\em Hyperbolic Coxeter n-polytopes with $n+2$ facets}, 
Math. Notes 75 (2004), 848--854.

 \bibitem{Tum2} P. Tumarkin,  {Non-compact hyperbolic Coxeter $n$-polytopes with $n+3$ facets}, 
Russian Math. Surveys 58 (2003), 805--806.  

 \bibitem{Tum3} P. Tumarkin,  {Compact hyperbolic Coxeter $n$-polytopes with $n+3$ facets}, 
Electron. J. Combin. 14 (2007), 36pp. 

\bibitem{V0} 
{\` E}.  B.  Vinberg, 
{\em Discrete groups generated by reflections in Lobachevskii space}, 
Math. USSR Sbornik 114 (1967), 429--444.


\bibitem{V1} 
{\` E}.  B.  Vinberg, 
{\em Hyperbolic reflection groups}, 
Uspekhi Mat. Nauk 40 (1985), 29--66, 255.

\bibitem{V2} 
{\` E}.  B.  Vinberg, 
{\em The absence of crystallographic groups of reflections in Lobachevsky spaces of large dimensions}, 
Trans. Moscow Math. Soc. 47 (1985), 75--112. 

\bibitem{Mathematica}
Wolfram Research,  Inc.   Mathematica,  Champaign, IL, 2010. 


\bibitem{Zimmer}R. J.  Zimmer, {\em Ergodic Theory and Semisimple Groups}, Monographs in Mathematics Vol. 81,  Birkhäuser,  Basel (1984).  


\end{thebibliography}
\end{document}